\documentclass[a4paper,12pt,oneside,reqno]{amsart}

\usepackage{hyperref}
\usepackage[headinclude,DIV13]{typearea}
\areaset{15.1cm}{25.0cm}
\usepackage{amsfonts,amssymb,amsmath,amsthm,bbm}
\usepackage[latin1]{inputenc}
\usepackage{graphicx, psfrag}

\usepackage{bold-extra}

\usepackage{color}

\newtheorem{theorem}{Theorem}[section]
\newtheorem{lemma}[theorem]{Lemma}
\newtheorem{proposition}[theorem]{Proposition}
\newtheorem{corollary}[theorem]{Corollary}

\newtheorem{remark}[theorem]{Remark}

\theoremstyle{definition}
\newtheorem{definition}[theorem]{Definition}

\def\paragraph#1{\noindent \textbf{#1}}

\numberwithin{equation}{section}

\def\dist{\mathop{\rm dist}\nolimits}

\def\d{\mathrm{d}}

\def\<{\langle}
\def\>{\rangle}
\def\a{\alpha}
\def\b{\beta}
\def\e{\epsilon}
\def\ve{\varepsilon}

\def\g{\gamma}
\def\tg{\tilde\gamma}

\def\l{\lambda}

\def\s{\sigma}
\def\t{\tau}

\def\o{\omega}

\def\L{\Lambda}
\def\G{\Gamma}
\def\O{\Omega}

\def\Th{\Theta}

\def\del{\partial}
\def\R{{\Bbb R}}  
\def\N{{\Bbb N}}  
\def\P{{\Bbb P}}  
\def\E{{\Bbb E}}  
\def\T{{\Bbb T}}

\let\cal=\mathcal
\def\AA{{\cal A}}
\def\BB{{\cal B}}

\def\DD{{\cal D}}
\def\EE{{\cal E}}
\def\FF{{\cal F}}

\def\II{{\cal I}}

\def\MM{{\cal M}}

\def\OO{{\cal O}}
\def\PP{{\cal P}}

\def\RR{{\cal R}}
\def\SS{{\cal S}}

\def\VV{{\cal V}}
\def\UU{{\cal U}}
\def\VV{{\cal V}}

 \def \G {{\Gamma}}
 \def \L {{\Lambda}}
 
 \def \b {{\beta}}
\def \e {{\epsilon}}
 \def \s {{\sigma}}

 \def \t {{\tau}}
 
 \def \T {{\Theta}}
 \def \g {{\gamma}}
 \def \l {{\lambda}}
 \def \d {{\delta}}
 \def \a {{\alpha}}
 \def \o {{\omega}}
 \def \O {{\Omega}}

 \def \del {{\partial}}

 \def \ba {\begin{array}}
 \def \ea {\end{array}}

\def \frf {{\mathfrak f}}

 \def \cR {{\cal R}}

 \newcommand{\be}{\begin{equation}}
 \newcommand{\ee}{\end{equation}}

\newcommand{\bea}{\begin{eqnarray}}
 \newcommand{\eea}{\end{eqnarray}}
\def\TH(#1){\label{#1}}\def\thv(#1){\ref{#1}}
\def\Eq(#1){\label{#1}}\def\eqv(#1){(\ref{#1})}

\def\sfrac#1#2{{\textstyle{#1\over #2}}}

 \def \1{\mathbbm{1}}
\def\wt {\widetilde}
\def\wh{\widehat}

 \def\asl{\hbox{\rm Asl}}

\def \frf {{\mathfrak f}}

\def\cc{{\cal C}}
\def\cd{{\cal D}}
\def\md{{\mathfrak D}}
\def\ff{{\mathfrak f}}
\def\mg{{\mathfrak G}}
\def\mj{{\mathfrak J}}

\def\ww{{\mathfrak w}}
\def\ms{{\mathfrak S}}
\def\mt{{\mathfrak T}}
\def\frp{{\mathfrak P}}

\def \mR {{\mathbb R}}

\def\ci{{\cal I}}

\def\ck{{\cal K}}

\def\cm{{\cal M}}

\def\cp{{\cal P}}
\def\cq{{\cal Q}}
\def\cu{{\cal U}}
\def\cv{{\cal V}}
\def\cw{{\cal W}}
\def\fn{\mathfrak{N}}

\def\cR{{\cal R}}
\def\cs{{\cal S}}
\def\cg{{\cal G}}
\def\ce{{\cal E}}

\def\cz{{\cal Z}}
\def\cy{{\cal Y}}
\def\ca{{\cal A}}

\def\bcv{\bar\cu}
\def\hcv{\hat\cu}
\def\tcv{\tilde\cu}

\def\ttau{\tilde\tau}
\def\bb{\bar\b}

\def\Ups{\Upsilon}
\def\eps{\epsilon}
\def\tv{\tilde\varphi}

\def\bxn{\bar X^N}
\def\cx{{\cal X}}

\def\ax{\acute X}

\def\uxn{\breve X^N}

\def\ed{\stackrel{d}{=}}

\def\nn{\nonumber}

\begin{document}
 \title[Low temperature cascading 2-GREM close to equilibrium]
{Asymptotic behavior and aging of a low temperature cascading 2-GREM dynamics at extreme time scales}
\author[L.R. Fontes]{Luiz Renato Fontes}
\address{
L.R. Fontes\\
Universidade de S\~ao Paulo, IME, 
Rua do Mat\~ao 1010,
05508-090 S\~ao Paulo SP,
Brasil 
}
\email{lrenato@ime.usp.br}
\author[V. Gayrard]{V\'eronique Gayrard}
\address{
V. Gayrard\\ Aix Marseille Univ, CNRS, Centrale Marseille, I2M, Marseille, France
}
\email{veronique.gayrard@math.cnrs.fr}

\subjclass[2010]{60K35, 82C44}

\keywords{
GREM, Random Hopping Dynamics, low temperature, fine tuning temperature, scaling limit, extreme time-scale, K process, aging, spin glasses.
}
\date{\today}

\begin{abstract} 
We derive scaling limit results for the Random Hopping Dynamics for the cascading two-level GREM at low temperature at extreme time scales. It is known that in the cascading regime there are two static critical temperatures. We show that there exists a (narrow) set of {\em fine tuning} temperatures; when they lie below the static lowest critical temperature, three distinct {\em dynamical} phases emerge below the lowest critical temperature, with three different types of limiting dynamics depending on whether the temperature is (well) above or below, or at a fine tuning temperature, all of which are given in terms of K processes. We also derive scaling limit results for temperatures between the lowest and he highest critical ones, as well as aging results for all the limiting processes mentioned above, by taking a second small time limit.
\end{abstract}

\thanks{L.R.F.~was supported in part by CNPq 311257/2014-3 and 305760/2010-6; FAPESP 2017/10555-0 and 09/52379-8, and thanks the Institut de Math\'ematiques de Marseille of AMU for kind hospitality and support.
V.G. thanks the Instituto de Matem\' atica e Estat\'\i stica of USP for kind hospitality.
}

\maketitle

\section{Introduction}
    \TH(1)
    
Unlike classical magnetic systems for which both the existence and the nature of phase transitions can be
 characterized through the purely static Gibbsian formalism, phase transitions in spin-glasses (and glasses in general) are primarily dynamic: in the low temperature glassy phase,  the time to reach equilibrium is so long that the relaxation dynamics is the dominant aspect of  any experimental observation; even more unusual, relaxation is history dependent and  dominated by increasingly slow transients  \cite{BCKM98}. The appearance of this phenomenon, known as \emph{aging}, was proposed as an operating definition of the spin-glass transition \cite{Bou92, Ku01} and simple Markovian dynamics on finite graphs,  the so-called \emph{trap-models}, were designed to reproduce the power law decay of two-time correlations functions that characterizes aging experimentally \cite{Bou92},  \cite{BD95},  \cite{SN00} (see also \cite{BCKM98} and references therein).  This is  to be contrasted with equilibrium (i.e.~stationary) dynamics where correlations become time-translation invariant and decrease exponentially fast in time as the size of the system diverges. Very popular in theoretical physics, these phenomenological models have often replaced  the microscopic   spin-glass models from which they are inspired, primarily mean-field models such as the REM, GREM and $p$-spin SK models, with no other  justification than their apparent effectiveness.

To study aging in mean-field  spin-glass models, one first endows the microscopic spin space $\{-1,1\}^N$ with a Glauber  dynamics, namely, a stochastic dynamics that is reversible with respect to the Gibbs measure at inverse temperature $\b>0$ associated to the model's Hamiltonian. Under mild conditions, such dynamics converges to equilibrium. Knowing that,  to leading order,  the time-scale of equilibrium is exponential in $N$ \cite{FIKP, GJ17}, times-scales of interest are of the form $\exp{(\g\b N)}$, $\g>0$ (possibly depending on $N)$, the aim  being to choose $\b$ and $\g$ in such a way that the process is out of or close to equilibrium, and away from its high temperature phase. The rigorous study of aging in such dynamics  has first  been carried out in the REM which is now well understood in the domain $0<\g<\min(\b, \b_\star)$, $\b_\star\equiv \sqrt{2\log 2}$ \cite{BBG03b, BC06b, G10b, G16} (see also the references therein) as well as on sub-exponential time-scales  \cite{Gun09}. The dynamical transition from the glassy to the high temperature phase occurs along the line $\g=\b$  \cite{GH} with the boundary value $\g=\b_\star$ yielding the time-scale of equilibrium or \emph{extreme time-scale}. For $\g>\b_\star$ the dynamics is stationary and aging is interrupted. The much harder,  strongly correlated $p$-spin SK models could also be dealt with albeit only for a particular choice of the dynamics, the so-called \emph{Random Hopping Dynamics} and in domains of the parameters for which a {\em REM-like universality} takes place, namely, when the dynamics has not had time to discover the full correlation structure of the random environment and, not being influenced by strong long-distance correlations, behaves essentially like a REM
 \cite{BBC08,BG13,BAGun11,BGS13}. These limitations reflect our very poor grasp of the structure of extremes of the $p$-spin Hamiltonian and  of course puts the understanding of the dynamical spin-glass transition out of reach.

In the present paper we  initiate the study of the aging dynamics of the GREM, a model for which both the process of extremes and the low temperature Gibbs measure are fully understood \cite{CCP,BK,BB07}. More specifically, we consider the 2-GREM evolving under the {Random Hopping Dynamics} at {extreme time scales}, where it is close to equilibrium, and visits the configurations in the support of the Gibbs measure. The parameters of the model and the temperature are chosen such that this support has a fully cascading structure, reflecting the fact that correlations do matter. For  the most part of this paper, we will be concerned with the scaling limit of the dynamics for those time scales and parameters. Having obtained the limiting dynamics -- given by K processes appropriate to each regime that emerges in the analysis -- (see Subsection~\ref{thms}), we then proceed to take a {\em further small time limit} for which aging results follow which, as will be seen, not only go beyond the REM-like picture but are also richer than predicted  in the physics literature on the basis of trap models (see Subsection~\ref{aging}). In a follow-up paper we will consider all shorter time-scales where the dynamics is aging.

\subsection{The model} \label{model}
   \TH(S1.1)    
          
We now specify our setting.
Let	{$\VV_N=\{-1,1\}^N$, $\VV_{N_i}=\{-1,1\}^{N_i}$}, $\s=\s_1\s_2$, $\s_i\in \VV_{N_i}$, $i=1,2$, and make $N_1=\lfloor pN\rfloor$ for some $p\in(0,1)$ and $N_2=N-N_1$; we view $\s_i$ as the $i$-th {\em hierarchy} or {\em level} of $\s$.
Given $a\in(0,1)$, set 
\be
H_N(\s)=H_N^{(1)}(\s)+H_N^{(2)}(\s), \quad \s\in\VV_N,
\Eq(01.1)
\ee
where \be\label{eq:Xi}
H_N^{(1)}(\s)=H_N^{(1)}(\s_1)=-\sqrt{aN}\Xi^{(1)}_{\s_1},\quad H_N^{(2)}(\s)=-\sqrt{(1-a)N} \Xi^{(2)}_{\s_1\s_2}, \quad \s\in\VV_N,
\ee
and $\Xi:=\{\Xi^{(1)}_{\s_1},\Xi^{(2)}_\s;\,\s\in\VV_N\}$ is a family of i.i.d.~standard Gaussian random variables.
We  call  \emph{random environment} and denote by $(\O, \FF, \PP)$  the probability space on which the sequence of processes $(H_N(\s), \s\in\VV_N)$, $N>1$, is defined. As usual, we call $H_N(\s)$ the {\em Hamiltonian} or {\em energy} of $\s$. We refer to the minima of 
$H_N(\cdot)$ as {\em low energy} or {\em ground state} configurations. We will 
also refer to them, for being minima of $H_N(\cdot)$,
as {\em top} configurations. Likewise for $H_N^{(i)}(\cdot)$, $i=1,2$. 
The associated Gibbs measure at inverse the temperature $\b>0$  is the (random) measure $G_{\b,N}$ defined on $\VV_N$ through $G_{\b,N}(\s)=e^{-\b H_N(\s)}/Z_{\b,N}$ where $Z_{\b,N}$ is a normalization.

Let us briefly recall the key  features of the statics of the 2-GREM (see \cite{BS02, GK15} for nicely detailed accounts of the 2-GREM). As regard the Hamiltonian, two scenarios may be distinguished, related to the composition of the ground state energies in terms of their first and second level constituents: the {\em cascading phase}, when $a>p$, and where those energies are achieved by adding up the minimal energies of the two levels, so that to each first level ground state configuration, there corresponds many second level ground state configurations; in the complementary {\em non-cascading phase}, the composition of the ground state energies is different, and for each first level constituent there corresponds a single second level constituent.

Consider now the case where  $a>p$. 
\textcolor{black}{ 
The free energy exhibits two discontinuities at the critical temperatures 
\be
\b^{cr}_1\equiv\b_*\sqrt{\frac{p}{a}}<\b^{cr}_2\equiv\b_*\sqrt{\frac{1-p}{1-a}}, \quad \b_{*}=\sqrt{2\ln 2},
\Eq(new.temp)
\ee
and Gibbs measure behaves as follows.
In the  \emph{high-temperature} region $\b<\b^{cr}_1$, no single configuration carries a positive mass in the limit $N\uparrow\infty$, $\PP$-a.s.; here the measure resembles the high temperature Gibbs measure of the REM.
On the contrary, in the \emph{low temperature} region
$
\b>\b^{cr}_2
$,
Gibbs measure becomes fully concentrated on the set of ground state configurations,
yielding Ruelle's two-level probability cascade. In between, when 
$
\b^{cr}_1<\b<\b^{cr}_2
$,
an intermediate situation occurs in which the first level Hamiltonian variables ``freeze'' close to their ground state values, but not the second level ones, so that, once again, no single configuration carries a positive mass in the limit $N\uparrow\infty$. To obtain a macroscopic mass, one must lump together an exponentially large number of second level configurations. }
In this paper we focus on the cascading phase  ($a>p$) of the model at low temperature  ($\b>\b^{cr}_2$). We will also treat the  case $\b^{cr}_1<\b<\b^{cr}_2$  in a sub-domain of the parameters where we can prove a scaling limit for the dynamics at the extreme time-scale. In the complementary sub-domain the process is in an aging phase and will thus be treated in the follow-up paper.

The dynamics we consider is the popular Random Hopping dynamics (hereafter RHD). This 
is  a Markov jump process $(\s^N(t), t>0)$ that evolves along the edges of $\VV_N$ with transition rates given by,
for spin configurations $\s,\s'\in\VV_N$,  
\be
Nw_N(\s,\s')
=
e^{\b H_N(\s)}\1_{\s\,{\buildrel 1\over\sim}\,\s'}+e^{\b H^{(2)}_N(\s)}\1_{\s\,{\buildrel 2\over\sim}\,\s'}
\Eq(01.11)
\ee
and $w_N(\s,\s')=0$ else, where $\s\,{\buildrel i\over\sim}\,\s'$ iff $\s\sim\s'$ and $\s_i\sim\s'_i$.
Following a standard notation $\s\sim\s'$ indicates that 
$d(\s,\s')=1$, where $d(\cdot)$ stands for the usual Hamming distance in $\VV_N$ -- we will below denote by
$d_i$ such distance in $\VV_{N_i}$, $i=1,2$. In other words, $\s\sim\s'$ indicates that 
$\s,\s'$ differ in exactly one coordinate; we say in this context that $\s,\s'$ are (nearest) neighbors (in $\VV_N$).
We recognize the graph whose vertices are $\VV_N$ and whose edges are the neighboring pairs of configurations of $\VV_N$, 
abusively denoted also by $\VV_N$, as the $N$-dimensional hypercube.
Clearly, $\s^N$   is reversible w.r.t.~$G_{\b,N}$.

It now remains to specify the time-scale in which we observe this process. As mentioned earlier, we are
interested in {extreme time-scales}, where the dynamics is close to equilibrium. What we mean here by the dynamics being close to equilibrium at a given extreme time-scale 
 is that the dynamics with time rescaled by that time-scale converges in distribution
to a nontrivial 
Markov process which is ergodic in the sense of having an irreducible (countable) state space and a unique equilibrium distribution. 
The limiting dynamics is thus close to equilibrium, since it converges to equilibrium as time diverges, and it is in this sense that we say that the original dynamics is close to equilibrium at the extreme time-scale
-- see Remark~\ref{mean_ext-2} for a more precise discussion.

For future reference we call $\P$ the law of $\s^N$ conditional on the $\s$-algebra $\FF$, i.e. for fixed 
realizations of the random environment, or $\P_\eta$, when the initial configuration $\eta$ is specified.
We will denote by $\cp\otimes\P_\eta$ the probability measure obtained by integrating $\P_\eta$ with respect to $\cp$. Expectation with respect to  $\P$,  $\PP$ and $\cp\otimes\P_\mu$ are denoted  by  $\E$, $\EE$
and $\ce\otimes\E_\mu$, respectively, where $\mu$ is the uniform probability measure on $\VV_N$. 

\subsection{Dynamical phase transitions}\label{heuristics}

The distinct {\em static} phases of the cascading 2-GREM, determined by $\b=\b^{cr}_i$, $i=1,2$, are expected to exhibit 
different {\em dynamical} behaviors under the RHD at extreme (and conceivably other) time scales. 
This will be seen when comparing the results of our analysis of the RHD below the lowest critical temperature ($\b>\b_2^{cr}$) on the one hand, and those for intermediate temperatures
($\b\in(\b_1^{cr},\b_2^{cr})$), on the other hand. Another source of dynamical phase transition in the RHD at extreme time scales 
is the {\em fine tuning} phenomenon discussed next.

\subsubsection{Fine tuning; heuristics.} \label{ft:heuuristics}
There are two competing factors governing the behavior of the RHD at extreme time-scales.
One is the number of jumps it takes for the dynamics to leave a first level ground state configuration $\s_1$.
This is a geometric random variable with mean 
$\frac{N_2}{N_1}\exp\{\b\sqrt{aN}\Xi^{(1)}_{\s_1}\}\sim \exp\{\b\b_* \sqrt{pa}N\}$.
The other factor is the number of jumps the process makes until it finds a second level low energy configuration. 
This is $\sim 2^{(1-p)N}$. The relative size of these numbers determines three temperature regimes.

At relatively high temperatures, 
the second number dominates, and so after leaving a ground state configuration
$\s=\s_1\s_2$, which it does at times of 
order  $\exp\{\b \sqrt{(1-a)N}\Xi^{(2)}_{\s}\}\\\sim \exp\{\b\b_* \sqrt{(1-p)(1-a)}N\}$,
 $\s^N$ will visit many first level ground state configurations before it finds a second level ground state configuration. 
When it first finds such a second level 
configuration, say $\s'_2$, while in a first level ground state configuration $\s'_1$
(meaning that it first returned to an overall ground state configuration $\s'_1\s'_2$), $\s'_1$ will be effectively distributed proportionally to $\exp\{{\b\sqrt{aN}\Xi^{(1)}_{\s'_1}}\}$;
this can be explained by a size-bias mechanism that operates in the selection of $\s'_1$. 
There is no such mechanism for the choice of $\s'_2$, and it is distributed uniformly.

On the other hand, at low enough temperatures, the first factor dominates, and while staying at a first level 
low energy configuration, the process has time to reach equilibrium at the second level, so at the time scale 
where we see (uniform) transitions between first level low energy configurations, the second level is in equilibrium.
This is a longer time scale, composed of the many jump times at second level till exiting first level.

In a narrow strip of borderline temperatures, we see nontrivial dynamics at both levels going on at the same time scale (corresponding to jump times out of second level ground state configurations, of magnitude 
$\exp\{\b\b_* \sqrt{(1-p)(1-a)}N\}$, 
as at high temperatures).

In order for the above picture to represent the dynamics, 
{we need the temperature to be below the static lowest phase transition temperature  $1/\b^{cr}_2$}, 
so that  the time spent off the ground state configurations is negligible. Moreover, this three-phase dynamical picture will take place  if (and only if) the borderline temperatures alluded to above -- and to be called {\em fine tuning} temperatures below -- are (well) below the static lowest phase transition temperature; otherwise, we will see only one dynamical phase below that lowest critical temperature, namely the low temperature phase alluded to above.

\subsubsection{Intermediate temperatures.}\label{intermediate}
For values of $\beta$ between $\b=\b^{cr}_1$ and $\b=\b^{cr}_2$, we investigate the behavior of the dynamics at a time scale when we see transitions between the first level ground state configurations at times of order 1. In order that this time scale corresponds to an extreme time scale (as stipulated above -- see the one but last paragraph of Subsection~\ref{model}), we need a further 
a restriction in the temperature, to be seen below. 
The behavior of the dynamics of the first level configuration for intermediate temperatures in those conditions 
is similar to the one below the 
minimum between the lowest critical and  the fine tuning temperature.

\subsection{Aging results.} \label{aging}
Let us briefly anticipate our main aging results, holding in the case described at the last paragraph of Subsection~\ref{ft:heuuristics}, where fine tuning temperatures are below the lowest static critical temperature. In this case, as mentioned above, we have three phases for the dynamics at extreme time scales below the lowest static critical temperature. As already explained briefly above, our aging results in this paper are obtained by first taking the scaling limit of the dynamics at extreme time scales, thus obtaining ergodic processes, and next taking a small time limit in those processes, thus obtaining aging results. Let us suppose we have already taken the first, extreme time scale limit. We obtain three distinct dynamics in each of the temperature ranges: above fine tuning, at fine tuning, and below fine tuning (see Theorems~\ref{thm:cm},~\ref{thm:cft}, and~\ref{thm:cl} in Subsection~\ref{thms}). Let us consider the events 
\begin{equation}\label{ni}
	\fn_i=\fn_i(t_w,t)=\{Y_i\mbox{ does not jump between times }t_w\mbox{ and }t_w+t\},\,\,i=1,2,
\end{equation}
where $Y_i$ represents the $i$-th level marginal of the process, and $t_w,t>0$, and let us define
\begin{equation}\label{overlap}
	\Pi(t_w,t_w+t)=P(\fn_1\cap\fn_2)+pP(\fn_1\cap\fn_2^c)+(1-p)P(\fn_1^c\cap\fn_2).
\end{equation}
$\Pi$ is an analogue of a (limiting) two-time overlap function of the RHD. In the regime considered in this subsection,
we have the following (vanishing time limit) aging result.
\begin{equation}\label{age_res}
\hspace{-9pt}	\lim_{t_w,t\to0\atop{t/t_w\to\theta}}\Pi(t_w,t_w+t)=
	\begin{cases}
	 \asl_{\a_2}\left(\frac{1}{1+\theta}\right),&\mbox{above fine tuning,}\cr
	p \asl_{\a_1\a_2}\left(\frac{1}{1+\theta}\right)+(1-p)\asl_{\a_2}\left(\frac{1}{1+\theta}\right),&\mbox{at fine tuning,}\cr
	p\asl_{\a_1}\left(\frac{1}{1+\theta}\right),&\mbox{below fine tuning,}	
	\end{cases}	
\end{equation}
with proper choices of $\a_1,\a_2\in(0,1)$, where $\asl_\cdot$ is the arcsine distribution function. 
See Section~\ref{age} and also the last paragraph of Section~\ref{inter} for details and other regimes.

\subsection{A 2-GREM-like trap model.} The idea behind the construction of trap models for low temperature glassy dynamics is as follows:
the traps represent the ground state configurations and, assuming that at low temperature the dynamics
spends virtually all of the time on those configurations -- here an extreme time-scale is assumed --, higher energy configurations are simply  not represented, 
 and all one needs to do is specify  the times spent by the dynamics at each visit to a ground state configurations, and the transitions among those configurations, in such a way that the resulting process be Markovian.

The simplest such model
 to be proposed in the study of aging was put forth in~\cite{BD95}, with $\{1,\ldots,M\}$ as configuration space, mean waiting time at $i$ given by $X_i$, with $X_1,X_2,\ldots$ iid random variables in the domain of attraction of of an $\a$-stable law, $\a\in(0,1)$, and uniform transitions among the configurations. This is the so called {\em REM-like trap model} or {\em trap model on the complete graph}. Models of a similar nature for the GREM were proposed in~\cite{BD95} and also in~\cite{SN00}. The scaling limit of the latter model for a {\em fine tuning} choice of level volumes was computed in~\cite{FGG} and its aging behavior away from fine tuning was studied in \cite{GG16}.

Out of our analysis of the RHD in the cascading phase at low temperatures 
comes up the following GREM-like trap model on the ground state configurations of the GREM. The configuration space is represented by $M_1$ first level ground state configurations, labeled in decreasing order, and for each of those configurations, we have $M_2$ second level ground state configurations, labeled in decreasing order. 
The transition probabilities $p(x,y)$ between $x=(x_1,x_2)$ and $y=(y_1,y_2)$,
 $1\leq x_i,y_i\leq M_i$, $i=1,2$, are given by
\be\label{trans}
\vspace{-5pt}
p(x,y)=
         \begin{cases}
	\left[(1-\l^{y_1})+\ \l^{y_1}\,\nu_1(y_1)\right]\frac{1}{M_2},&\mbox{if }x_1=y_1,\cr
	\nu_1(y_1)\l^{y_1}\frac{1}{M_2},&\mbox{otherwise},
	\end{cases}	
\ee
where 
\be
\l^{y_1}=\frac1{1+M_2\psi\g_1(y_1)} \quad\text{and}\quad
\nu_1(y_1)=\frac{
\g_1(y_1)\l^{y_1}
}
{
\sum_{z=1}^{M_1}\g_1(z)\l^{z}
}.
\Eq(trans.2)
\ee
The factor $\psi\in[0,\infty]$ in \eqv(trans.2) interpolates between higher temperatures, above fine tuning ($\psi=0$) 
and low temperatures, below fine tuning ($\psi=\infty$);
$\psi\in(0,\infty)$ corresponds to borderline, fine tuning temperatures in the picture outlined above, and to be described 
more precisely below. The factors $\g_1(\cdot)$ correspond to the scaled $M_1$ maxima first-level {\em Boltzmann factors} 
$\exp\{\b\sqrt{aN}\Xi^{(1)}_{\cdot}\}$.

The time spent at each visit to $x$ in the appropriate time scale is an exponential random variable with mean
$\g_2(x)$, where for each $x_1$, $\g_2(x_1,\cdot)$  corresponds to the scaled $M_2$ maxima second-level 
{\em Boltzmann factors}  $\exp\{\b\sqrt{(1-a)N}\Xi^{(2)}_{\s_1\cdot}\}$, with $\s_1$ the first level configuration labeled $x_1$, as explained above. It must be said that this time scale is of the order of magnitude of the time needed for the dynamics to jump out of ground state configurations, and it is indeed {\em extreme} in the above sense only for fine tuning temperatures and above. In these cases,~(\ref{trans}) indeed represents the transitions among ground states (in the extreme time scale). At lower temperatures, as explained in the discussion on dynamical phase transition above, the extreme time scale is longer, with uniform transitions on first level, with exponential waiting times, and on second level the dynamics is a trivial product of equilibria at different  times.

The results indicated above do not seem to be in or be predicted by the physics literature, which has focused on short time scales,
where all levels age simultaneously, and thus no effect of the longer time dynamical phase transition is present. 
This matches our {\em short} extreme times aging results only at fine tuning, where that simultaneity takes place.
Also, our GREM-like trap model differs from those considered in the literature (in~\cite{BD95,SN00}).

\subsection{Organization.} In Section~\ref{mres} we make precise the notions introduced in this introduction, and formulate our 
scaling limit results for $\s^N$ on extreme time scales for $\beta>\beta_2^{cr}$. In Sections~\thv(1.5)-\thv(S4.4) we formulate and argue entrance law 
results leading in particular to the transition probabilities between ground state configurations described in~\ref{trans}. These results are key ingredients to the proofs of the above mentioned scaling limit results, which are undertaken in 
Sections~\ref{proofcl}-\ref{proofcft}. 
Section~\ref{age} is devoted to a brief discussion about aging results that we obtain for the limit processes, as already mentioned.
In Section~\ref{inter} we briefly discuss results for the intermediate temperature phase $(\beta_1^{cr},\beta_2^{cr})$.
An appendix contains definitions of the limit processes entering our 
scaling limit results, as well as auxiliary results.

\section{Scaling limit of $\s^N$. Main Results}\label{mres}

\subsection{Choice of parameters} 

As mentioned above, we will study the cascading phase, which, we recall, corresponds to
\begin{equation}\label{casca}
a>p.
\end{equation}
As regards temperatures, we want to take volume dependent ones (this is needed in order to capture the fine tuning phase transition).
We also want low temperatures, so in the cascading phase this corresponds to
\begin{equation}\label{low}
\liminf_{N\to\infty}\b>\b_*\sqrt{\frac{1-p}{1-a}},
\end{equation}
where the dependence of $\b$ on $N$ is implicit. In order to describe that dependence, 
let us start by setting $\b_{*}=\sqrt{2\ln 2}$, $\kappa=\sfrac{1}{2}{(\ln\ln 2 + \ln 4\pi)}$,
\begin{equation}\label{a1c1}
\a_1^N=\frac{\b_{*}}{\b}\sqrt{\frac{N_1}{Na}},\quad 
c_1^N=\exp\left\{-\frac{1}{\a_1^N}\left(\b^2_{*}N_1-\frac{1}{2}\ln N_1+\kappa\right)\right\}.
\end{equation}
Given a sequence $\zeta_N< N_2\b^2_{*}/2$ of real numbers, let $\b(a, p, N,\zeta_N)$ be the solution in $\b$ of the equation
\be
c_1^N2^{N_2}=e^{\zeta_N + \kappa/\a_1^N}. 
\Eq(1.temp.1)
\ee
In explicit form
\be\label{betan}
\b(a, p, N,\zeta_N)=\frac{\b_{*}}{2}\frac{N_2}{N_1}\sqrt{\frac{N_1}{Na}}
\frac{
	1-\frac{2\zeta_N}{N_2\b^2_{*}}
}{
	1-\frac{\ln N_1}{2\b^2_{*}N_1}
}
=
\b^{FT}
\left(1-\sfrac{2\zeta_N}{N_2\b^2_{*}}\right)(1+o(1))
\ee
where 
$
\b^{FT}\equiv\frac{\b_{*}}{2}\frac{1-p}{\sqrt{pa}}
$
is the inverse of the \emph{fine tuning} temperature.
Depending on the behavior of $\zeta_N$ we distinguish three types of  temperature regimes. (Given two sequence $s_N$ and
$\bar s_N$ we write  $s_N\sim \bar s_N$ iff $\lim_{N\rightarrow\infty} s_N/ \bar s_N=1$. We also write $s_N=\OO(1)$, resp. $s_N=o(1)$, iff $|s_N|\leq C'<\infty$, for  some $C'$ and all $N>0$, resp. $\lim_{N\rightarrow\infty}s_N=0$.)

\begin{definition}[At/above/below fine tuning]
	\TH(1.def1)
	We say that a  sequence $\b^{-1}\equiv\b^{-1}_N>0$  of temperatures is in the fine tuning (FT) regime if there exists 
	a finite real constant $\zeta$ and  a convergent sequence $\zeta_N\sim\zeta$ such that $\b=\b(a, p, N,\zeta_N)$.
	We say that $\b^{-1}$ is below fine tuning if there exists a sequence $\zeta^-_N$ satisfying $\zeta^-_N\rightarrow -\infty$ as $N\rightarrow\infty$, and such that $\b=\b(a, p, N,\zeta^-_N)$. 
	Finally, we say that $\b^{-1}$ is above fine tuning if there exists a sequence $\zeta^+_N\leq (1-\d)N_2\b^2_{*}/2$, $\d>0$, satisfying $\zeta^+_N\rightarrow +\infty$ as $N\rightarrow\infty$, and such that $\b=\b(a, p, N,\zeta^+_N)$. 
	(Note that for $\b=\b(a, p, N,\zeta_N)$ to be a convergent sequence, $\zeta_N/N_2$ must  be convergent.)
\end{definition}

In order to precisely describe our results, we start with some technical preliminaries.
As described above, the way the ground state configurations are arranged in the cascading phase
naturally suggests the following relabeling of the state space
$\VV_N$.

\subsection{Change of representation} \label{cor}

Let $\DD_i=\{1\dots2^{N_i}\}$, $i=1,2$. Call $\xi_1^{x_1}$, $x_1 \in\DD_1$, the vertices  (of $\VV_{N_1}$) 
that carry the ranked variables 
\be
\Xi^{(1)}_{\xi_1^{1}}\geq  \Xi^{(1)}_{\xi_1^{2}}\geq \dots \Xi^{(1)}_{\xi_1^{x_1}}\geq\dots
\Eq(01.17)
\ee
and, similarly, for each $x_1 \in\DD_1$ call $\xi_2^{x_1x_2}$, $x_2 \in\DD_2$,  the vertices (of $\VV_{N_2}$) such that
\be
\Xi^{(2)}_{\xi_1^{x_1}\xi_2^{x_1 1}}\geq  \Xi^{(2)}_{\xi_1^{x_1}\xi_2^{x_1 2}}\geq \dots \Xi^{(2)}_{\xi_1^{x_1}\xi_2^{x_1x_2}}\geq\dots
\Eq(01.18)
\ee

Let $\xi:\DD\to\VV_N$ be such that $\xi(x)=\xi^x:=\xi_1^{x_1}\xi_2^{x_1x_2}$. 
This is a one to one mapping for almost every realization of $\Xi$.
Let now $X^N=X^N_1X^N_2$ be the mapping of $\s^N$ on $\DD$ by the inverse of $\xi$. This is the process we will state scaling limit results for. This alternative representation suits our purpose of taking scaling limits, mainly due to the convenience of working with a state space which naturally extends to set of the natural numbers, which will be the state space of the limiting processes. The class to which these processes belong, namely K processes, is described in the appendix. In the build up for those scaling limit results, let us introduce next scaling factors, and then the scaling limit of the environment. 

\subsection{Scalings}  Set 
\be
\a_2^N=\frac{\b_{*}}{\b}\sqrt{\frac{N_2}{N(1-a)}},\quad
c_2^N=\exp\left\{-\frac{1}{\a_2^N}\left(\b^2_{*}N_2-\frac{1}{2}\ln N_2+\kappa\right)\right\},\,\,\,
\Eq(01.2)
\ee
and define the scaled variables 
\bea
\gamma^N_1(\s_1)
&\equiv&
c_1^Ne^{\b\sqrt{a_1N} \Xi^{(1)}_{\s_1}}
=e^{u^{-1}_{N_1}(\Xi^{(1)}_{\s_1})/\a_1^N}
\Eq(01.3)
\\
\gamma^N_2(\s_1\s_2)
&\equiv&
c_2^N e^{\b\sqrt{a_2N} \Xi^{(2)}_{\s_1\s_2}}
=e^{u^{-1}_{N_2}(\Xi^{(2)}_{\s_1\s_2})/\a_2^N}
\Eq(01.4)
\eea
where for $i=1,2$, $u_{N_i}$ is the scaling function for the maximum of $2^{N_i}$ i.i.d.~standard Gaussians,
\be 
u_{N_i}(x)=\b_{*}\sqrt{N_i}+\frac{1}{\b_{*}\sqrt{N_i}}\left\{x-(\ln({N_i}\ln 2) + \ln 4\pi)/2 \right\},\quad x\in \R.
\Eq(01.5)
\ee

For later use set 
\be
\psi^{-1}_N=\sfrac{N_1}{N_2}e^{\zeta_N + \kappa/\a_1^N}.
\Eq(01.9)
\ee
Clearly, in the fine tuning regime, 
\be
\textstyle
\b\rightarrow \b^{FT},
\a_1^N\rightarrow \a_1\equiv 2\frac{p}{1-p},
\a_2^N\rightarrow \a_2\equiv2\sqrt{\frac{a}{1-a}\frac{p}{1-p}}
\,\,\,\text{and}\,\,\,
\psi^{-1}_N\rightarrow
\psi^{-1}
\equiv\sfrac{p}{1-p}e^{\zeta + \kappa\frac{1-p}{2p}}.
\Eq(01.10)
\ee

\subsection{Scaling limit of the environment}
For the remainder of this section, we assume that $\lim_{N\to\infty}\zeta_N^+/N_2$ exists. It follows that
so does 
\be\label{limb}
\lim_{N\to\infty}\b=:\bb.
\ee

Let
 ${\g^N_1}=\{\g^N_1(x_1),\,x_1\in\cd_1\}$,\, and for $x_1\in\cd_1$, set
${\g^{N,x_1}_2}:=\{\g^N_2(x_1x_2),\,x_2\in\cd_2\}$, where $\g^N_1(x_1)$ and $\g^N_2(x_1x_2)$ stand for 
$\g^N_1(\xi_1^{x_1})$ and $\g^N_2(\xi_2^{x_1x_2})$, respectively.
Then 
\begin{equation}
\label{g1ntog1}
{\g_1^N}\to{\g_1},\,\,\g^{N,x_1}_2\to{\g_2^1},
\end{equation}
$x_1\in\N$,
in distribution as $N\to\infty$, as point processes on $\R^+$,  and random measures on $\N$, 
where ${\g_1}:=\{\g_1(x_1),\,x_1\in\N\}$, ${\g_2^1}:=\{\g_2(1\,x_2),\,x_2\in\N\}$
are independent Poisson point processes in $\R^+$, enumerated in decreasing order, 
with respective intensity functions given by $\a_i/x^{1+\a_i}$, $i=1,2$, with
\begin{equation}
\label{alfas}
\a_1=\frac{\b^{cr}_1}{\bb},
\,\,
\a_2=\frac{\b^{cr}_2}{\bb}.
\end{equation}
Notice that, as follows from our assumptions for this phase, $0<\a_1<\a_2<1$.
We also have that
\begin{equation}
\label{gtog}
{\g^N}:=\{\g_1^N(x_1)\g_2^N(x_1x_2),\,x_1x_2\in\cd\}\to{\g}:=\{\g_1(x_1)\g_2(x_1x_2),\,x_1x_2\in\N^2\}
\end{equation}
in distribution as $N\to\infty$ as point processes on $\R^+$, 
and (a.s.~finite) random measures on $\N^2$, where 
$\g_{2}^{x_{1}}:=\{\g_2(x_{1}x_2),\,x_2\in\N\}$, $x_{1}\geq2$, 
are independent copies of ${\g_2^1}$.
We will sometimes below let ${\g_2}$ stand for the family $\{\g_2(x_{1}x_2),\,x_1x_2\in\N^2\}$.

\begin{remark}\label{bk}
	All of the convergence claims made in the above paragraph follow readily from convergence results of~\cite{BK}. 
	Indeed, we may apply Theorems 1.3 and 1.7 therein as follows. We preliminarily point out that in the cascading 
	two-level system we are dealing with in the present paper, we have the following in terms of the notation of~\cite{BK}: 
	$n=m=2$; $J_1=1$, $J_2=2$; $\bar X_{\s_1}=\Xi^{(1)}_{\s_1}$, $\bar X^{\s_1}_{\s_2}=\Xi^{(2)}_{\s_1\s_2}$;
	$a_1=\bar a_1=a$, $a_2=\bar a_2=1-a$; $\a_1=\bar\a_1=2^p$, $\a_2=\bar\a_2=2^{1-p}$\footnote{In this sentence, $\a_1$ and $\a_2$ are notations from~\cite{BK}, and should not be confused with the notation of the present paper introduced in~(\ref{alfas}).}.
	We then have that for $\s_1\s_2\in\cv_N$
	\begin{equation}
	\gamma^N_1(\s_1)=\exp\{\b u^{-1}_{p\log2,N}(\sqrt{a}\Xi^{(1)}_{\s_1})\},\,
	\gamma^N_2(\s_1\s_2)=\exp\{\b u^{-1}_{(1-p)\log2,N}(\sqrt{1-a}\Xi^{(2)}_{\s_1\s_2})\},
	\end{equation}
	where $u_{\cdot,N}(\cdot)$ is defined in (1.7) of~\cite{BK}. Theorem 1.3 of~\cite{BK} now asserts the convergence 
	of $\{u^{-1}_{p\log2,N}(\sqrt{a}\Xi^{(1)}_{\s_1}),\,u^{-1}_{(1-p)\log2,N}(\sqrt{1-a}\Xi^{(2)}_{\s_1\s_2});\,\s_1\s_2\in\cv_N\}$
	to the Poisson cascade introduced in~\cite{BK}. (\ref{g1ntog1}) and (\ref{gtog}) follow from that and Theorem 1.7 in the same reference after straightforward considerations -- see also Proposition 1.8 of~\cite{BK}.
\end{remark}

\begin{remark}\label{rmk:gibbs2}
	It follows from the above results that the Gibbs measure $G_{\b,N}$ converges suitably to $G_{\bb}$ --- the normalized ${\g}$ --- as $N\to\infty$.
\end{remark}

\subsection{Scaling limit of $X^N$} \label{thms}

In order to have the
three cases outlined in the heuristics discussion, namely, above, at and below fine tuning temperatures,
we need that 
$
\b^{FT}>\b^{cr}_2
$,
namely, that
\begin{equation}
\label{cond1}
\sqrt{\frac{1-p}{1-a}}<\frac{1-p}{2\sqrt{pa}}; 
\end{equation}
otherwise, all low temperatures according to~(\ref{low}) are below fine tuning.
In each case we find a different scaling and different scaling limit for $X^N$.

To state the first theorem, we take $\zeta_N^+$ as in Definition \thv (1.def1), 
with the extra assumption~(\ref{limb}), i.e., we let the sequence of real numbers $\zeta_N^+$ satisfy 
\be\label{aft}
\lim_{N\to\infty}\zeta_N^+=\infty\quad\mbox{ and }\quad
\lim_{N\to\infty}\zeta_N^+/N_2<\frac{\b^2_{*}}2\left(1-2\sqrt{\frac p{1-p}}\sqrt{\frac a{1-a}}\right).
\ee
The latter condition is equivalent to the second condition in~(\ref{low}) once we replace the 'lim inf'
by the 'lim' there.
Set
\begin{equation}
\label{scay}
\tilde X^N(t)=X^N(t/c_2^N),\,t\geq0.
\end{equation}
We recall that this is a process in the random environment $\g^N$. The limiting processes, which are K processes,
described in the appendix, will naturally also be processes in random environment.

Let $\ff_{1},\ww_{1}:\N^2\to(0,\infty)$ be such that
$\ff_{1}(x_{1}x_{2})=\tg_{2}(x_{1}x_{2}):=\frac1{1-p}\g_{2}(x_{1}x_{2})$, $\ww_{1}(x_{1}x_{2})=\g_{1}(x_{1})$ for
all $x_{1}x_{2}\in\N^2$. These functions will play the role of random environment for the limiting process in this case.

\begin{theorem}[Above fine tuning temperatures]
	\label{thm:cm}
	As $N\to\infty$ 
	\be
	\tilde X^N \Rightarrow \ck(\ff_1,\ww_1); 
	\ee
	where $\Rightarrow$ stands for convergence in $\cp\otimes\P_\mu$-distribution. 
	The convergence takes place on the Skorohod space 
	of trajectories of both processes, with the $J_1$ metric.
\end{theorem}
See the definition of $\ck(\cdot,\cdot)$ in the first subsection of the appendix.

To state our second theorem, we assume $\lim_{N\to\infty}\zeta_N^+=\zeta$ for some real finite $\zeta$.

Let $\ff_{2}:\N\to(0,\infty)$,
$\ff'_{2}:\N^2\to(0,\infty)$ such that
$\ff_{2}(x_{1})=\tg_{1}(x_{1}):=\psi\g_{1}(x_{1})$, 
$\ff'_{2}(x_{1},x_{2})=\tg_{2}(x_{1},x_{2})$  for
all $x_{1}x_{2}\in\N^2$. 

\begin{theorem}[At fine tuning temperatures]
	\label{thm:cft}
	As $N\to\infty$
	\be 
	\tilde X^N\Rightarrow\ck_2(\ff_{2},\ff'_{2}).
	\ee
	The convergence takes place in the Skorohod space of trajectories of both processes, 
	with the $J_1$ metric.
\end{theorem}

See the definition of $\ck_2(\cdot,\cdot)$ in the first subsection of the appendix.

\begin{remark}
	\label{finiteness}
	In order for the above mentioned two-level K-process to be well defined, we need to
	make sure that $\ff_{2},\ff^\prime_{2}$ satisfy (almost surely) the summability 
	conditions~(\ref{eq:sumff},~\ref{eq:sumffp}). 
	This is a classic result for~(\ref{eq:sumff}) --- recall that $\ww\equiv1$ in this case ---, and 
	follows by standard arguments for~(\ref{eq:sumffp}) from the fact that $\a_1<\a_2<1$ 
	(as noted above, below~(\ref{alfas})).
\end{remark} 

For our last theorem, we take $\zeta_N^-$ as in Definition \thv (1.def1).

Let $\bar c^N = c_1^N 2^{N_2} c_2^N$ and make
\begin{equation}
\label{scayh}
\bar X^N(t)=X^N(t/\bar c^N),\,t\geq0.
\end{equation}
Let also $\ff_{3}:\N\to(0,\infty)$, with
$\ff_{3}(x_{1})={\g_{1}(x_{1})}{\sum_{x_2\in\N}}\tg_2(x_1x_2)$, $x_1\in\N$.

\begin{theorem}[Below fine tuning temperatures]
	\label{thm:cl}
	As $N\to\infty$ 
	\be
	\bar X^N\Rightarrow \bar X_{1}\bar X_{2},
	\ee
	where $\bar X_{1}\sim\ck(\ff_{3},1)$
    and, given $\g_{2}$
	and $\bar X_{1}=x_1\in\N$, 
	$\bar X_{2}$ is an iid family of random variables on $\N$ (indexed by time) each of which
	is distributed according to the weights given by $\g_{2}^{x_1}$. 
	The marginal convergence of the first coordinate takes place in the Skorohod space 
	of trajectories of both processes, with the $J_1$ metric, and the convergence of the second coordinate
	is in the sense of finite dimensional distributions only.
\end{theorem}

\begin{remark}
	\label{nocond1}
	If Condition~\ref{cond1} is not satisfied (within the cascading, low temperature regime 
	treated in this paper), then we are below fine tuning temperatures 
	and Theorem~\ref{thm:cl} holds for all $\b>\b^{cr}_2$ temperatures, as can be readily checked from its proof. The other regimes are not present in this case.
\end{remark}

\begin{remark}
	\label{keq}
	It is either known or follows readily from known results that the limiting processes in the above theorems 
	are ergodic Markov processes\footnote{I.e., Markov processes that have an irreducible state space and a unique invariant distribution.}, having the infinite volume Gibbs measure $G_{\bb}$ (see Remark~\ref{rmk:gibbs2} above) as their unique equilibrium distribution. See~\cite{Gav} 
	for the case of the 2-level K-process, and~\cite{FP}  
	for the cases involving weighted/uniform K-processes.
\end{remark}

\begin{remark}
	\label{mean_ext-2}
	As discussed earlier (see 
	the one  before last paragraph of Section \thv(S1.1))
	we may then say, after Remark~\ref{keq}, that the time scale 
	$1/c_2^N$ is an extreme time scale at and above fine tuning, and the time scale $1/\bar c^N$ is an extreme time scale below fine tuning.
\end{remark}

\section{Entrance law. Main result.}
\TH(1.5)

In this and the next three sections, we will mostly not be concerned with limits, so we find it convenient to revert to the original representation of $\s^N$ as spin configuration. 
Given a subset $A\subset\VV_N$, the hitting times $\t_A$ of $A$  by the  continuous and discrete time 
processes $\s^N$  and  $J^{*}_N$ are defined, respectively, as
\be
\t_A=\inf\{t>0 \mid \s^N(t)\in A \}\quad\text{and}\quad\t_A=\inf\{i\in\N \mid J^{*}_N(i)\in A \}.
\Eq(01.16)
\ee

\subsection{The Top}  
\label{top}
Given two integers $M_i<2^{N_i}$, $i=1,2$, set $\MM=\MM_1\times\MM_2$  where $\MM_i=\{1,\dots,M_i\}$, $i=1,2$. 
We then let the Top be the set 
\be
T
\equiv T((M_i)_{i\leq 2}, (N_i)_{i\leq 2})
\equiv\{\s\in\VV_N \mid \s=\xi_1^{x_1}\xi_2^{x_1x_2}, x_1x_2\in\MM\}
\Eq(01.19)
\ee
Note that we may also write $T=\cup_{x_1\in\MM_1}T^{x_1}$ where, for each $x_1\in\MM_1$
\be
T^{x_1}
\equiv T^{x_1}(M_1,N_1)
\equiv\{\s\in\VV_N \mid \s=\xi_1^{x_1}\xi_2^{x_1x_2}, x_2\in\MM_2\}.
\Eq(01.20)
\ee
Further let 
\be
T_1\equiv T_1(M_1,N_1)\equiv\{\s_1\in\VV_{N_1}\mid \s=\xi_1^{x_1}, x_1\in\MM_1\}
\Eq(1.20)
\ee
be the canonical projection of $T$ on  $\VV_{N_1}$.
To each $\xi_1^{x_1}$ in $T_1$ we associate the cylinder set
\be
W^{x_1}\equiv W^{x_1}((N_i)_{i\leq 2})\equiv\{\s\in\VV_N\mid \s_1= \xi_1^{x_1}\}.
\Eq(1.21)
\ee
Clearly, $T^{x_1}$ is the restriction of $T$ to this cylinder,
$
T^{x_1}
=W^{x_1}\cap T
$.
Finally, we set
\be
\overline W \equiv \cup_{x_1\in\MM_1}W^{x_1}.
\Eq(1.24)
\ee

\subsection{Main entrance law results}
   \TH(S1.2)
   
From now on  we fix $(\zeta_N)$, a sequence of real numbers such that $\b=\b(a, p, N,\zeta_N)>0$ for all $N$,  and let $\psi_N$ be as in \eqv(01.9). 
For each $x_1\in\MM_1$ and $A\subseteq T^{x_1}$
set
\be
\l^{x_1}_N(A)\equiv\l^{x_1}_N(|A|, N,\psi_N)
=
\frac{1}{1+|A|\psi_N\gamma^N_1(\xi_1^{x_1})}.
\Eq(P.top.004)
\ee
We will see that this quantity can be interpreted as  the  probability that, starting in $W^{x_1}$, 
the process exits $W^{x_1}$ before finding an element of $A$.
 Note that $\l^{x_1}_N(A)$ is a random variable. We use it to define the random probability measure $\nu_1$  
 on $\MM_1$ that assigns to $x_1$ the mass
\be
\nu_1^N(x_1)
=\frac{1-\l^{x_1}_N(T^{x_1})}
{\sum_{x'_1\in\MM_1}(1-\l^{x'_1}(T^{x'_1}))}
\Eq(P.top.4)
\ee
Similarly, given $\bar\eta\in T$,  we denote by $\overline\nu_1$ the random measure on $\MM_1$ define through
\be
\overline\nu_1^N(x_1)
=\frac{1-\l^{x_1}_N(T^{x_1}\setminus\bar\eta)}
{\sum_{x'_1\in\MM_1}(1-\l^{x'_1}(T^{x'_1}\setminus\bar\eta))}
\Eq(P.top.4')
\ee
(where clearly $T^{x_1}\setminus\bar\eta=T^{x_1}$ if $\bar\eta\notin T^{x_1}$).

\begin{proposition}
    \TH(P.top)
There exists a subset $\wt\O\subset\O$ with $\PP(\wt\O)=1$ such that on $\wt\O$, for all $N$ large enough, 
in the temperature domain determined by $\b=\b(a, p, N,\zeta_N)>0$ and $\zeta_N\ll \log N$,
the following holds. Let
$
\varepsilon_{N}=\OO\left(N^{-1}\right)
$.

\item{i)}  Entrance law. Let $x_1\in\MM_1$ and $\eta \in T^{x_1}$. Then

\item{i-1)} for all $\s\in W^{x_1}\setminus T^{x_1}$
\be
\P_{\s}\left(\t_{\eta}<\t_{T\setminus\eta}\right)=
\frac{1}{M_2}\left[
(1-\l^{x_1}_N(T^{x_1}))
+\
\nu_1^N(x_1)\l^{x_1}_N(T^{x_1})
\right]
\left(1+\varepsilon_{N}\right),
\Eq(P.top.5)
\ee

\item{i-2)} for all $x'_1\in\MM_1\setminus x_1$ and all
$\s\in W^{x'_1}\setminus T^{x'_1}$
\be
\P_{\s}\left(\t_{\eta}<\t_{T\setminus\eta}\right)=
\frac{1}{M_2}\nu_1^N(x_1)\l^{x'_1}(T^{x'_1})
\left(1+\varepsilon_{N}\right),
\Eq(P.top.6)
\ee

\item{i-3)} for all $\s\in\VV_N\setminus\overline W$
\be
\P_{\s}\left(\t_{\eta}<\t_{T\setminus\eta}\right)=
\frac{\nu_1^N(x_1)}{M_2}
\left(1+\varepsilon_{N}\right).
\Eq(P.top.7)
\ee

\item{i-4)}  Entrance in cylinder sets. For all $\s\in \VV_N\setminus\overline W$
and all $x_1\in\cm_1$
\be
\P_\s\left(\t_{W^{x_1}\setminus T^{x_1}}
<\t_{\overline W\setminus (W^{x_1}\setminus T^{x_1})}\right)=\frac{1}{M_1}(1+\varepsilon_N).
\Eq(C.inter.3a)
\ee

\item{ii)} Level 2 transitions.
For all $\eta\in T^{x_1}$, $\bar\eta\in T^{\bar x_1}$, $x_1=\bar x_1$
$\eta\neq\bar\eta$,
\be
\P_{\bar\eta}\left(\t_{\eta}<\t_{T\setminus\{\eta,\bar\eta\}}\right)=
\frac{1}{M_2-1}\left[
(1-\l^{x_1}_N(T^{x_1}\setminus\bar\eta))+
\overline\nu_1^N(x_1)\l^{x_1}_N(T^{x_1}\setminus\bar\eta)
\right]
\left(1+\varepsilon_{N}\right).
\Eq(P.top.8)
\ee

\item{iii)} Level 1 transitions.
For all $\eta\in T^{x_1}$, $\bar\eta\in T^{\bar x_1}$, $x_1\neq\bar x_1$
\be
\P_{\bar\eta}\left(\t_{\eta}<\t_{T\setminus\{\eta,\bar\eta\}}\right)
=
\frac{1}{ M_2}\overline\nu_1^N(x_1)\l^{x_1}_N(T^{\bar x_1}\setminus\bar\eta)
\left(1+\varepsilon_{N}\right).
\Eq(P.top.9)
\ee
\end{proposition}

\begin{remark}\label{asymp}

Taking limits of the above quantities we get, for  $\psi$ is defined in \eqv(01.10)
\be\label{limla}
\lim_{N\to\infty}\l^{x_1}_N(T^{x_1})=
\begin{cases}
\quad\quad	1, &\mbox{above FT},\\
\frac1{1+M_2\psi\g_1(x_1)}, &\mbox{at FT},\\
\quad\quad		0, &\mbox{below FT}	,	
\end{cases}
\ee
which leads to
\be\label{limnu}
\lim_{N\to\infty}\nu_1^N(x_1)=\nu_1(x_1), 
\quad
\nu_1(x_1)\equiv
\begin{cases}
\frac{\g_1(x_1)}{\sum_{z_1\in\cm_1}\g_1(z_1)}	, &\mbox{above FT},\\
\frac{h(\g_1(x_1))}{\sum_{z_1\in\cm_1}h(\g_1(z_1))}, &\mbox{at FT},\\
\quad\quad	\frac1{M_1}		, &\mbox{below FT},	
\end{cases}
\ee
where $h=h_{M_2,\psi}:\mR^+\to\mR^+$ is such that $h(r)=r/(1+M_2\psi r)$. 
At FT both limits hold weakly with respect to the environment. 
Below FT, there is a window of values of $\zeta^-_N$ for which both limits hold almost surely, 
and above which both hold in probability. One may readily check that the following window has these
properties: $\zeta^-_N\ll-\log\log N$; see Lemma~\thv(L5new.lev2), and its proof.
Above FT, we have a mixed situation. For $\l^{x_1}_N(T^{x_1})$, there is a window above which the convergence
is almost sure: $\zeta^+_N\gg\log N$. And for $\nu_1^N(x_1)$, we need in addition the existence of $\lim_{N\to\infty}\zeta_N/N$, and
the convergence is weak.
The asymptotics of the probabilities follow readily. 	
\end{remark}

Getting the estimates in Proposition~\thv(P.top) above fine tuning when we do not have that $\zeta_N\ll\log N$, 
requires an extra level of precision, related to the fact that, in that regime, $\nu_1^N(\cdot)$
is a quotient of vanishing terms. We state next a separate result where we deal with this case. Since it is a limit result,
we require the existence of $\lim_{N\to\infty}\zeta_N/N$.

\begin{proposition}[Above fine tuning temperatures]
\label{paft}
Suppose~(\ref{aft}) holds. Then for all $\s\notin T$ and $\eta=\xi^y\in T$
\be\label{lim1}
\lim_{N\to\infty}\P_{\s}\left(\tau_{\eta}<\tau_{T\setminus\eta}\right)=\frac1{M_2}\nu_1(y),
\ee
where the limit holds in distribution in $(\O,\FF,\cp)$.	
\end{proposition}

The proof of Proposition \thv(P.top) follows a strategy initiated in \cite{BBG03a} and developed in \cite{BG08} which consists firstly in reducing the probabilities of interest to quantities which are functions only of the simple random walks $J^{\circ}_{N_i}$ and then, using  the lumping techniques of \cite{BG08} to express these quantities. The tools needed to implement this strategy are prepared in Section \thv(S3.3). They are used in Section \thv(2) to prove basic probability estimates for the jump chain that,  in turn, are the key ingredients of the proof of Proposition  \thv(P.top),  concluded in Section \thv(S4.4), after which we prove Proposition~\ref{paft}.

\section{Entrance law. Key tools}
\TH(S3.3)  

The section has three parts. Subsection \thv(S2.11) gathers simple lemmata needed to link probabilities for the original chain $\sigma^N$ (with rates \eqv(01.11)) to probabilities for  the jump chain  $J^{*}_N$ (with transitions \eqv(01.15))  and to link the latter to quantities depending only on the simple random walks $J^{\circ}_{N_i}$. In Subsection \thv(S2.2) we introduce the notion of lumped chain. 
Finally, in Subsection \thv(S2.3), we state and prove properties of various sets needed in Section \thv(2) to make use of known lumped chain estimates from \cite{BG08}. The last two  sections can be skipped at first reading.

The RHD can be alternatively, described through its jump chain, $J^{*}_N$ and jump rates $w_N$, where 
\be
w_N(\s)=\sum_{\s'\sim \s}w_N(\s,\s')
\Eq(01.12)
\ee
and  $(J^{*}_N(i), i\in\N)$ is the discrete time Markov chain with one step transition probabilities
\be
p^{*}_N(\s,\s')=w_N(\s,\s')w^{-1}_N(\s).
\Eq(01.13)
\ee
Introducing the parameters 
\be
q^*_N(\s_1)\equiv \frac1{1+\frac{N_2}{N_1}e^{-\b H^{(1)}_N(\s_1)}}
\Eq(01.14)
\ee
we have
\be
p^{*}_N(\s,\s')
=q^*_N(\s_1)p^{\circ}_{N_1}(\s_1,\s'_1)
+(1-q^*_N(\s_1))p^{\circ}_{N_2}(\s_2,\s'_2)
\Eq(01.15)
\ee
where, for $i=1,2$, $p^{\circ}_{N_i}(\s_i,\s'_i)=N_i^{-1}$ if $\s_i\sim\s'_i$ and  $p^{\circ}_{N_i}(\s_i,\s'_i)=0$ 
else denote the one step transition probabilities of the simple random walk $(J^{\circ}_{N_i}(j), j\in\N)$ on $\VV_{N_i}$.
The jump chain $J^{*}_N$ 
is reversible w.r.t.~the measure $G^{*}_{\b,N}$ defined through 
\be
G^{*}_{\b,N}(\s)=w_N(\s)e^{-\b H_N(\s)}\left(Z^{*}_{\b,N}\right)^{-1}=(N_1/N)\left(q^*_N(\s_1)Z^{*}_{\b,N}\right)^{-1}
\Eq(01.21)
\ee
where $Z^{*}_{\b,N}$ is a normalization making this measure a probability.

For future reference, we call $\P^{*}$ the law of the process $J^{*}_N$ conditional 
on $\FF$. We denote by $\P^{\circ, i}$ the law of $J^{\circ}_{N_i}$, $i=1,2$. If the initial state, say $\eta$, has 
to be specified we write  $\P_{\eta}$, $\P^{*}_{\eta}$ and $\P^{\circ, i}_{\eta_i}$. 
We will denote by $\cp\otimes\P_\eta$ the probability measure obtained by integrating $\P_\eta$ with respect to $\cp$.
Expectation with respect to 
$\P$, $\P^{*}$, $\P^{\circ,i}$, $\PP$ and $\cp\otimes\P_\mu$ are denoted  by  $\E$, $\E^{*}$, $\E^{\circ, i}$, $\EE$
and $\ce\otimes\E_\mu$, respectively, where $\mu$ is the uniform probability measure on $\VV_N$.

\subsection{Comparison lemmata} 
    \TH(S2.11)  

Our starting point is the observation that
\begin{lemma} 
  \TH(comp.lem1)
For all  $A,B\subseteq \VV_{N}$ such that $A\cap B=\emptyset$ and for all $\s\in\VV_N\setminus (A\cup B)$,
\be
\P_{\s}\left(\t_{A}<\t_{B}\right)=\P^*_{\s}\left(\t_{A}<\t_{B}\right).
\Eq(comp.lem1.1)
\ee
\end{lemma} 
We skip the proof of Lemma \thv(comp.lem1) which is elementary. 

The next two lemmata deal with two classes of events that can be expressed through 
just  one of the simple random walks $J^{\circ}_{N_i}$ on $\VV_{N_i}$.
The first are \emph{REM-like events} that can be reduced to those of a REM (which is a 1-level GREM).
Let $\pi_i$, $i=1,2$, denote the canonical projection of $\VV_N$ onto $\VV_{N_i}$, that is
\be
\pi_i\s= \s_i.
\Eq(comp.1)
\ee

\begin{lemma}[REM-like events]
  \TH(L.comp.1)
Let $A,B\subseteq \VV_{N_1}$ be such that
$A\cap B=\emptyset$. Then, for all $\s_1\in\VV_{N_1}\setminus A$ and all $\s\in\pi_1^{-1}\s_1$,
\be
\P^*_{\s}\left(\t_{\pi_1^{-1}A}<\t_{\pi_1^{-1}B}\right)
=\left(1+\d_{\s_1\in B}(q^*_N(\s_1)-1)\right)\P^{\circ,1}_{\s_1}\left(\t_{A}<\t_{B}\right).
\Eq(comp.2)
\ee
\end{lemma}

\begin{proof}
Note that
$
\P^*_{\s}\bigl(\t_{\pi_1^{-1}A}<\t_{\pi_1^{-1}B}\bigr)
=\P^{\pi_1}_{\s_1}\left(\t_{A}<\t_{B}\right)
$
where $\P^{\pi_1}_{\s_1}= \P^*_{\s}\circ\pi_1^{-1}$ denotes the law of the projection $\pi_1J^{*}_N$ 
of the jump chain on $\VV_{N_1}$. By \eqv(01.15) this is  a Markov chain with transition probabilities
$
p^{\pi}_N(\s_1,\s'_1)
=q^*_N(\s_1)p^{\circ}_{N_1}(\s_1,\s'_1)\1_{\s_1\sim\s'_1}
+(1-q^*_N(\s_1))\1_{\s_1=\s'_1}
$.
The lemma now easily follows.
\end{proof}

Given $\s=\s_1\s_2\in\VV_N$, define the cylinder sets
\be
C(\s_i)\equiv\pi_i^{-1}\s_i= \{\s'\in\VV_N\mid \s'_i=\s_i\}, \quad i=1,2.
\Eq(P.inter.12)
\ee
The next lemma deals with so-called \emph{level-$2$ events}, namely, events whose 
trajectories are confined to a given cylinder  set $C(\s_1)$, and  that can thus be expressed 
through  just the simple random walk $J^{\circ}_{N_2}$ on $\VV_{N_2}$.
Define the outer boundary of a set $A\subset\VV_N$ as
\be
\del A\equiv\{\s'\in( \VV_N\setminus A)\mid
\exists\, \s\in A\,\,\hbox{\rm s.t.}~\s\sim\s'\}.
\Eq(comp.6)
\ee

\begin{lemma} [Level-2 events]
  \TH(L.comp.4)
 Given $\s_1\in\VV_{N_1}$, let $A,B\subseteq C\equiv C(\s_1)$ be such that $A\cap B=\emptyset$. 
 Set  $u(\s_1)\equiv\log\left(1-q^*_N(\s_1)\right)$. Then, for all $\s\in C$,
\be
\P^*_{\s}\left(\t_{A}\leq\t_{B\cup\del C}\right)
=
\begin{cases}
\E^{\circ,2}_{\s_2}\left(
e^{u(\s_1){\t_{\pi_2A}}}
\1_{\{\t_{\pi_2A}< \t_{\pi_2B}\}}\right)
& \text{if}\,\, B\neq \emptyset,
\\
\E^{\circ,2}_{\s_2}\left(
e^{u(\s_1){\t_{\pi_2A}}}
\right)
& \text{if}\,\, B=\emptyset.
\end{cases}
\Eq(comp.16')
\ee
\end{lemma} 

\begin{proof}
Write
\be
\textstyle
\P^*_{\s}\left(\t_{A}\leq\t_{B\cup\del C}\right)=\sum_{k=1}^{\infty}
\P^*_{\s}\left(k=\t_{A}\leq\t_{B}\mid \t_{\del C}>k\right)
\P^*_{\s}\left(\t_{\del C}>k\right)
\ee
and note that by \eqv(01.15), $\t_{\del C}$ is a geometric r.v.~with success probability 
$q^*_N(\s_1)$. Thus, on the one hand,
$
\P^*_{\s}\left( \t_{\del C}>k\right)
=
\left(1-q^*_{1}(\s_{1})\right)^k
$
while 
$
\P^*_{\s}\left(k=\t_{A}\leq\t_{B}\mid \t_{\del C}>k\right)
=\P^{\circ,2}_{\pi_2\s}\left(k=\t_{\pi_2A}\leq\t_{\pi_2B}\right)
$
on the other hand.
\end{proof}

The probabilities $\P^{\circ,1}_{\s_1}\left(\t_{A}<\t_{B}\right)$ appearing in Lemma \thv(L.comp.1) and the Laplace transform of Lemma \thv(L.comp.4) are estimated in \cite{BG08} using lumping techniques. We briefly recall the basics of lumping in Subsection \thv(S2.2)  and collect in Subsection \thv(S2.3) the ingredients that are needed to make use of the results of  \cite{BG08},
i.e.,  to check that their conditions of validity are satisfied.
For the one-dimensional case, lumping reduces to the classical Ehrenfest chain. 
We recall an expression for the probability generating function of 
hitting times of such chain, appearing in~\cite{K} (see (4.28,29) in that reference), and to be used in a later section. 
 
For $t\in[0,1)$ and $\s_2,\s_2'\in\VV_{N_2}$
\begin{equation}
\label{fr1}
\E^{\circ, 2}_{\s_2}(t^{\tau_{\s'_2}})=
\frac{B_i(t')}{B_0(t')},
\end{equation}
where $i=d_2(\s_2,\s_2')$, $t'=\frac{N_2}2\frac{1-t}t$, and for $\a>0$
\begin{equation}
\label{fr2}
B_i(\a)=\int_0^1(1-u)^i(1+u)^{N_2-i}u^{\a-1}du=\sum_{j=0}^{N_2-i}{N-i\choose j}\frac{\G(i+1)\G(\a+j)}{\G(\a+i+j+1)}.
\end{equation}

\subsection{Lumped chains and $K$-lumped chains.} 
 \TH(S2.2)  
 
 In this section we introduce certain functions of the simple random walks $J^{\circ}_{N_i}$ on $\VV_{N_i}$,  $i\in\{1,2\}$,
that play a key r\^ole in our proofs. Fix $i\in\{1,2\}$. Given a partition  $\L_i$ of $\{1, . . . , N_i\}$ into $d$ classes, that is, 
non-empty disjoint subsets $\L^1_i,\dots,\L^d_i$, $1\leq d\leq N_i$, satisfying $\L^1_i\cup\dots\cup\L^d_i=\{1,...,N_i\}$, 
let $m_{i}$ be the many-to-one function that maps the elements of $\VV_{N_i}$ onto $d$-dimensional vectors
\be
 m_{i}(\s_i)=\left(
 m_{i}^1(\s_i),\dots, m_{i}^{k}(\s_i),\dots, m_{i}^d(\s_i)
\right),
\,\,\,\,\,\s_i\in\VV_{N_i}
\Eq(lump.7)
\ee
by setting, for all $k\in\{1,\dots,d\}$,
\be
 m_{i}^{k}(\s_i)\equiv\frac{1}{|\L^k_i|}\sum_{j\in\L^k_i}\s_{i,\underline j}
\Eq(lump.8)
\ee
where $\s_{i,\underline j}$ denotes the $j$-th cartesian co-ordinate of $\s_i$. 
The image $\overline I_{i}\equiv m_{i}(J^{\circ}_{N_i})$  of the simple random walk $J^{\circ}_{N_i}$, 
called \emph{lumped chain},  also is a Markov chain that now takes value in a discrete grid $\G_{N_i,d}$ 
that contains $\VV_{d}=\{-1,1\}^d$. This d-dimensional process was exploited for the study of the 
dynamics of the random field Curie-Weiss model in \cite{BEGK1}, and  of the Random Energy Model (REM) 
in \cite{BBG03a,BBG03b}. It was later studied in detail in \cite{BG08} in view, in particular, of dealing with more 
involved spin-glass models such as the GREM. We extensively use the results of \cite{BG08} in the sequel.

Different choices of the partition $\L_i$ yield different lumped chains. Given an integer $n$ and 
a collection $K=\left\{\eta^1,\dots,\eta^{x},\dots,\eta^n\right\}$ of elements of $\VV_{N_i}$, 
the so-called \emph{$K$-lumped chain} is induced  by a partition $\L_i(K)$ of $\{1, . . . , N_i\}$ 
into $d=2^n$ classes,   $\L^1_i(K)\cup\dots\cup\L^d_i(K)=\{1,...,N_i\}$, defined as follows. 
Let us identify the set $K$ with the $n\times N_i$ matrix $(\eta^{x}_{j})_{j=1,\dots,N_i}^{x=1,\dots,n}$
whose row vectors are the $\eta^{x}$'s, that is,
$
\eta^{x}\equiv(\eta^{x}_{j})_{j=1,\dots,N_i}\in\VV_{N_i}
$,
$x\in\{1,\dots,n\}$,
and denote by $\eta_{j}$ the column vectors
$
\eta_{j}\equiv(\eta^{x}_{j})^{x=1,\dots,n}\in\VV_{n}
$,
$j\in\{1,\dots,N_i\}$.
Given an arbitrary labelling  $\{e_1,\dots,e_k,\dots,e_{d}\}$ of the set of all $d=2^{n}$
elements of $\VV_{n}$, we then set
\be
\L^k_i(K)\equiv\{j\in \{1,\ldots,N_i\}\mid \eta_j=e_k\}, \quad 1\leq k\leq d.
\Eq(lump.12)
\ee
We denote by $m_{i,K}$ the function \eqv(lump.7)-\eqv(lump.8) resulting from \eqv(lump.12),  by
\be
\overline I_{i,K}\equiv m_{i,K}(J^{\circ}_{N_i})
\Eq(lump.13)
\ee
the associated $K$-lumped chain and by $\overline\P^{i,K}$  its law.

\subsection{Properties of the Top and other sets}
    \TH(S2.3)  
    
The aim of this section is to facilitate the use of results of  \cite{BG08} for $K$-lumped chains 
by establishing that certain conditions, that only depend on the set $K$ and the partition 
\eqv(lump.12), are verified for three types of sets $K$ that we encounter in our proofs: the Top, 
the Top plus a non random point, and large random subsets of $\VV_{N_1}$.

In what follows $K=\left\{\eta^1,\dots,\eta^{x},\dots,\eta^{m_i}\right\}$ denotes a collection of $m_i$ elements of $\VV_{N_i}$, 
and $\L_i(K)$ is the partition of $\{1, . . . , N_i\}$ into $d_i=2^{m_i}$ classes, $\L_i^k(K)$, $1\leq k\leq d_i$, induced by $K$ through \eqv(lump.12).

\subsubsection{The Top.}
 \TH(S3.3.1) 
Consider the partitions $\L_1(T_1)$ and $\L_2(T^{x_1})$ induced respectively by  
$T_1$ and  $\pi_2T^{x_1}$, $x_1\in\MM_1$, through \eqv(lump.12).
Let $K$ be any of the sets $T_1$ or $\pi_2T^{x_1}$, $x_1\in\MM_1$ (thus $m_i=M_i$ and
$i=1$ if $K=T_1$ and $i=2$ if $K=\pi_2T^{x_1}$, $x_1\in\MM_1$). 
Introducing the sets
\be
\overline\O_{N_i}(K)\equiv
\left\{
\left|
\frac{d_i}{N_i}|\L_i^k(K)|-1
\right|
<
\d_i(N_i), \,\,1\leq k\leq d_i
\right\},\,\,
\d_i(N_i)
\equiv2\sqrt{\frac{d_i}{N_i}}\log{N_i}
\Eq(ls.10')
\ee
and
$
\textstyle\overline\O_i(K)\equiv\bigcup_{N_i'\geq i}\bigcap_{N_i\geq N_i'}\overline\O_{N_i}(K),
$
define
\be
\textstyle
\overline\O
\equiv
\overline\O_1(T_1)\bigcap\left(\bigcap_{x_1\in\MM_1}\overline\O_2(T^{x_1})\right).
\Eq(ls.10)
\ee

\begin{lemma}
  \TH(L1.ls) 
  $\PP(\overline\O)=1$.
\end{lemma}

\begin{proof}[Proof of Lemma \thv(L1.ls)]
The proof is an easy adaptation of that of Lemma 4.2 of \cite{Ga}.
\end{proof}

For $i=1,2$, $\eta\in\VV_{N_i}$ and $\rho>0$ set
$
\BB_{\rho}(\eta)=\left\{\s\in\VV_{N_i}\mid\dist(\s,\eta)\leq \rho\right\}
$.

\begin{lemma}
  \TH(L2.ls)
On $\overline\O$, for  all large enough $N$ the following holds: 
denoting by $K$ any of the sets $T_1$, $\pi_2T^{x_1}$, $x_1\in\MM_1$, or 
$
\cup_{x_1\in\MM_1}\pi_2T^{x_1}
$
we have, for all $\eta\in K$  and $\bar\eta\in K$, $\eta\neq\bar\eta$, 
\be
\left|\dist(\eta,\bar\eta)-(N_i/2)\right|\leq(N_i/2)\d_i(N_i)
\Eq(ls.15)
\ee
and for all $0\leq\e< 1/2$
\be
\BB_{\e N_i}(\eta)\cap\BB_{\e N_i}(\bar\eta)=\emptyset.
\Eq(ls.14)
\ee

\end{lemma}

\begin{proof}
This is the analogue of Lemma 2.12 of [BBG1]. It is proved in the same way.
\end{proof}

Let $r=p/(1-p)$. Let $C_N$ be the event that for all $ \eta_1\in T_1$ and $\s_1\in\VV_{N_1}$ such that $d_1(\s_1,\eta_1)\leq\eps_0N_1$ we have that 
\begin{equation}\label{bin}
\#\{\s_1'\in\VV_{N_1}:\,\s'_1\sim \s_1\mbox{ and }q^*_N(\s'_1)\leq re^{-\sqrt{N}}\}\geq \eps_1N_1,
\end{equation}
\begin{lemma}\label{rm2}
	There exists $\eps_0,\eps_1>0$ such that $\cp(C_N)\to1$ as $N\to\infty$.
\end{lemma}

\begin{proof}
Since $\{\Xi^{(1)}_{\s_1},\,\s_1\in\VV_{N_1}\}$ are iid standard Gaussian random variables,
we have that the left hand side of~(\ref{bin}) above dominates a binomial random variable 
with $N_1$ trials and probability of success $\Phi(-1/(\beta\sqrt a))$ in each trial, where $\Phi$ is the standard Gaussian distribution function. Therefore, by a classical large deviation bound, there exists $\eps_1>0$ such that the probability of the complement of~(\ref{bin}) may be bounded above by $c_1 2^{-\eps_1'N_1}$
for some constant $c_1$, and $\eps'_1>0$. 

Now the probability of $C_N^c$ may be bounded above by
$$
c_1M_1(\eps_0 N_1+1)\sqrt{N_1}{N_1\choose \eps_0 N_1}2^{-\eps_1'N_1}
\leq c_0N_12^{-(\eps_1'-\eps_0')N_1},
$$
for some constant $c_0$, and $\eps'_0=\eps'_0(\eps_0)>0$ such that $\eps'_0\to0$ as $\eps_0\to0$; the result follows by
choosing $\eps_0,\eps_1>0$ such that $0<\eps'_0<\eps'_1$.
\end{proof}

The following almost sure (but rough) bounds  on  the ranked variables $\gamma^N_1(\xi_1^{x_1})$ are needed in the sequel. Let $\wh\O\equiv\cap_{M>1}\wh\O_{M}$ where
$
\wh\O_{M}\equiv\bigcup_{N_1'\geq 1}\bigcap_{N_1\geq N_1'}\wh\O_{M,N_1}
$
and
\be
\wh\O_{M,N_1}\equiv
\bigcap_{1\leq x_1\leq M}\left\{
\o\in\O\mid
N_1^{-2/\a_1^N}\leq \left(\gamma^N_1(\xi_1^{x_1})\right)^{-1}< (\ln N_1)^{2/\a_1^N}
\right\}.
\Eq(L5new.lev2.2)
\ee
\begin{lemma}
	\TH(L5new.lev2) $\PP(\wh\O)=1$.
\end{lemma}
\begin{proof}[Proof of Lemma \thv(L5new.lev2)]
	By \eqv(01.3) and \eqv(01.17),
	\be
	\exp\left\{-u^{-1}_{N_1}(\Xi^{(1)}_{\xi_1^{1}})/\a_1^N\right\}
	\leq
	\left(\gamma^N_1(\xi_1^{x_1})\right)^{-1}
	\leq
	\exp\left\{-u^{-1}_{N_1}(\Xi^{(1)}_{\xi_1^{M}})/\a_1^N\right\}
	\ee
	for each $1\leq x_1\leq M$.
	Using  the  well known  asymptotic distribution of  $\Xi^{(1)}_{\xi_1^{k}}$  (the $k$-th extreme order statistics, 
	see e.g.~\cite{LLR} Section 2) we get that
	\be
	\PP(
	u^{-1}_{N_1}(\Xi^{(1)}_{\xi_1^{M}})\leq -2\ln\ln N_1
	)
	\leq e^{-(\ln N_1)^2}\sfrac{(\ln N_1)^{2M}}{M!}(1+o(1)) < N_1^{-2}(1+o(1))
	\Eq(L5new.lev2.3)
	\ee
	and
	$
	\PP(
	u^{-1}_{N_1}(\Xi^{(1)}_{\xi_1^{1}})> 2\ln N_1
	)
	\leq N_1^{-2}(1+o(1)).
	$
	The lemma now easily follows from Borel-Cantelli Lemma.
\end{proof}

In order to make use of the results of \cite{BG08} we need upper bounds on the following key quantities:
given a subset $A$  of $\VV_{N_i}$ define
\bea
U_{N_i,d_i}(\s,A)
&\equiv&
\textstyle
\sum_{\eta\in A\setminus \s}F_{N_i,d_i}(\dist(\s,\eta)), \quad  \s\in\VV_{N_i},
\Eq(srwB.1)
\\
\UU_{N_i,d_i}(A)
&\equiv&
\textstyle
\max_{\s\in A}U_{N_i,d_i}(\s,A),
\Eq(srwI.12)
\eea
where $F_{N_i,d_i}$ is a function depending on $N_i$ and $d_i$, whose
definition is stated in  (3.5)-(3.8) of Section 3 of \cite{BG08}
and whose properties are analyzed in detail in Appendix A3 of \cite{BG08}. We do not repeat its lengthy definition.
We set $\UU_{N_i,d_i}(A)=0$ if $A=\emptyset$ and $U_{N_i,d_i}(\s,A)=0$ if $A\setminus \s=\emptyset$.

\begin{remark}
  \TH(R1.ls)
Upper bounds of the functions $F_{N_i,d_i}$, $U_{N_i,d_i}$ and $\UU_{N_i,d_i}$ 
imply upper bounds on the quantities $\phi$, $V_{N_i,d_i}$ and $\VV_{N_i,d_i}$
defined (with obvious notation) in (4.1), (5.4) and (5.9) of \cite{BG08}
(see Lemma 4.2, Lemma 5.2,  Lemma 5.4, and Lemma 5.7 of \cite{BG08}). Furthermore, they imply 
upper bounds on the quantities $U^{\circ}_{N_i,d_i}$, $\UU^{\circ}_{N_i,d_i}$, $V^{\circ}_{N_i,d_i}$ 
and $\VV^{\circ}_{N_i,d_i}$ for the associated lumped chain, defined  in (5.3), (5.10), and (5.11) of  \cite{BG08}
(see Lemma 5.3 and  Lemma 5.5 of  \cite{BG08}). We do not repeat these arguments in
the proofs of the statements of Section \thv(S2.1) (namely, in Section \thv(S2.3)).
\end{remark}

\begin{lemma}
  \TH(L3.ls)
With the notation of Lemma \thv(L2.ls), the following holds on $\overline\O$ for  all large enough $N$: 
for all $\eta\in K$  and $\bar\eta\in K$, $\eta\neq\bar\eta$, 
\be
F_{N_i,d_i}(\dist(\eta,\bar\eta))
\leq 2^{-N_i/4},
\Eq(ls.16)
\ee
\be
U_{N_i,d_i}(\eta,K)\leq\UU_{N_i,d_i}(K)
\leq |K|\,2^{-N_i/4}.
\Eq(ls.17)
\ee
\end{lemma}

\begin{proof}
Eq.~\eqv(ls.16) follows from \eqv(ls.15)  of Lemma \thv(L2.ls) and (10.7) of Lemma 10.1 of \cite{BG08} and implies the leftmost inequality of \eqv(ls.17) which in turn implies the rightmost one.
\end{proof}

For any $\s\in\VV_{N_i}$ and any subset $A\subset\VV_{N_i}$,  set
\be
j(\s,A)\equiv
\begin{cases}
1& \text{if}\,\, \dist(\s,A)=1,\\
2 & \text{else}.\\
\end{cases}
\Eq(ls.28bis)
\ee
\begin{lemma}
   \TH(C5.ls)
With the notation of Lemma \thv(L2.ls), the following holds on $\overline\O$ for all large enough $N$: 
for all $\eta\in K$ and all $\s\in\VV_{N_i}\setminus K$,
\bea
&&
F_{N_i,d_i}(\dist(\s,\eta))\leq \frac{j}{N_i^j}(1+o(1)),\quad j=j(\eta,\s),
\Eq(ls.29)
\\
&&
U_{N_i,d_i}(\s,K)\leq\frac{j}{N_i^j}(1+o(1)),\quad j=j(\s,K).
\Eq(ls.28)
\eea
\end{lemma}

\begin{proof}
Eq.~\eqv(ls.29) follows from the definition of the definition of $F$ and case (a) and (b) of Lemma 10.1 of \cite{BG08}.
To prove \eqv(ls.28)  we distinguish two cases:
(a) there exists $\eta\in K$ such that $\dist(\s,\eta)\leq \e N_i$ for some
$0\leq\e< 1/2$ and 
(b) for all $\eta\in K$, $\dist(\s,\eta)> \e N_i$.
In case (b) we have:
$
U_{N_i,d_i}(\s,K)\leq|K|\,2^{-N_i/4}\leq o(N_i^{-2})
$.
This is proven just as \eqv(ls.17). In case (a) we write
\be
\textstyle
U_{N_i,d_i}(\s,K)=F_{N_i,d_i}(\dist(\s,\eta))+\sum_{\eta'\in K\setminus \eta}F_{N_i,d_i}(\dist(\s,\eta')).
\Eq(ls.31)
\ee
By \eqv(ls.14) of Lemma \thv(L2.ls) we may apply the bound just obtained in case (b)
to bound the second term (namely the sum) in the right-hand side of \eqv(ls.31)
whereas the  first term  is bounded as in \eqv(ls.29).
\end{proof}

\subsubsection{The Top and a non random point.}
 \TH(S3.3.2) 
We will frequently need to lump the simple random walk $J^{\circ}_{N_i}$ on $\VV_{N_i}$ with sets of the form
$
K\cup\s_i\subset\VV_{N_i}
$
where $\s_i\in\VV_{N_i}$ is arbitrary and $K=T_1$ (then $i=1$) or $K=\pi_2T^{x_1}$ for some $x_1\in\MM_1$ (then $i=2$).
We are now interested in the partition  $\L_i(K\cup\s_i)$ of $\{1, . . . , N_i\}$ into $d'_i=2^{M_i+1}$ classes, 
$\L_i^k(K\cup\s_i)$, $1\leq k\leq d'_i$, induced by $K\cup\s_i$ through \eqv(lump.12). 
 Lemma \thv(L3.ls) and Lemma \thv(C5.ls) can be extended  to this setting as follows.

\begin{lemma}
  \TH(L6.ls)
With the notation of Lemma \thv(L2.ls), the following holds on $\overline\O$ for all large enough $N$: 
 for all $\eta\in K$ and all $\s'\in\VV_{N_i}\setminus (K\cup \s_i)$,  $\s''\in\VV_{N_i}$,
\bea
F_{N_i,d'_i}(\dist(\s',\s''))&\leq& ({j}/{N_i^j})(1+o(1))\,,\quad j=j(\s',\s'')
\Eq(ls.33)
\\
\UU_{N_i,d'_i}(K\cup\s_i)&\leq&({j}/{N_i^j})(1+o(1)), \quad j=j(\s_i,K).
\Eq(ls.36)
\eea
\end{lemma}

\begin{proof}
This is a simple adaptation of the proofs of Lemma \thv(L3.ls) and Lemma \thv(C5.ls). 
\end{proof}

\subsubsection{Large random subsets of $\VV_{N_1}$}
 \TH(S3.3.3) 

Given a positive decreasing sequence $\e_{N_1}$ satisfying
$\lim_{N_1\uparrow\infty}\e_{N_1}=0$, define
\bea
\VV^+_{N_1}&=\{\s_1\in\VV_{N_1}\mid \Xi^{(1)}_{\s_1}\geq \e_{N_1}\}.
\Eq(lsb.1)
\eea
The cardinality of this set typically grows exponentially fast in $N_1$ so that, typically, most classes
of the partition $\L_1(\VV^+_{N_1})$ defined through \eqv(lump.12)  will either be empty or contain 
a single element, which renders the construction of a  lumped chain based on this partition meaningless.
The aim of this subsubsection is to show that $\VV^+_{N_1}$ will nevertheless contain a large sparse set, 
$K^+_{N_1}$, whose size diverges with $N_1$ and such that all classes of the partition $\L_1(K^+_{N_1})$ 
also have diverging sizes. To do this, we first construct a deterministic set $K_{N_1}\subset\VV_{N_1}$ with 
these properties and  next show that the intersection $K^+_{N_1}\equiv\VV^+_{N_1}\cap K_{N_1}$ roughly 
contains  half the elements of $K_{N_1}$.  The idea behind the construction of $ K_{N_1}$ is simple:  
rather than constructing a partition given a set as in \eqv(lump.12), we reverse the procedure, namely, 
we fix a partition $\L_1$ and construct a set of configurations $K_{N_1}$ such that  $\L_1=\L_1(K_{N_1})$. 

More precisely, let $\L_1$ be a given partition of $\{1, . . . , N_1\}$ into $d$ non empty classes, 
$\L_1^x$, $1\leq x\leq d$. Next, let $\VV_d=\{e_1,\dots,e_k,\dots,e_{2^d}\}$, be an arbitrary 
labelling of all $2^d$ elements $e_k=(e_k^{x})_{1\leq x\leq d}$ of the $d$-dimentional discrete
cube $\VV_{d}$. To each $1\leq k\leq 2^d$ we uniquely associate an element 
$\zeta^{k}=(\zeta^{k}_j)_{1\leq j\leq  N_1}$ of $\VV_{N_1}$ defined through
\be
\zeta^{k}_j=e_k^{x} \text{ for all } j\in\L_1^{x}, 1\leq x\leq d.
\Eq(lsb.6)
\ee
These are configurations that are piecewise constant on the sets $(\L_1^{x})_{1\leq x\leq d}$.
We then define $K_{N_1}$ as the set of all $2^d$ such configurations:
\be
K_{N_1}=(\zeta^{k})_{1\leq k\leq 2^d}.
\Eq(lsb.7)
\ee
Clearly, applying the construction \eqv(lump.12) to the set $K_{N_1}$ yields $\L_1(K_{N_1})= \L_1$ as desired.

The point of interest is of course to choose $d=d(N_1)$ as an increasing function of $N$.

\begin{lemma}
  \TH(L2.lsb)
Assume that
$
d\equiv d(N_1)=o(N_1)
$ 
and let $\L_1$ be any partition of $\{1, . . . , N_1\}$ into $d$ classes $\L_1^x$, $1\leq x\leq d$,
satisfying
\be
|\L_1^{x}|=({N_1}/{d})(1+o(1))\,,\quad 1\leq x\leq d.
\Eq(lsb.5)
\ee
Then, there exists a constant $0<\rho<1$ such that for all $1\leq k\leq 2^d$
\bea
\UU_{N_1,d}(K_{N_1})
&\leq& \rho^{N_1/d}.
\Eq(lsb.8)
\eea
\end{lemma} 

\begin{proof}
We first prove  \eqv(lsb.8).
By \eqv(srwB.1), \eqv(srwI.12) and the definition of $K_{N_1}$,
$
\UU_{N_1,d}(K_{N_1})=\max_{k'\in\{1,\dots,2^d\}}U(\zeta^{k'},K_{N_1})
$
where for each $1\leq k'\leq 2^d$
\be
U(\zeta^{k'},K_{N_1})=\sum_{k\in\{1,\dots,2^d\}\setminus k'}F_{N_1,d}(\dist(\zeta^{k},\zeta^{k'})).
\Eq(lsb.9)
\ee
Note that for any pair $\zeta^{k'}, \zeta^k\in K_{N_1}$,
by  \eqv(lsb.6),
$
\dist(\zeta^{k},\zeta^{k'})=\sum_{1\leq x\leq d : e^{x}_k\neq e^{x}_{k'}}|\L^{x}_1|.
$
Thus
\be
\sum_{k\in\{1,\dots,2^d\}\setminus k'}F_{N_1,d}(\dist(\zeta^{k},\zeta^{k'}))
=\sum_{n=1}^d\sum_{\II\subseteq\{1,\dots,d\} : |\II|=n}F_{N_1,d}(\sum_{x\in\II}|\L^{x}_1|)
\Eq(lsb.10)
\ee
Now, by \eqv(lsb.5), $\sum_{x\in\II}|\L^{x}_1|=|\II|\frac{N_1}{d}(1+o(1))$.
Inserting this in \eqv(lsb.10) using the estimates of Lemma 10.1 of \cite{BG08} 
to evaluate the resulting expression yields
$
U_{N_1,d}(\zeta^k,K_{N_1})\leq \rho^{N_1/d}
$,
which in turn implies \eqv(lsb.8).
This concludes the proof of the lemma.\end{proof}

\begin{lemma}
  \TH(L3.lsb)
Under the assumptions of Lemma \thv(L2.lsb) the following holds:
for all $\zeta^k\in K_{N_1}$ and all $\s_1\in\VV_{N_1}\setminus K_{N_1}$, for $j(\cdot,\cdot)$ defined in \eqv(ls.28bis)
\be
F_{N_1,d}(\dist(\s_1,\zeta^k))\leq ({j}/{N_1^j})(1+o(1)),
\quad j=j(\s_1,\zeta^k).
\Eq(lsb.14)
\ee
\end{lemma}

\begin{proof}
Eq.~\eqv(lsb.14) is proved just as \eqv(ls.29).
\end{proof}

We next want to construct a set $K_{N_1}$ that contains a prescribed subset of configurations.

\begin{lemma}
  \TH(L4.lsb)
One can always construct the set $K_{N_1}$ defined in \eqv(lsb.7) in such a way that:
i) $T_1\subset K_{N_1}$,
and ii) the assumptions of lemma \thv(L2.lsb) are satisfied on the set $\overline\O_1(T_1)$.
\end{lemma}

\begin{proof}
To construct such a set $K_{N_1}$, start from the partition
$
\L^1_1(T_1)\cup\dots\cup\L^{d_1}_1(T_1)=\{1,...,N_1\}
$ 
induced by $T_1$ through  \eqv(lump.12) and partition each of the sets $\L^k_1(T_1)$ into 
$d$ subsets that satisfy \eqv(lsb.5). This induces a partition of $\L_1$ into $d d_1$ subsets which, 
by Lemma  \thv(L1.ls), satisfies \eqv(lsb.5) for all $\o\in\overline\O_1(T_1)$ (see \eqv(ls.10)) 
and all large enough $N_1$.
\end{proof}

\begin{remark}
  \TH(R1.lsb)
 For $K_{N_1}$ as in Lemma \thv(L4.lsb), $T_1\subseteq \VV^+_{N_1}\cap K_{N_1}$.
\end{remark}

Now let $K_{N_1}$ (and thus let $d$) be given and, for $\VV^+_{N_1}$ as in \eqv (lsb.1), set
\be
K_{N_1}^+\equiv\VV^+_{N_1}\cap K_{N_1}
\Eq(lsb.19)
\ee
Writing $p=p(N_1)=\int_{\e_{N_1}}^{\infty}\sfrac{dx}{\sqrt{2\pi}}e^{-\frac{x^2}{2}}=\frac{1}{2}(1+\OO(\e_{N_1}))$, define the sets
\bea
&
\textstyle
\O_{N_1}^+\equiv\left\{\o\in\O\,\,\left|\,\,\bigl||K_{N_1}^+|-p2^d\right|\leq \sqrt{8p(1-p)2^d\log N_1}\right\},
\Eq(lsb.21)
&
\\
&
\textstyle
\O_1^+\equiv\bigcup_{N_1'\geq 1}\bigcap_{N_1\geq N_1'}\O_{N_1}^+.
&
\Eq(lsb.20)
\eea

\begin{lemma}
  \TH(L5.lsb)
 Let $d$ be such that $\frac{\log{N_1}}{2^d}=o(1)$.
Then $\PP\left(\O_1^+\right)=1$.
\end{lemma}

\begin{proof} 
We skip this elementary proof.
\end{proof}

\section{ Basic estimates for the jump chains.}
    \TH(2)

This Section is concerned with the jump chain only. We state and prove a collection of probability estimates that will later be shown, in Section \thv(S4.4), to form the basic blocks of the proof of Proposition \thv(P.top).

 \subsection{Main estimates} 
    \TH(S2.1)  

We recall that the sets $\overline\O$ and $\O_1^+$ are defined in \eqv(ls.10) and \eqv (lsb.20) respectively.
We drop the dependence on $N$ in the notation from now on.

\begin{proposition}[REM-like estimates]
    \TH(P.rem-like)
On $\overline\O$, for  all large enough $N$, the following holds:
for all $x_1\in\MM_1$ and all $\s\in \VV_N\setminus\overline W$,
\be
\left|
\P^*_{\s}\left(\t_{W^{x_1}}<\t_{\overline W\setminus W^{x_1}}\right)
-\frac{1}{M_1}
\right|=
\OO\left(N_1^{-i}\right)
\Eq(P.rem-like.1)
\ee
where $i=1$ if $\dist(\s_1,\xi_1^{x_1})=1$ and $i=2$ otherwise.
\end{proposition}

\begin{proposition}[Level-$2$ motion]
   \TH(P1.lev2)

On $\overline\O$, for  all large enough $N$, the following holds:
For $x_1\in\MM_1$ and  $A\subset T^{x_1}$, set
\bea
u(\xi^{x_1}_1)&=&\log(1-q^*(\xi^{x_1}_1)),
\Eq(P1.lev2.001)
\\
s^{x_1}(A)
&=&\bigl|u(\xi^{x_1}_1)\bigr|\frac{2^{N_2}}{|A|}(1+N_2^{-1}),
\\
\Eq(P1.lev2.01)
\l^{x_1}(A)
&=&\frac{s^{x_1}(A)}{1+s^{x_1}(A)}.
\Eq(P1.lev2.1)
\eea
Then, for all non empty subset $A\subseteq T^{x_1}$ and all $\s\in W^{x_1}\setminus A$, we have:

\item{i)} (Motion within the cylinder set $W^{x_1}$.) For all $\eta\in A$,
\be
\P^*_\s\left(\t_{\eta}<\t_{(A\setminus\eta)\cup\del W^{x_1}}\right)
=
(1-\l^{x_1}(A))|A|^{-1}+b_{A}(\s,\eta)
\Eq(P1.lev2.2)
\ee
where, setting $i=1$ if $\dist(\s_2,\eta_2)=1$ and $i=2$ otherwise,
\be
\begin{split}
0\leq b_{A}(\s,\eta)
&\leq
(1-\l^{x_1}(A))|A|^{-1}\OO(N_2^{-i})+\l^{x_1}(A)F_{N_2,d_2}(\dist(\s_2,\eta_2))
\\
&+\frac{1}{1+2^{N_2}\psi_N\gamma^N_1(\xi_1^{x_1})} \OO\left(N_2^{{(d_2+1)}/{2}}\log N_2\right)
\end{split}
\Eq(P1.lev2.2')
\ee
and where $F_{N_2,d_2}$, $d_2=2^{M_2}$, is the function  introduced above Remark \eqv(R1.ls); in particular,
\be
F_{N_2,d_2}(\dist(\s_2,\eta_2))\leq \frac{i}{N_2^i}(1+o(1)).
\Eq(P1.lev2.2'')
\ee

\item{ii)} (Leaving the cylinder set $W^{x_1}$.) For all non empty subset $A\subseteq T^{x_1}$,
and all $\s\in W^{x_1}\setminus A$,
\be
\P^*_{\s}\left(\t_{\del W^{x_1}}<\t_{A}\right)=\l^{x_1}(A)+c_{A}(\s)
\Eq(P1.lev2.5)
\ee
where $c_{A}(\s)=\sum_{\eta\in A}b_{A}(\s,\eta)$.
\end{proposition}

\begin{lemma}
  \TH(L1new.lev2)

On $\wh\O$, for all but a finite number of indices $N_1$ we have 
\bea
\l^{x_1}(A)
&=&
\frac{1}{1+|A|\psi_N\gamma^N_1(\xi_1^{x_1})}
(1+\OO(N_2^{-1})),
\Eq(L1new.lev2.2)
\\
1-\l^{x_1}(A)
&=&
\frac{|A|\psi_N\gamma^N_1(\xi_1^{x_1})}{1+|A|\psi_N\gamma^N_1(\xi_1^{x_1})}
(1+\OO(N_2^{-1})).
\Eq(L1new.lev2.1)
\eea
\end{lemma}
\begin{proof}[Proof of Lemma \thv(L1new.lev2)] By \eqv(01.14) and Lemma \thv(L5new.lev2) and the fact that $c_1^N$ decays exponentially fast to zero (see \eqv(01.2))
\be
\textstyle
0<q^*_N(\xi_1^{x_1})=
\frac{N_1}{N_2}\frac{c_1^N}{\g_1^N(\xi_1^{x_1})}\left(1+\OO\left(c_1^N( \ln N_1)^{2/\a_1^N}\right)\right)\ll 1.
\Eq(L1new.lev2.5)
\ee
Thus by \eqv(P1.lev2.001) 
\be
\textstyle
u(\xi^{x_1}_1)=-q^*(\xi^{x_1}_1)\left(1+\OO\left(c_1^N( \ln N_1)^{2/\a_1^N}\right)\right).
\Eq(L1new.lev2.4)
\ee
The lemma now follows from \eqv(P1.lev2.1), \eqv(1.temp.1) and  \eqv(01.9).
For later use let us observe that $q^*_N(\xi_1^{x_1})$ decays exponentially fast. 
Indeed by \eqv(1.temp.1) and  \eqv(01.9),
$
(N_1/N_2)c_1^N=2^{-N_2}\psi^{-1}_N
$,
whereas, by assumption on $\zeta_N$, there exists $\d>0$ such that $\zeta_N\leq (1-\d) N_2\b^2_{*}/2$, so that 
\be
q^*_N(\xi_1^{x_1})
=\frac{1}{1+2^{N_2}\psi_N\gamma^N_1(\xi_1^{x_1})}
\leq e^{-\d N_2\b^2_{*}/2+ \kappa/\a_1^N}( \ln N_1)^{2/\a_1^N}(1+o(1)).
\Eq(L1new.lev2.3)
\ee
\end{proof}

\begin{remark}
  \TH(R1.lev2)

Note that if $1-\l^{x_1}(A)\gg N^{-1}$ the term $b_{A}(\s,\eta)$  in \eqv(P1.lev2.2) is sub-leading 
and if $\l^{x_1}(A)\gg N^{-1}$  the term $c_{A}(\s)$ in \eqv(P1.lev2.5) is sub-leading. This will still be true when 
$1-\l^{x_1}(A)< const. N^{-1}$, respectively, $\l^{x_1}(A)<  const. N^{-1}$ provided only that $\s_2$ and $\eta_2$ 
are far enough. Indeed  the function $F_{N_2,d_2}(\dist(\s_2,\eta_2))$ is decreasing
and can be made exponentially small in $N_2$ by choosing $\dist(\s_2,\eta_2)$ proportional to $N_2$ (see (10.7) of Lemma 10.1, Appendix A3 of \cite{BG08}) while by \eqv(L1new.lev2.3) the last term in \eqv(P1.lev2.2') always is exponentially small.
\end{remark}

We now turn to ``inter-level motions''.

\begin{proposition}[Inter-level motion]
  \TH(P.inter)

On $\O_1^+\cap\overline\O$, for all large enough $N$, the following holds:
for all $\eta\in T$ and all  $\s\in \VV_N\setminus\overline W$,
setting $i=1$ if $\dist(\s_1,\eta_1)=1$ and $i=2$ otherwise,
\be
\P^*_\s\left(\t_{\eta}<\t_{\overline W\setminus\eta}\right)
\leq \frac{i}{N_1^i}(1+o(1)).
\Eq(P.inter.3)
\ee
\end{proposition}

It is not difficult to deduce from Proposition \thv(P.inter) that:

\begin{corollary}
  \TH(C.inter)
On $\O_1^+\cap\overline\O$, for all large enough $N$, the following holds
for all $\s\in \VV_N\setminus\overline W$:
setting $i=1$ if $\dist(\s_1,\xi^{x}_1)=1$ and $i=2$ otherwise, for all $x_1\in\MM_1$,
\be
\P^*_\s\left(\t_{T^{x_1}}<\t_{\overline W\setminus T^{x_1}}\right)
\leq
|T^{x_1}|\frac{i}{N_1^i}(1+o(1)),
\Eq(C.inter.1)
\ee
\be
\left|\P^*_\s\left(\t_{W^{x_1}\setminus T^{x_1}}
<\t_{\overline W\setminus (W^{x_1}\setminus T^{x_1})}\right)-\frac{1}{M_1}\right|
\leq
\frac{2i}{N_1^i}(1+o(1)).
\Eq(C.inter.3)
\ee
\end{corollary}

\subsection{Proofs of the statements of Section \thv(S2.1) } 
    \TH(S2.3')  

\begin{proof}[Proof of Proposition \thv(P.rem-like) (REM-like estimates)]
Since $\s\notin\overline W$ then $\pi_1\s\notin T_1$ so that using Lemma \thv(L.comp.1) with 
$
A
=\{\xi^{x_1}_1\}
$
and
$
B
=T_1\setminus \xi^{x_1}_1
$
\be
\P^*_{\s}\left(\t_{W^{x_1}}<\t_{\overline W\setminus W^{x_1}}\right)
=
\P_{\s_1}^{\circ,1}\left(\t_{\xi^{x_1}_1}<\t_{T_1\setminus \xi^{x_1}_1}\right).
\Eq(P.rem-like.3)
\ee
Proposition  \thv(P.rem-like) then follows from Theorem 1.4 of \cite{BG08} using the partition
$\L_1(T_1)$ induced  by  $T_1$ through \eqv(lump.12) together with
\eqv(ls.17)  of Lemma \thv(L3.ls) and \eqv(ls.29) of Lemma \thv(C5.ls). 
\end{proof}

\begin{proof}[Proof of Proposition \thv(P1.lev2)]

Throughout the proof we place ourselves on the set of full measure $\wh\O\cap\overline\O\subset \O$ (see \eqv(L5new.lev2.2) and \eqv(ls.10)). 
We first prove Assertion (i). 
Set $d_2=2^{|\pi_2T^{x_1}|}=2^{M_2}$ and let $\L_2(\pi_2T^{x_1})$ be the partition into $d_2$ classes  obtained by taking   $i=2$ and $K=\pi_2T^{x_1}$ in  \eqv(lump.12). (The needed properties of this partition are established in Subsection \thv(S3.3.1)).
Recall that $m_{2,\pi_2T^{x_1}}$ is the function defined in \eqv(lump.7)-\eqv(lump.8) using the partition $\L_2(\pi_2T^{x_1})$ and that
$\overline\E^{2,\pi_2T^{x_1}}$ denotes the expectation w.r.t.~to the law $\overline\P^{2,\pi_2T^{x_1}}$ of the  
$\pi_2T^{x_1}$-lumped chain $\overline I_{2,\pi_2T^{x_1}}$ \eqv(lump.13).
By  \eqv(comp.16') of Lemma \thv(L.comp.4)  (with $i=1$, $j=2$, $A=\{\eta\}$, $B=A\setminus\eta$ and $C(\s_1)=W^{x_1}$)
and Lemma 2.5 of \cite{BG08} we have, setting
$y\equiv m_{2,\pi_2T^{x_1}}(\s_2)$,
$x\equiv m_{2,\pi_2T^{x_1}}(\eta_2)$,
$\AA\equiv m_{2,\pi_2T^{x_1}}(\pi_2A)$
and $u(\xi^{x_1}_1)\equiv\log(1-q^*(\xi^{x_1}_1))$
\be
\P^*_\s\left(\t_{\eta}<\t_{(A\setminus\eta)\cup\del W^{x_1}}\right)
=
\overline\E^{2,\pi_2T^{x_1}}_{y}
\left(
e^{u(\xi^{x_1}_1)\t_{x}}
\1_{\{\t_{x}< \t_{\AA\setminus x}\}}
\right)
\equiv G^y_{x,\AA}(u(\xi^{x_1}_1)).
\Eq(P1.lev2.50)
\ee
We now  use Proposition 7.7 of \cite{BG08} to express the Laplace transform $G^y_{x,\AA\setminus x}(u)$, $u<0$. 
In view of remark \thv(R1.ls) and Lemma \thv(L3.ls), it is easy to check that condition (7.38) of Proposition 7.7  of \cite{BG08} is satisfied: indeed, by (5.15) of Lemma 5.5 of \cite{BG08}, (5.16) of Lemma 5.7 of \cite{BG08}  and  \eqv(ls.17) of Lemma \thv(L3.ls), for any $A\subseteq T^{x_1}$,
\be
\VV^{\circ}_{N_2,d_2}(\AA)
=
\VV_{N_2,d_2}(\pi_2 A)
\leq
\UU_{N_2,d_2}(\pi_2A)
\leq |\pi_2A|2^{-N_2/4}
\leq M_22^{-N_2/4}.
\Eq(P1.lev2.51)
\ee
Next, the quantities ${\underline u}(d_2)$ and $\bar u$ appearing in (7.39) and (7.40) of \cite{BG08} are, here, given by
\be
\textstyle
\bar u^{-1}\equiv\frac{2^{N_2}}{|\AA|}\Bigl(1+\frac{1}{N_2}\Bigr), 
\quad{\underline u}(d_2)\equiv C_{d_2}N_2^{d_2/2+1}(\log N_2)^2(1+o(1))
\Eq(P1.lev2.52)
\ee
where $C_{d_2}=C^2 (2d_2/\pi)^{d_2/2}$ for some constant $0<C<\infty$.
Indeed, using that on $\overline\O$ the partition  $\L_2(\pi_2T^{x_1})$ is very close to an equipartition, we have
by Lemma 2.6 of \cite{BG08} that 
$
\E^{\circ}\t^0_0 =\left({\pi}/{2d_2}\right)^{d_2/2}N_2^{d_2/2}(1+o(1))
$
and by (6.7) of Theorem 6.3 of \cite{BG08} that
\be
\wh\Th(d_2)=CN_2^{{(d_2+1)}/{2}}\log N_2(1+o(1)).
\Eq(P1.lev2.53)
\ee
Finally, let us check that the point $u(\xi^{x_1}_1)$ lies in the segment $(-\rho  {\underline u}(d_2), \bar u)$ for some $0< \rho<1$.
Inserting \eqv(L1new.lev2.3) in \eqv(L1new.lev2.4) and using the bound on ${\underline u}(d_2)$ from \eqv(P1.lev2.52), 
we see that 
\be
\textstyle
0>u(\xi^{x_1}_1)
=-\frac{1}{1+2^{N_2}\psi_N\gamma^N_1(\xi_1^{x_1})}(1+o(1))
\gg-{\underline u}(d_2).
\Eq(P1.lev2.54)
\ee
Thus  $u(\xi^{x_1}_1)$ lies in the segment $(-\rho  {\underline u}(d_2), 0)$ for any $0< \rho<1$, and so, 
we can use assertion (i)-(a) of Proposition 7.7 of \cite{BG08} to express the Laplace transform  $G^y_{x,\AA}(u(\xi^{x_1}_1))$
in \eqv(P1.lev2.50).
Using the notation \eqv(P1.lev2.001)-\eqv(P1.lev2.1) 
and writing $\overline\P\equiv\overline\P^{2,\pi_2T^{x_1}}$, this yields
\be
G^y_{x,\AA\setminus x}(u(\xi^{x_1}_1))
=
\overline\P_y\left(\t_x<\t_{\AA\setminus x}\right)
(1-\l^{x_1}(A))
+
\overline\P_y\left(\t_{x}<\t_{(\AA\setminus x)\cup 0}\right)
\l^{x_1}(A)
+
\RR_{0}
\Eq(P1.lev2.18)
\ee
where
\be
\RR_0=\overline\P_0\left(\t_{x}<\t_{\AA\setminus x}\right)(1-\l^{x_1}(A))
\left[\RR_1+\overline\P_y\left(\t_{0}<\t_{\AA}\right)\RR_2\right]+\RR_3
\Eq(P1.lev2.19)
\ee
and where, by \eqv(P1.lev2.51)-\eqv(P1.lev2.54), the remainder terms are given by
\be
\begin{split}
\RR_i&=\frac{1}{1+2^{N_2}\psi_N\gamma^N_1(\xi_1^{x_1})} \OO\left(N_2^{{(d_2+1)}/{2}}\log N_2\right),\quad  i=1,3,\\
\RR_2 &=\OO\left(\max\left\{\frac{1}{N^2_2}\l^{x_1}(A),
\frac{1}{1+2^{N_2}\psi_N\gamma^N_1(\xi_1^{x_1})} N_2^{d_2/2+1}(\log N_2)^2\right\}\right).
\end{split}
\Eq(P1.lev2.20)
\ee
In particular, in view of \eqv(L1new.lev2.3) and the fact that $0<\l^{x_1}(A)\leq 1$, 
$\RR_2 =\OO(N_2^{-2})$  and $\RR_i$, i=1,3, decay exponentially fast.
It now remains to estimate the four probabilities that enter the above expressions. 
We trivially bound
$
0\leq \overline\P_y\left(\t_{0}<\t_{\AA}\right)\leq 1
$.
To deal with the prefactors of $(1-\l^{x_1}(A))$ in \eqv(P1.lev2.18) and \eqv(P1.lev2.19), namely, the harmonic measures stating from $0$ and from $y$ we use,
respectively, Lemma 4.4 of \cite{BG08} and Theorem 4.5 of \cite{BG08} combined with Lemma 4.2 of \cite{BG08}.
Together with Lemma \thv(L3.ls) and Lemma \thv(C5.ls), they yield,
\be
\overline\P_0\left(\t_{x}<\t_{\AA\setminus x}\right)= |A|^{-1}\left(1+\OO(N_2^{-2})\right)
\Eq(P1.lev2.26)
\ee
\be
0\leq
\overline\P_y\left(\t_{x}<\t_{\AA\setminus x}\right)-|A|^{-1}\left(1+\OO(N_2^{-2})\right)
\leq F_{N_2,d_2}(\dist(\s_2,\eta_2))
\Eq(P1.lev2.27)
\ee
where $F_{N_2,d_2}$ is the function  introduced in Definition 3.3 of \cite{BG08} and studied in Appendix 11 of \cite{BG08}. It is a decreasing function that decays polyniomally fast in $N_2$ for small distances and exponentially fast in $N_2$ for distances proportional to $N_2$. In particular, setting  $i=1$ if $\dist(\eta_2,\s_2)=1$ and $i=2$ otherwise, we have the rough bound
\be
F_{N_2,d_2}(\dist(\s_2,\eta_2))\leq \frac{i}{N_2^i}(1+o(1)).
\ee
It remains to deal with the pre-factors of $\l^{x_1}(A)$ in \eqv(P1.lev2.18). For an upper bound, write
\be
0\leq
\overline\P_y\left(\t_{x}<\t_{(\AA\setminus x)\cup 0}\right)
\leq
\overline\P_y\left(\t_{x}<\t_{ 0}\right)
\leq
F_{N_2,d_2}(\dist(\s_2,\eta_2))
\Eq(P1.lev2.24)
\ee
where the last inequality is Theorem 3.2 of \cite{BG08}.
Inserting our estimates in \eqv(P1.lev2.18) and combining the result with \eqv(P1.lev2.50), we arrive at
\be
\P^*_\s\left(\t_{\eta}<\t_{(A\setminus\eta)\cup\del W^{x_1}}\right)
=
(1-\l^{x_1}(A|)|A|^{-1}+b_{A}(\s_2,\eta_2)
\Eq(P1.lev2.2*)
\ee
where
\be
\begin{split}
0\leq b_{A}(\s_2,\eta_2)&\leq
(1-\l^{x_1}(A))|A|^{-1}\Bigl(F_{N_2,d_2}(\dist(\s_2,\eta_2))+\OO(N_2^{-2})\Bigr)\\
&+\l^{x_1}(A)F_{N_2,d_2}(\dist(\s_2,\eta_2))+\RR_3
\end{split}
\Eq(P1.lev2.2**)
\ee
which readily yields \eqv(P1.lev2.2).
The second assertion of Proposition \thv(P1.lev2) is a direct consequence of the first, observing that
\be
\P^*_\s\left(\t_{\del W^{x_1}}<\t_{A}\right)
=1-\sum_{\eta\in A}\P^*_\s\left(\t_{\eta}<\t_{(A\setminus\eta)\cup\del W^{x_1}}\right)
=1-\sum_{x\in A}G^y_{x,\AA\setminus x}(u(\xi^{x_1}_1))
\Eq(P1.lev2.33)
\ee
where the last equality is  \eqv(P1.lev2.50), and use \eqv(P1.lev2.2*)-\eqv(P1.lev2.2**).
The proof of  Proposition \thv(P1.lev2) is now complete. 
 \end{proof}

\begin{proof}[Proof of Proposition \thv(P.inter)]
With the definition \eqv(P.inter.12), $\eta=\eta_1\eta_2\in T$ can be written as  $\eta=C(\eta_2)\cap C(\eta_1)$ 
where $C(\eta_1)=W^{x_1}$ for some $x_1\in\MM_1$. Thus, for the event 
$\{\t_{\eta}<\t_{\overline W\setminus\eta}\}$ to take place, the chain must reach $\eta$ ``from within'' 
the set $C(\eta_2)$, without of course having ever visited $\overline W$. Building on this observation 
we begin by establishing an priori upper bound on the probability \eqv(P.inter.3)
that is valid for all starting points $\s$ in $\VV_N\setminus\overline W$.

\begin{lemma}
  \TH(L1.inter) 
The following holds on $\overline\O$ (see \eqv(ls.10)) for all large enough $N$: for all $\eta\in T$ and all
$\s\in \VV_N\setminus\overline W$
\be
\P^*_\s\left(\t_{\eta}<\t_{\overline W\setminus\eta}\right)
\leq
\frac{ q^*_N(\s_1)}{M_1}\left(1+\OO(M_1/N_1)\right).
\Eq(L1.inter.1)
\ee
\end{lemma}

\begin{proof} Using the renewal identity  (see e.g. Corollary 1.9 in \cite{BEGK1})
\be
\P^*_\s\left(\t_{\eta}<\t_{\overline W\setminus\eta}\right)
\leq
\frac{\P^*_\s\left(\t_{\eta}<\t_{(\overline W\setminus\eta)\cup\s}\right)}
{\P^*_\s\left(\t_{\overline W}<\t_{\s}\right)}.
\Eq(L1.inter.2)
\ee
To deal with the numerator we use reversibility and \eqv(01.21) to write 
\be
\P^*_\s\left(\t_{\eta}<\t_{(\overline W\setminus\eta)\cup\s}\right)
=\frac{G^{*}_{\b,N}(\eta)}{G^{*}_{\b,N}(\s)}
\P^*_\eta\left(\t_{\s}<\t_{\overline W}\right)
=\frac{q^*_N(\s_1)}{q^*_N(\eta_1)}
\P^*_\eta\left(\t_{\s}<\t_{\overline W}\right).
\Eq(L1.inter.3)
\ee
Now for the event
$\{\t_{\s}<\t_{\overline W}\}$ to take place the chain, starting in $\eta$,
must exit $\eta$ through the set
$C(\eta_2)$. Thus, setting
$
\del_1\eta\equiv C(\eta_2)\cap\del\eta
=\{\s\in\VV_N\mid \s_1\sim\eta_1, \s_2=\eta_2\}
$
 (recall \eqv(comp.6))
\bea
\P^*_\eta\left(\t_{\s}<\t_{\overline W}\right)
&=&
\P^*_\eta\left(\t_{\del_1\eta}<\t_{\s}<\t_{\overline W}\right)
\Eq(L1.inter.5'')
\\
&=&
\textstyle
\sum_{\s'\in\del_1\eta}
\P^*_\eta\left(\t_{\s'}<\t_{(\del_1\eta\setminus\s')\cup\overline W\cup\s}\right)
\P^*_{\s'}\left(\t_{\s}<\t_{\overline W}\right).
\Eq(L1.inter.5')
\eea
To bound the last probability in the right-hand side of \eqv(L1.inter.5') we simply observe that
on $\overline\O$, for all large enough $N$, 
\be
\P^*_{\s'}\left(\t_{\s}<\t_{\overline W}\right)
\leq \P^*_{\s'}\left(\t_{C(\s_1)}<\t_{\overline W}\right)
=
\frac{1}{M_1+1}\left(1+\OO(M_1/N_1)\right)
\Eq(P.inter.1)
\ee
where the last equality is proved just as Proposition \thv(P.rem-like), 
namely, using first Lemma \thv(L.comp.1) with 
$
A
=\{\s_1\}
$,
$
B
=T_1
$
and $\s'_1\neq\s_1$ to write that
$
\P^*_{\s'}\left(\t_{C(\s_1)}<\t_{\overline W}\right)
=\P^{\circ,1}_{\s'_1}\left(\t_{\s_1}<\t_{T_1}\right)
$,
and using next Theorem 1.4 of \cite{BG08}  with the partition $\L_1(T_1\cup\s_1)$ 
induced  by  $T_1\cup\s_1$  (that is to say, the Top and a non random point) 
together with  \eqv(ls.33) and \eqv(ls.36) of Lemma  \thv(L6.ls) and \eqv(ls.29) of Lemma \thv(C5.ls),
 under the assumptions therein
(namely, on $\overline\O$, for all large enough $N$)
which are assumed from now on to be verified. 
Plugging \eqv(P.inter.1) in \eqv(L1.inter.5'), we get
\be
\P^*_\eta\left(\t_{\s}<\t_{\overline W}\right)
\leq 
\frac{1}{M_1+1} \left(1+\OO(M_1/N_1)\right)\P^*_\eta\left(\t_{\del_1\eta}<\t_{\overline W\cup\s}\right)
\Eq(L1.inter.5)
\ee
where
$
\textstyle
\P^*_\eta\left(\t_{\del_1\eta}<\t_{\overline W\cup\s}\right)
\leq
\P^*_\eta\left(\t_{\del_1\eta}<\t_{W^{x_1}}\right)
=\sum_{\s''\in\del_1\eta}p^*_2(\eta,\s'')
=q^*_N(\eta_1)
$.
Combined with \eqv(L1.inter.3), this finally gives
\be
\P^*_\s\left(\t_{\eta}<\t_{(\overline W\setminus\eta)\cup\s}\right)
\leq 
\frac{q^*_N(\s_1)}{M_1+1}\left(1+\OO(M_1/N_1)\right).
\Eq(L1.inter.7)
\ee
It now remains to bound the denominator of \eqv(L1.inter.2). For this we simply decompose on the first step of the jump chain,
\be
\textstyle
\P^*_\s\left(\t_{\overline W}<\t_{\s}\right)
=
\sum_{\s'\sim\s : \s'\in\overline W}p^{*}_N(\s,\s')
+\sum_{\s'\sim\s : \s'\notin\overline W}p^{*}_N(\s,\s')\P^*_{\s'}\left(\t_{\overline W}<\t_{\s}\right),
\Eq(L1.inter.8)
\ee
and observe that
$
\P^*_{\s'}\left(\t_{\overline W}<\t_{\s}\right)
\geq 
\P^*_{\s'}\left(\t_{\overline W}<\t_{C(\s_1)}\right)
=1-\frac{1}{M_1+1}\left(1+\OO(M_1/N_1)\right)
$
where the last equality follows from \eqv(P.inter.1). Since $\sum_{\s'\sim\s}p^{*}_N(\s,\s')=1$ we get
\be
\P^*_\s\left(\t_{\overline W}<\t_{\s}\right)
\geq 1-\frac{1}{M_1+1}\left(1+\OO(M_1/N_1)\right).
\Eq(L1.inter.9)
\ee
Inserting \eqv(L1.inter.7) and \eqv(L1.inter.9) in  \eqv(L1.inter.2) yields \eqv(L1.inter.1) and proves the lemma. 
\end{proof}

Consider now the set $\VV^+_N\equiv\VV^+_{N_1}\times\VV_{N_2}$ where $\VV^+_{N_1}$ 
is the set \eqv(lsb.1) obtained by choosing
\be
\e_{N_1}=4\frac{\log{N_1}}{\beta\sqrt{a_1N}}.
\Eq(P.inter.13bis)
\ee
\begin{corollary}
  \TH(C1.inter)
Under the assumptions of Lemma \thv(L1.inter), for all $\eta\in T$ and all $\s\in \VV^+_N\setminus\overline W$, 
\be
\P^*_\s\left(\t_{\eta}<\t_{\overline W\setminus\eta}\right)
\leq 
\frac{1}{M_1N_1^3}.
\Eq(C1.inter.1)
\ee
\end{corollary}

\begin{proof} By \eqv(01.14), \eqv(01.3) and \eqv(P.inter.13bis), 
$
q^*_N(\s_1)\leq N_1^{-3} \text{ for all } \s\in\VV^+_N
$.
Inserting this in Lemma \thv(L1.inter) yields \eqv(C1.inter.1).
\end{proof}

This implies that the bound \eqv(P.inter.3) holds true for all $\s\in \VV^+_N\setminus\overline W$.
To extend this result to the entire set $\VV_N\setminus\overline W$,
observe that  $\overline W\subset\VV^+_N$
(see the definition  \eqv(1.20)-\eqv(1.24) of  $\overline W$) and decompose the
probability in \eqv(P.inter.3) according to whether the jump chain visits 
$\VV^+_N\setminus\overline W$ before $\eta$ or not, namely,  for $\s\in\VV_N\setminus\VV^+_N$  write
\be
\P^*_\s\left(\t_{\eta}<\t_{\overline W\setminus\eta}\right)
=\P^*_\s\left(\t_{\VV^+_N\setminus\overline W}<
\t_{\eta}<\t_{\overline W\setminus\eta}\right)
+
\P^*_\s\left(\t_{\eta}<\t_{(\overline W\setminus\eta)\cup(\VV^+_N\setminus\overline W)}\right)
\Eq(P.inter.14)
\ee
Call $Q_1$ and $Q_2$, respectively, the first and second probabilities  in the r.h.s.~of \eqv(P.inter.14). Then
\bea
Q_1&=&
\sum_{\s'\in\VV^+_N\setminus\overline W}
\P^*_\s\left(\t_{\s'}<\t_{\VV^+_N\setminus\s'}\right)
\P^*_{\s'}\left(\t_{\eta}<\t_{\overline W\setminus\eta}\right)
\Eq(P.inter.16)
\\
&\leq&\frac{1}{M_1N_1^3}\P^*_\s\left(\t_{\VV^+_N\setminus\overline W}<\t_{\overline W}\right) 
\leq\frac{1}{M_1N_1^3}
\Eq(P.inter.16'').
\eea
where first inequality in \eqv(P.inter.16'') is Corollary  \thv(C1.inter).
To bound $Q_2$ note that given any set $A_{N_1}\subset\VV^+_{N_1}$, 
$
A_N\equiv\cup_{\s_1\in A_{N_1}}C(\s_1)\subset\cup_{\s_1\in\VV^+_{N_1}}C(\s_1)
$,
so that for all $\s\in\VV_N\setminus\VV^+_N$
\be
Q_2=
\P^*_\s\left(\t_{\eta}<\t_{\VV^+_N\setminus\eta}\right)
\leq
\P^*_\s\left(\t_{C(\eta_1)}<\t_{A_N\setminus C(\eta_1)}\right)
=
\P^{\circ,1}_{\s_1}\left(\t_{\eta_1}<\t_{A_{N_1}\setminus\eta_1}\right)
\Eq(P.inter.18)
\ee
where the last equality is Lemma \thv(L.comp.1) applied with 
$
A
=\{\eta_1\}
$,
$
B
=A_{N_1}\subset\VV^+_{N_1}
$
and $\s_1\notin \VV^+_{N_1}$.  The point now is to find a big enough set $A_{N_1}$ 
and a compatible partition 
by which Theorem 1.4 of \cite{BG08} yields a suitably small estimate. The set
$K^+_{N_1}$ introduced in Section \thv(S3.3.3) was tailored to do precisely this.
We thus make the following choices:
take any $d$ such that $d/N=o(1)$ and $2^{-d}=o(N^{-2})$,
let $K_{N_1}$ be constructed as in Lemma \thv(L4.lsb) and take
$
A_{N_1}=K_{N_1}^+\equiv\VV^+_{N_1}\cap K_{N_1}
$
as in \eqv(lsb.19). By Remark \thv(R1.lsb), $T_1\subset K_{N_1}^+$.
Moreover, by Lemma \thv(L1.ls)  and Lemma \thv(L5.lsb),  on $\O_1^+\cap\overline\O_1(T_1)$,
the assumptions of Lemma \thv(L2.lsb) applied to $K_{N_1}^+$ are verified for all large enough $N_1$, 
whereas by Lemma \thv(L5.lsb), on $\overline\O_1^+$,
$
|K_{N_1}^+|= 2^{d}(1+o(1))
$.
Thus, by \eqv(lsb.8) of Lemma \thv(L2.lsb)  and \eqv(lsb.14) of Lemma \thv(L3.lsb), 
on $\O_1^+\cap\overline\O_1(T_1)$, for all but a finite number of indices $N_1$, Theorem 1.4 of \cite{BG08} yields
\be
\P^{\circ,1}_{\s_1}\left(\t_{\eta_1}<\t_{K_{N_1}^+\setminus\eta_1}\right)
\leq 
2^{-d}(1+o(1)) +({i}/{N_1^i})(1+o(1)),
\Eq(P.inter.22)
\ee
and by our assumption on $d$, inserting \eqv(P.inter.22) in \eqv(P.inter.18),
\be
0\leq Q_2\leq  ({i}/{N_1^i})(1+o(1)).
\Eq(P.inter.23)
\ee
Plugging \eqv(P.inter.16) and \eqv(P.inter.23) in \eqv(P.inter.14) we obtain that,
on $\O_1^+\cap\overline\O_1(T_1)$ and for all $N_1$ large enough,
for all $\s\in\VV_N\setminus\VV^+_N$
\be
\P^*_\s\left(\t_{\eta}<\t_{\overline W\setminus\eta}\right)
\leq
 ({i}/{N_1^i})(1+o(1)).
\Eq(P.inter.24)
\ee
Combining  \eqv(P.inter.24)  and \eqv(C1.inter.1) yields the claim of  Proposition \thv(P.inter).
\end{proof}

\begin{proof}[Proof of Corollary \thv(C.inter)]
The bound \eqv(C.inter.1) follows from Proposition \thv(P.inter) and the identity
\be
\textstyle
\P^*_\s\left(\t_{T^{x_1}}<\t_{\overline W\setminus T^{x_1}}\right)
=
\sum_{\eta\in T^{x_1}}
\P^*_\s\left(\t_{\eta}<\t_{\overline W\setminus\eta}\right).
\Eq(C.inter.4)
\ee
To prove  \eqv(C.inter.3) note that
\be
\P^*_\s\left(\t_{W^{x_1}}<\t_{\overline W\setminus W^{x_1}}\right)
=
\P^*_\s\left(\t_{T^{x_1}}<\t_{\overline W\setminus T^{x_1}}\right)
+
\P^*_\s\left(\t_{W^{x_1}\setminus T^{x_1}}<\t_{\overline W\setminus (W^{x_1}\setminus T^{x_1})}\right)
\Eq(C.inter.5)
\ee
which implies that
\bea
&&
\left|\P^*_\s\left(\t_{W^{x_1}\setminus T^{x_1}}
<\t_{\overline W\setminus (W^{x_1}\setminus T^{x_1})}\right)-\frac{1}{M_1}\right|
\\
&
\leq
&
\left|\P^*_\s\left(\t_{W^{x_1}}<\t_{\overline W\setminus W^{x_1}}\right)
-\frac{1}{M_1}\right|
+\P^*_\s\left(\t_{T^{x_1}}<\t_{\overline W\setminus T^{x_1}}\right).
\Eq(C.inter.6)
\eea
Using \eqv(P.rem-like.1) of Proposition \thv(P.rem-like) to bound the first probability in \eqv(C.inter.6)
and \eqv(C.inter.1) to bound the second, we get
\be
\left|\P^*_\s\left(\t_{W^{x_1}\setminus T^{x_1}}
<\t_{\overline W\setminus (W^{x_1}\setminus T^{x_1})}\right)-\frac{1}{M_1}\right|
\leq
\frac{2i}{N_1^i}(1+o(1))
\Eq(C.inter.7)
\ee
where $i=1$ if $\dist(\s_1,\xi^{x}_1)=1$ and $i=2$ otherwise. But this is \eqv(C.inter.3). 
\end{proof}

\section{Entrance law. Proofs}
\TH(S4.4)  
\subsection{Proof of Proposition  \thv(P.top)}
\TH(S4.41)  

The set $\wt\O$ of Propositions \thv(P.top) is chosen to be
$
\wt\O=\overline\O\cap\O_1^+\cap\wh\O
$
where $\overline\O$, $\O_1^+$, and $\wh\O $ are defined, respectively,
in \eqv(ls.10), \eqv(lsb.20) and \eqv(L5new.lev2.2).
By Lemma \thv(L1.ls), Lemma \thv(L5.lsb) and Lemma \thv(L5new.lev2), 
$\PP(\wt\O)=1$. From now on we will assume that $\o\in\wt\O$.
Lemma \thv(comp.lem1) will be frequently used without making explicit mention of it.

\begin{proof}[Proof of Assertion (i) of Proposition \thv(P.top)]

We first work out general expressions, valid at all temperature, that relate the entrance
probabilities $\P_{\s}\left(\t_{\eta}<\t_{T\setminus\eta}\right)$,
$\s\notin T$, $\eta\in T$, to the basic REM-like, level-$2$ and
inter-level probabilities estimated in
Propositions \thv(P.rem-like),  \thv(P1.lev2) and \thv(P.inter).
To shorten the notations we write,
given $x_1\in\MM_1$ and $\eta\in T^{x_1}$
\be
\P_{\s}\left(\t_{\eta}<\t_{T\setminus\eta}\right)=
 \begin{cases}
P_{\eta}(\s) &\mbox{if }  \s\in W^{x_1}\setminus  T^{x_1},\\
Q_{\eta}(\s) &\mbox{if } \s\in W^{x'_1}\setminus  T^{x'_1}
\hbox{ \rm for some } x'_1\in\MM_1\setminus x_1,\\
R_{\eta}(\s)  &\mbox{if } \s\in\SS_N\setminus\overline W. \\
\end{cases} 
\Eq(P.top.12)
\ee
Note that the probabilities \eqv(P.top.5), \eqv(P.top.6) and \eqv(P.top.7) 
are of the form, respectively, $P_{\eta}(\s)$, $Q_{\eta}(\s)$, and $R_{\eta}(\s)$. As the next lemma shows,
both $P_{\eta}(\s)$ and $Q_{\eta}(\s)$ can be expressed as functions of $R_{\eta}(\s)$ whereas, using a 
renewal kind of argument, $R_{\eta}(\s)$ itself is the solution of a linear system of equations.

\begin{lemma}
   \TH(L.top.6)
Given $x_1\in\MM_1$ and $\eta\in T^{x_1}$ define, for  $\s\in\VV_N\setminus\overline W$,
\be
b_{\eta}(\s)=
\P_{\s}\left(\t_{\eta}<\t_{\overline W\setminus\eta}\right)
+\sum_{\s'\in W^{x_1}\setminus T^{x_1}}
\P_{\s}\left(\t_{\s'}<\t_{\overline W\setminus\s'}\right)
\P_{\s'}\left(\t_{\eta}<\t_{(T^{x_1}\setminus\eta)\cup\del W^{x_1}}\right)
\Eq(P.top.29)
\ee
and, for $\s\in\VV_N\setminus\overline W$ and  $\s''\in\del\overline W$,
\be
a(\s,\s'')=
\sum_{x'_1\in\MM_1}
\sum_{\s'\in W^{x'_1}\setminus T^{x'_1}}
\P_{\s}\left(\t_{\s'}<\t_{\overline W\setminus\s'}\right)
\P_{\s'}\left(\t_{\s''}<\t_{(\del W^{x'_1}\setminus \s'')\cup T^{x'_1}}\right).
\Eq(P.top.70)
\ee
If $A$ denotes the square matrix $A=(a(\s,\s''))_{\s,\s''\in\del\overline W}$ and 
$b_{\eta}$ the vector $b_{\eta}=(b_{\eta}(\s))_{\s\in\del\overline W}$,
then the vector $R_{\eta}\equiv(R_{\eta}(\s))_{\s\in\del\overline W}$ obeys
\be
R_{\eta}=b_{\eta}+AR_{\eta}.
\Eq(P.top.72)
\ee
Moreover, if $R^*_{\eta}=(R^*_{\eta}(\s))_{\s\in\del\overline W}$ solves the linear
system \eqv(P.top.72), then
\item{(i)} for all $\s\in\SS_N\setminus\overline W$,
\be
R_{\eta}(\s)= b_{\eta}(\s)+\sum_{\s''\in\del\overline W}a(\s,\s'')R^*_{\eta}(\s''),
\Eq(P.top.26')
\ee

\item{(ii)} for all $\s\in W^{x_1}\setminus  T^{x_1}$,
\be
P_{\eta}(\s)=
\P_{\s}\left(\t_{\eta}<\t_{(T^{x_1}\setminus\eta)\cup\del W^{x_1}}\right)
+
\sum_{\s'\in\del W^{x_1}}
\P_{\s}\left(\t_{\s'}<\t_{(\del W^{x_1}\setminus \s')\cup T^{x_1}}\right)
R^*_{\eta}(\s'),
\Eq(P.top.16')
\ee

\item{(iii)} for all $\s\in W^{x'_1}\setminus T^{x'_1}$ and all $x'_1\in\MM_1\setminus x_1$,
\be
Q_{\eta}(\s)=\sum_{\s'\in\del W^{x'_1}}
\P_{\s}\left(\t_{\s'}<\t_{(\del W^{x'_1}\setminus \s')\cup T^{x'_1}}\right)
R^*_{\eta}(\s').
\Eq(P.top.17')
\ee
\end{lemma}

\begin{proof} 
Let us first consider $P_{\eta}(\s)$. Decomposing the event
$
\{\t_{\eta}<\t_{T\setminus\eta}\}
$
according to whether, starting in $\s$, the chain visits $\eta$ before visiting
the boundary $\del W^{x_1}$ or not, we get:
\bea
P_{\eta}(\s)
&=&
\P_{\s}\left(\{\t_{\eta}<\t_{T\setminus\eta}\}
\cap\{\t_{\eta}<\t_{\del W^{x_1}}\}
\right)
+
\P_{\s}\left(\{\t_{\eta}<\t_{T\setminus\eta}\}
\cap\{\t_{\del W^{x_1}}<\t_{\eta}\}
\right)
\nonumber
\\
&=&
\P_{\s}\left(\t_{\eta}<\t_{(T^{x_1}\setminus\eta)\cup\del W^{x_1}}\right)
+
\P_{\s}\left(\t_{\del W^{x_1}}<\t_{\eta}<\t_{T\setminus\eta}\right)
\nonumber
\\
&=&
\P_{\s}\left(\t_{\eta}<\t_{(T^{x_1}\setminus\eta)\cup\del W^{x_1}}\right)
+
\sum_{\s'\in\del W^{x_1}}
\P_{\s}\left(\t_{\s'}<\t_{(\del W^{x_1}\setminus \s')\cup T^{x_1}}\right)
R_{\eta}(\s')\quad\quad
\Eq(P.top.16)
\eea
where we used \eqv(P.top.12) in the last line.
Proceeding in the same way with $Q_{\eta}(\s)$ yields
\be
Q_{\eta}(\s)=\sum_{\s'\in\del W^{x'_1}}
\P_{\s}\left(\t_{\s'}<\t_{(\del W^{x'_1}\setminus \s')\cup T^{x'_1}}\right)
R_{\eta}(\s').
\Eq(P.top.17)
\ee
We now focus on $R_{\eta}(\s)$.
Clearly $\P_{\s}(\t_{\overline W}<\infty)=1$
and by definition of $\overline W$
\be
\textstyle
\{\t_{\overline W}<\infty\}=
\bigcup_{x'_1\in\MM_1}
\{\t_{W^{x'_1}}<\t_{\overline W\setminus W^{x'_1}}\}.
\Eq(P.top.20)
\ee
Hence
\be
R_{\eta}(\s)=
\sum_{x'_1\in\MM_1}\P_{\s}\left(\{\t_{\eta}<\t_{T\setminus\eta}\}
\cap \{\t_{W^{x'_1}}<\t_{\overline W\setminus W^{x'_1}}\}
\right).
\Eq(P.top.21)
\ee
Setting 
$
\EE^{x'_1}\equiv
\{\t_{\eta}<\t_{T\setminus\eta}\}
\cap \{\t_{W^{x'_1}}<\t_{\overline W\setminus W^{x'_1}}\}
$
and observing that for $x'_1=x_1$
\be
\textstyle
\EE^{x_1}
=\Bigl(\bigcup_{\s'\in W^{x_1}\setminus T^{x_1}}
\{\t_{\s'}<\t_{\overline W\setminus\s'}\}
\cap
\{\t_{\eta}<\t_{T\setminus\eta}\}
\Bigr)\cup\{\t_{\eta}<\t_{\overline W\setminus\eta}\}   
\Eq(P.top.22)
\ee
whereas  for all $x'_1\in\MM_1\setminus x_1$
\be
\textstyle
\EE^{x'_1}
= \bigcup_{\s'\in W^{x'_1}\setminus T^{x'_1}}
\{\t_{\s'}<\t_{\overline W\setminus\s'}\}
\cap
\{\t_{\eta}<\t_{T\setminus\eta}\},
\Eq(P.top.23)
\ee
\eqv(P.top.21) becomes
\bea
R_{\eta}(\s)&=&
\P_{\s}\left(\t_{\eta}<\t_{\overline W\setminus\eta}\right)
+\sum_{\s'\in W^{x_1}\setminus T^{x_1}}
\P_{\s}\left(\t_{\s'}<\t_{\overline W\setminus\s'}\right) P_{\eta}(\s')
\\
&+&\sum_{x'_1\in\MM_1\setminus x_1}
\sum_{\s'\in W^{x'_1}\setminus T^{x'_1}}
\P_{\s}\left(\t_{\s'}<\t_{\overline W\setminus\s'}\right) Q_{\eta}(\s').
\Eq(P.top.24)
\eea
Plugging the expressions \eqv(P.top.16) and \eqv(P.top.17) of $P_{\eta}(\s)$ and $Q_{\eta}(\s)$
in \eqv(P.top.24) readily yield that for $b_{\eta}(\s)$ and $a(\s,\s'')$ as  in \eqv(P.top.29) and \eqv(P.top.70),
$R_{\eta}(\s)$ obeys, for all $\s\in\SS_N\setminus\overline W$,
\be
R_{\eta}(\s)= b_{\eta}(\s)+\sum_{\s''\in\del W^{x'_1}}a(\s,\s'')R_{\eta}(\s'').
\Eq(P.top.26)
\ee
The restriction of this last relation to $\s\in\del\overline W$  enables us to see
the vector  $(R_{\eta}(\s))_{\s\in\del\overline W}$ as solution
of the linear system of equations \eqv(P.top.72). This observation together with \eqv(P.top.26),
\eqv(P.top.16) and \eqv(P.top.17) prove, respectively, \eqv(P.top.26'),
\eqv(P.top.16') and \eqv(P.top.17'). Lemma \thv(L.top.6) is proven.\end{proof}

\begin{lemma}
\TH(L.top.2)
Under the assumptions and with the notation of Proposition \thv(P.top), the linear system 
\eqv(P.top.72) has a unique solution, $R^*_{\eta}=(R^*_{\eta}(\s))_{\s\in\del\overline W}$, that obeys
\be
R^*_{\eta}(\s)=
\frac{\nu_1(x_1)}{M_2}\left(1+\varepsilon_{N}\right)\,\, \forall  \s\in\del\overline W.
\Eq(P.top.28)
\ee
\end{lemma}
The proof of Lemma \thv(L.top.2) makes use of the following two lemmata.

\begin{lemma}
   \TH(L.top.2')
The matrix $A$ has the following properties:
for each $\s\in\del\overline W$
\be
0\leq\sum_{\s''\in\del \overline W}a(\s,\s'')=1-\sum_{\eta\in T}b_{\eta}(\s)\leq 1
\Eq(P.top.73)
\ee
\end{lemma}
\begin{proof} 
Summing both sides of \eqv(P.top.72) over $\eta\in T$ and using that by \eqv(P.top.12) and \eqv(P.top.29),
\be
\sum_{\eta\in T}R_{\eta}(\s)=1
\,\,  \text{for all } \s\in\SS_N\setminus\overline W
\Eq(P.top.81)
\ee
yields the equality of \eqv(P.top.73). The first and final upper and lower bounds  
simply reflect the fact that $A$ is a positive matrix and $b_{\eta}$ a positive vector.
\end{proof} 

\begin{lemma}
\TH(L.top.2'')
A necessary and sufficient condition for a solution to  \eqv(P.top.72) to exist is that 
\be
\textstyle
\min_{\s\in\del\overline W}\sum_{\eta\in T}b_{\eta}(\s)>0.
\Eq(P.top.75)
\ee
In this case the solution is unique, positive and given by
$
R_{\eta}=(I-A)^{-1}b_{\eta}
$,
where $I$ denotes the identity matrix and 
$
(I-A)^{-1}=\sum_{k=1}^{\infty}A^k
$
exists.
\end{lemma}
\begin{proof}
Denote by $\rho(A)$ the Perron-Frobenius eigenvalue of $A$. By \eqv(P.top.73) 
and the standard min and max row-sum bounds on the Perron-Frobenius eigenvalue 
of primitive matrices (see \cite{S}, p.~8 Corollary 1)
\be
\textstyle
0\leq
1-\max_{\s\in\del\overline W}\sum_{\eta\in T}b_{\eta}(\s)
\leq\rho(A)\leq
1-\min_{\s\in\del\overline W}\sum_{\eta\in T}b_{\eta}(\s)
\leq 1.
\Eq(P.top.74)
\ee
The claim of the lemma now follows from \eqv(P.top.73), \eqv(P.top.74) and 
Theorem  2.1 p.~30 of  \cite{S} (see also Corollary 3 p.~31).
\end{proof} 

We are now ready to prove Lemma \thv(L.top.2).

\begin{proof}[Proof of Lemma \thv(L.top.2)] Let us establish that if
$
1-\l^{x_1}(T^{x_1})\gg \OO(N^{-1})
$
then, there exists a  constant $0<c<\infty$ (that depends on $a_1,a_2$) such that, 
for all $x_1\in\MM_1$, all $\eta\in T^{x_1}$ and  all  $\s\in\VV_N\setminus\overline W$
\be
v({x_1})\leq b_{\eta}(\s)\leq v({x_1})+cN^{-1}\,\, \text{ where }\,\, v({x_1})\equiv\frac{1}{M_1M_2}(1-\l^{x_1}(T^{x_1})).
\Eq(L.top.new1)
\ee
Inserting  \eqv(P1.lev2.2) of Proposition \thv(P1.lev2) in \eqv(P.top.29) and using  \eqv(C.inter.3) of Corollary \thv(C.inter) yields that
for all $\s\in\VV_N\setminus\overline W$ and all $\eta\in T^{x_1}$,
$
b_{\eta}(\s)=v({x_1})(1+
N_1^{-1}(1+o(1)))+\d b_{\eta}(\s)
$
where $v({x_1})$ is as in  \eqv(L.top.new1) and 
\be
\d b_{\eta}(\s)
\equiv
\P_{\s}\left(\t_{\eta}<\t_{\overline W\setminus\eta}\right)
+\sum_{\s'\in W^{x_1}\setminus T^{x_1}}
\P_{\s}\left(\t_{\s'}<\t_{\overline W\setminus\s'}\right)
b_{A}(\s,\eta).
\Eq(L.top.new2)
\ee
By \eqv(P1.lev2.2'), \eqv(P1.lev2.2'') and \eqv(L1new.lev2.3),
$
0\leq b_{T^{x_1}}(\s,\eta)\leq N_2^{-1}(1+o(1))
$
for all for all $\s\in\VV_N\setminus\overline W$ and all $\eta\in T^{x_1}$.
Inserting this rough bound in \eqv(L.top.new2) and using again \eqv(C.inter.3) to bound 
the resulting sum, the first term in \eqv(L.top.new1) being bounded in \eqv(P.inter.3), we get
\be
0\leq \max_{\s\in\del \overline W, \eta\in T^{x_1}}\d b_{\eta}(\s)
\leq 
 \frac{1}{N_1}(1+o(1))+\frac{1}{M_1N_2}(1+({2M_1}/{N_1})(1+o(1)).
\Eq(L.top.new3)
\ee
This proves the claim \eqv(L.top.new1). Eq.~\eqv(L.top.new1) in particular implies 
that the solution $R^*_{\eta}$ of the linear system  \eqv(P.top.72) obeys
\be
v({x_1})(I-A)^{-1}\1\leq R^*_{\eta}\leq (v({x_1})+cN^{-1})(I-A)^{-1}\1
\ee
where  $\1$  is the vector with all components equal to one and
where the inequalities hold component wise, for each $R^*_{\eta}(\s)$, $\s\in\del \overline W$.
Now, by \eqv(P.top.73) of \thv(L.top.2') and Lemma \thv(L.top.2''),
\be
\textstyle
(I-A)^{-1}\1=\sum_{k=1}^{\infty}A^k\1=\left(\sum_{\eta\in T}b_{\eta}(\s)\right)^{-1}\1
\ee
where  by \eqv(L.top.new1),
$
\sum_{\eta\in T}v({x_1})
\leq 
\sum_{\eta\in T}b_{\eta}(\s)
\leq 
\sum_{\eta\in T}(v({x_1})+cN^{-1})
$
and
\be
\textstyle
\sum_{\eta\in T}v({x_1})
=\sum_{x_1\in\MM_1}\sum_{\eta\in T^{x_1}}v({x_1})
=\sum_{x_1\in\MM_1}\frac{1}{M_1}(1-\l^{x_1}(T^{x_1})).
\ee
Thus if 
$
1-\l^{x_1}(T^{x_1})\gg N^{-1}
$,
$
v({x_1})
\gg N^{-1}
$.
Now by \eqv(L5new.lev2.2) of Lemma \thv(L5new.lev2),  $\zeta_N\ll \log N$ implies that 
$
\psi_N\gamma^N_1(\xi_1^{x_1})\gg N^{-1}
$
which in turn implies that 
$
1-\l^{x_1}(T^{x_1})\gg N^{-1}
$.
The claim of the lemma now readily follows.
\end{proof} 

We are now ready to prove \eqv(P.top.5), \eqv(P.top.6) and \eqv(P.top.7). Clearly, Eq.~\eqv(P.top.7) 
follows from \eqv(P.top.26'), \eqv(P.top.28) of Lemma \thv(L.top.2) and the bounds \eqv(L.top.new1) 
which are valid for all  $x_1\in\MM_1$, all $\eta\in T^{x_1}$ and all $\s\in\VV_N\setminus\overline W$.
Next, inserting  \eqv(P.top.28) in \eqv(P.top.17') gives
\be
Q_{\eta}(\s)
=\P_{\s}\left(\t_{\del W^{x'_1}}<\t_{T^{x'_1}}\right)\frac{\nu_1(x_1)}{M_2}\left(1+\varepsilon_{N}\right).
\Eq(P.top.44)
\ee
Using \eqv(P1.lev2.5) to express the probability in \eqv(P.top.44) and proceeding as in 
the proof of Lemma \thv(L.top.2) to bound the  term $c_{T^{x'_1}}(\s)$ (that is,  
the terms $b_{T^{x'_1}}(\s,\eta)$) appearing in that expression yields \eqv(P.top.6).
Finally,  inserting  \eqv(P.top.28) in \eqv(P.top.16') gives
\be
P_{\eta}(\s)=
\P_{\s}\left(\t_{\eta}<\t_{(T^{x_1}\setminus\eta)\cup\del W^{x_1}}\right)
+
\P_{\s}\left(\t_{\del W^{x_1}}<\t_{T^{x_1}}\right)\frac{\nu_1(x_1)}{M_2}\left(1+\varepsilon_{N}\right).
\ee
Using Lemma \thv(L.top.2) to bound the two probabilities appearing above, reasoning  again  as in the proof of  Lemma \thv(L.top.2) to bound the  terms $b_{T^{x_1}}(\s,\eta)$ and $c_{T^{x_1}}(\s)$, proves \eqv(P.top.5). 
As for (i-4), it follows from Corollary~\thv(C.inter).
The proof of Assertion (i) of Proposition \thv(P.top) is complete.
\end{proof} 

\begin{proof}[Proof of Assertions (ii) and (iii) of Propositions \thv(P.top)]
The proofs of these two assertions are similar to those of \eqv(P.top.5) and \eqv(P.top.6) and present no new difficulties.
As before they center on ``renewal systems'' that closely ressemble \eqv(P.top.26) and that we now describe.

Given $x_1, \bar x_1\in\MM_1$, let $\eta\in T^{x_1}$ and $\bar\eta\in T^{\bar x_1}$.
Instead of the three quantities of \eqv(P.top.12) we now need to distinguish
six quantities, denoted by $P_{\eta}^{=}(\s),Q_{\eta}^{=}(\s),R_{\eta}^{=}(\s)$ and
$P_{\eta}^{\neq}(\s),Q_{\eta}^{\neq}(\s),R_{\eta}^{\neq}(\s)$, and defined as follows:
letting the symbol $*$ stand for $=$ if $x_1=\bar x_1$ and $\neq$ if $x_1\neq\bar x_1$,
\be
\P_{\s}\left(\t_{\eta}<\t_{T\setminus\{\eta,\bar\eta\}}\right)\equiv
 \begin{cases}
P_{\eta}^{*}(\s) &\mbox{if }   \s\in W^{x_1}\setminus  (T^{x_1}\setminus\bar\eta),\\
Q_{\eta}^{*}(\s) &\mbox{if }  \s\in W^{x'_1}\setminus  (T^{x'_1}\setminus\bar\eta)
\quad\hbox{\rm for some} \quad x'_1\in\MM_1\setminus x_1,\\
R_{\eta}^{*}(\s) &\mbox{if }\s\in \SS_N\setminus(\overline W\setminus\bar\eta)\\
\end{cases} 
\Eq(P.top.46')
\ee
where $T^{x'_1}\setminus\bar\eta=T^{x'_1}$ if $\bar\eta\notin T^{x_1}$.
Note that the probabilities \eqv(P.top.8) and \eqv(P.top.9)
are of the form, respectively, $\overline P_{\eta}(\eta)$ and $\overline Q_{\eta}(\eta)$.
As before they can be expressed as functions of, respectively, $R_{\eta}^{=}(\s)$
and $R_{\eta}^{\neq}(\s)$. Proceeding exactly as in the derivation of
\eqv(P.top.16) and \eqv(P.top.17), we get
\bea
P_{\eta}^{=}(\bar\eta)&=&
\P_{\bar\eta}\left(\t_{\eta}<\t_{(T^{x_1}\setminus\{\eta,\bar\eta\})\cup\del W^{x_1}}\right)
+
\sum_{\s'\in\del W^{x_1}}
\P_{\bar\eta}\left(\t_{\s'}<\t_{(\del W^{x_1}\setminus \s')\cup (T^{x_1}\setminus\bar\eta)}\right)
R_{\eta}^=(\s'),
\nonumber
\\
 Q_{\eta}^{\neq}(\bar\eta)&=& \sum_{\s'\in\del W^{\bar x_1}}
\P_{\bar\eta}\left(\t_{\s'}<\t^{\bar\eta}_{(\del W^{\bar x_1}\setminus \s')\cup (T^{\bar x_1}\setminus\bar\eta)}\right)
R_{\eta}^{\neq}(\s').
\Eq(P.top.47)
\eea
A reasoning similar to that which leads to \eqv(P.top.26) yields, with the same notational convention as above
\be
R_{\eta}^{*}(\s)=b^{*}_{\eta}(\s)+\sum_{\s''\in\del W^{x'_1}}a^{*}(\s,\s'')R^{*}_{\eta}(\s''),
\Eq(P.top.48)
\ee
where  for all $\s\in\SS_N\setminus(\overline W\setminus\bar\eta)$
\bea
b^{=}_{\eta}(\s)
&=&
\left(1+\P_{\eta}^{=}(\bar\eta)\right)\P_{\s}\left(\t_{\eta}<\t_{\overline W\setminus\eta}\right)
\nonumber
\\
&+&
\sum_{\s'\in W^{x_1}\setminus T^{x_1}}
\P_{\s}\left(\t_{\s'}<\t_{\overline W\setminus\s'}\right)
\P_{\s}\left(\t_{\eta}<\t_{(T^{x_1}\setminus\{\eta,\bar\eta\})\cup\del W^{x_1}}\right),
\Eq(P.top.29=)
\\
b^{\neq}_{\eta}(\s)
&=&
\left(1+ Q_{\eta}^{\neq}(\bar\eta)\right)\P_{\s}\left(\t_{\eta}<\t_{\overline W\setminus\eta}\right)
\nonumber
\\
&+&
\sum_{\s'\in W^{x_1}\setminus T^{x_1}}
\P_{\s}\left(\t_{\s'}<\t_{\overline W\setminus\s'}\right)
\P_{\s}\left(\t_{\eta}<\t_{(T^{x_1}\setminus\eta)\cup\del W^{x_1}}\right)
\Eq(P.top.29neq)
\eea
and, 
for $\s\in\SS_N\setminus(\overline W\setminus\bar\eta)$ and $\s''\in\del W^{x'_1}$,
$
a^{=}(\s,\s'')=a^{\neq}(\s,\s'')=\bar a(\s,\s'')
$
where
\be
\bar a(\s,\s'')\equiv
\sum_{x'_1\in\MM_1}
\sum_{\s'\in W^{x'_1}\setminus T^{x'_1}}
\P_{\s}\left(\t_{\s'}<\t_{\overline W\setminus\s'}\right)
\P_{\s'}\left(\t_{\s''}<\t_{(\del W^{x'_1}\setminus \s'')\cup (T^{x'_1}\setminus\bar\eta)}\right).
\Eq(P.top.70=)
\ee
Note that by virtue of Proposition \thv(P.inter),
the terms $P_{\eta}^{=}(\bar\eta)$ and $\overline Q^{\neq}_{\eta}(\bar\eta)$ in \eqv(P.top.29=) and \eqv(P.top.29neq) 
are absorbed in the $\varepsilon_N$ term of \eqv(P.top.8) and \eqv(P.top.9).
The proof of Assertions (ii) and (iii) of Propositions \thv(P.top) are now reruns of the proof of Assertions (i).
We omit the details.
 \end{proof}

\subsection{Proof of Proposition~\ref{paft}}\label{pre_aft}

\subsubsection{Transition within $\cm$: Leaving $\cm$.}\label{leaving}
By the rules of our dynamics 
when leaving $x_1x_2\in\cm$, the probability to jump to 
$x_1'x_2$ for some $x_1'\sim x_1$ equals $q^*_N(\xi^{x_1})$
(recall~\eqv(01.14)), which 
vanishes as $N\to\infty$ for $x_1\in\cm_1$. 
Once $\bxn$ leaves $x_1x_2\in\cm$ and goes to some neighboring $x_1x_2'$, while $\bxn_1$ rests,
the number of jumps $\bxn_2$ would have to take before coming back to $\cm_2$\footnote{Let us recall from Lemma \thv(L2.ls)
	that $x_2'$ may be assumed not in $\cm_2$.} is of the order of $2^{N_2}$ (by
Corollary 1.8 of \cite{BG08}), 
which in this temperature regime is much larger than $1/q^*_N(\xi^{x_1})$, the order of the number of jumps of $\bxn$ before $\bxn_1$ moves. The upshot is that with probability tending to 1 as $N\to\infty$, starting from $\cm$, 
$\bxn$ first leaves $\cm$ in such a way that $\bxn_1$ leaves $\cm_1$ before $\bxn$ returns to $\cm$.

\subsubsection{Transition within $\cm$: Return to $\cm$}

In the presentation of the arguments in the remainder of this subsection, we find it convenient 
to go back to the representation $\bxn$ of our process (introduced in Subsection~\ref{cor}).
Let $\tau^1=\inf\{t>0:\bxn_1\in\cm_1\}$, $\check\tau^1=\inf\{t>\tau^1:\bxn_1\notin\cm_1\}$
and, for $i\geq2$, $\tau^i=\inf\{t>\check\tau^{i-1}:\s^N\in\cm_1\}$,
$\check\tau^i=\inf\{t>\tau^{i}:\bxn_1\notin\cm_1\}$.
We then have that $\tau^i$,
$i=1,2,\ldots$, represent the successive hitting times of $\cm_1$ by $\bxn_1$.
Notice that $\tau^1=\tau_{\overline W}$.
Let also $A_i$, resp.~$A_i^y$, $i=1,2,\ldots$, denote the event that $\bxn_2$ hits $\cm_2$, resp.~$y\in\cm_2$ 
during the $i$-th visit of $\bxn_1$ to $\cm_1$. 
Let $\ci=\min\{i\geq1:\, A_i\mbox{ occurs}\}$. Below we will compute the limit as $N\to\infty$ of 
\be\label{rel1}
\P_\s(\bxn_1(\tau^{\ci})=y_1,A_{\ci}^{y_2}).
\ee
From the discussion on Subsubsection~\ref{leaving}, we may take $\s\notin W$.
The expression in~(\ref{rel1}) is not quite the probability in the left hand side of~(\ref{lim1}), 
but close enough in the sense that they turn
out to have the same limit, as will also be argued below, in the conclusion of the our proof.

For $x\in\cm_1$, set
\be\label{pi}
\pi(x)=\P_{\mu_2}(A_1^y|\bxn_1(0)=x);\quad\hat\pi(x)=\P_{\mu_2}(A_1|\bxn_1(0)=x),
\ee
where $\mu_2$ is the uniform initial distribution of $\bxn_2$ (on $\VV_{N_2}$).
Notice that $\pi(x)$ does not depend on $y$.
We show in the appendix -- see Lemma~\ref{lm:pis} -- that
\be\label{pis}
\pi(x)\sim\frac{N_2}{N_1}\,\g_1(x)\frac1{c_1^N2^{N_2}};\quad\hat\pi(x)\sim M_2\pi(x).
\ee

Let now
$D_i$ denote the event that $\bxn_2$ gives at least $N_2^3$ steps between times $\tau^{i-1}$ and $\tau^i$, $i\geq2$.
Now let $L_1=\inf\{i\geq2:\,D_i \mbox{ occurs}\}$, and, for $k\geq2$,  $L_k=\inf\{i> L_{k-1}:\,D_i \mbox{ occurs}\}$ , and
define $\bar\tau^1=\tau^1$ and $\bar\tau^k=\tau^{L_k}$, $k\geq2$.
Let also $\bar A_i$, $i=1,2,\ldots$, denote the event that $\bxn_2$ hits $\cm_2$ while $\bxn_1\in\cm_1$  between 
times $\bar\tau^{i}$ and $\bar\tau^{i+1}$, $i\geq1$. Finally, let $\bar\ci=\min\{i\geq1:\,\bar A_i\mbox{ occurs}\}$.

By Lemma~\ref{rm5a} and 
Lemma 3.1 of \cite{BG13}, we may couple $\bxn$ to a process $\uxn=\uxn_1\uxn_2$ such that
$\uxn_1=\bxn_1$ and, defining the random times in the above paragraph in the same way for $\uxn$, we have that $\uxn_2$ is
uniformly distributed on $\cd_2$ at the times $\bar\tau^i$, $i\geq2$, such that with probability tending to 1 as 
$N\to\infty$, $\uxn_2=\bxn_2$ for all times till $\bar\tau^{\bar\ci}$. 
Notice that, since 
$\bar\ci\leq\ci$, Lemma~\ref{rm5a} holds also for $\bar\ci$.

Let $\ci'=\min\{i\geq0:\,A_{L_{\bar\ci}+i} \mbox{ occurs}\}$. Then for $x\in\cm_1$, apart from an $o(1)$ error according to
the above paragraph, we have that 
\begin{equation}
\label{adj1}
\P_\s(\bxn_1(\tau^{\ci})=x,A_{\ci}^y)=\sum_{k=1}^\infty P_k(xy),
\end{equation}
where
\begin{eqnarray}\nn
P_k(xy)&=&\sum_{\ell=0}^\infty\sum_{z\in\cm_1} 
 \P_\s(\uxn_1(\bar\tau^{k})=z,\uxn_1(\tau^{L_{k}+\ell})=x,\bar\ci=k,\ci'=\ell,A^y_{L_k+l})\\\nn
&=&\sum_{\ell=0}^\infty\sum_{x_1,\ldots,x_{k+\ell-1}\in\cm_1}\!\!\! \P_\s(\uxn_1(\bar\tau^{1})=x_1,\ldots,\uxn_1(\bar\tau^{k-1})=x_{k-1},\\\nn
&&
\uxn_1(\tau^{L_k})=x_k,\ldots,\uxn_1(\tau^{L_k+\ell-1})=x_{k+\ell-1},
\uxn_1(\tau^{L_k+\ell})=x,\\\label{adj2}
&&\bar A_1^c,\ldots,\bar A_{k-1}^c, A^c_{L_k},F_1,\dots, A^c_{L_k+\ell-1},F_\ell, A^y_{L_k+\ell}),
\end{eqnarray}
where $F_i$ is the event that $\uxn_2$ gives less than $N_2^3$ jumps between $\tau^{L_k+i-1}$ and
$\tau^{L_k+i}$, $i=1,2,\ldots$, and all the other quantities and events in the latter probability
should be defined with $\bxn$ replaced by $\uxn$. Using the Markov property, the right hand side above 
can be written as
\begin{eqnarray}\nn
&{\displaystyle {M_1}^{-1}\!\!\!\!\sum_{x_0,\ldots,x_{k-1}\in\cm_1}}\!\!&\!\!\prod_{i=1}^{k-1}(1-\bar\pi(x_i)) 
\P(\uxn_1(\bar\tau^{i})=x_i|\uxn_1(\bar\tau^{i-1})=x_{i-1},\bar A_{i-1}^c)\\\label{adj3}
&\times&\!\!\!\!\!\!\!\!\!\!\!\!\!\!\!\!\!\!\!\!\!\!\!\left[\pi(x) 
\P(\uxn_1(\tau^{L_k})=x|\uxn_1(\bar\tau^{k-1})=x_{k-1},\bar A_{k-1}^c)
+\ca_{x_{k-1},x}\right]\!\!, 
\end{eqnarray}
where $\bar\pi(x_i)=\P_{\mu_2}(\bar A_i|\bxn_1(\bar\tau^{i})=x_i)$ (it does not depend on $i\geq1$ except through $x_i$) 
and 
$\ca_{x_{k-1},x}$ equals
\begin{eqnarray}\label{adj4}
\sum_{\ell=1}^\infty\sum_{x_k,\ldots,x_{k+\ell-1}\in\cm_1}\cq^{\ell}_{x_k,\ldots,x_{k+\ell-1}},
\end{eqnarray}
where $\cq^{\ell}_{x_k,\ldots,x_{k+\ell-1}}$ equals
\begin{eqnarray}\nn
&\P(\uxn_1(\tau^{L_k})=x_k,A^c_{L_k},F_1,\ldots,\uxn_1(\tau^{L_k+\ell-1})=x_{k+\ell-1},A^c_{L_k+\ell-1},F_\ell,\\
\label{adj5}
&\uxn_1(\tau^{L_k+\ell})=x,A^y_{L_k+\ell}|\uxn_1(\bar\tau^{k-1})=x_{k-1},\bar A_{k-1}^c).&
\end{eqnarray}

Removing the $A^c$'s and using the Markov property, we find that the expression in~(\ref{adj5}) is bounded above by
\begin{eqnarray}\nn
&\P(\uxn_1(\tau^{L_k})=x_k|\uxn_1(\bar\tau^{k-1})=x_{k-1},\bar A_{k-1}^c)\P(F_1|\uxn_1(\tau^{L_k})=x_k)\times&\\\nn
&\prod_{i=1}^{\ell-1} \P(\uxn_1(\tau^{L_k+i})=x_{k+i}|\uxn_1(\tau^{L_k+i-1})=x_{k+i-1},F_{i})
\P(F_{i+1}|\uxn_1(\tau^{L_k+i})=x_{k+i})\times&\\\label{adj6}
&\P(\uxn_1(\tau^{L_k+\ell})=x|\uxn_1(\tau^{L_k+\ell-1})=x_{k+\ell-1},F_\ell)
\P(A^y_{L_k+\ell}|\uxn_1(\tau^{L_k+\ell})=x),&
\end{eqnarray}
since the events $F_1,F_2,\ldots$ depend only on $\uxn_1$, and starting with a uniform distribution on $\cd_2$,
its invariant distribution, at time $L_k$, $\uxn_2(L_k+\ell)$ still has that same distribution. 

Now from Lemma~\ref{rm3} and Remark~\ref{rm4} above, each probability of the form $\P(F_i|\cdot)$ 
in~(\ref{adj6}) 
are bounded above by $c/N$ for some constant $c$. We also notice that the latter probability in the 
same expression equals $\pi(x)$. 
We thus have that~(\ref{adj4}) is bounded above by
\begin{eqnarray}\nn
&\pi(x)\sum_{\ell=1}^\infty\left(\frac c{N}\right)^\ell 
\sum_{x_k,\ldots,x_{k+\ell-1}\in\cm_1}
\P(\uxn_1(\tau^{L_k})=x_k|\uxn_1(\bar\tau^{k-1})=x_{k-1},\bar A_{k-1}^c)\times&\\\nn 
&\prod_{i=1}^{\ell-1} \P(\uxn_1(\tau^{L_k+i})=x_{k+i}|\uxn_1(\tau^{L_k+i-1})=x_{k+i-1},F_{i}),&
\end{eqnarray}
and since the latter sum is over probabilities, it equals 1. It follows
that the expression within brackets on the bottom of~(\ref{adj3})
equals 
\begin{equation}\label{adj6a}
[\P(\uxn_1(\tau^{L_k})=x|\uxn_1(\bar\tau^{k-1})=x_{k-1},\bar A_{k-1}^c)+o_{x_{k-1},x}]\pi(x),
\end{equation}
where $o_{x',x}$ is an $o(1)$ for every $x,x'\in\cm_1$.

Let us now consider $\bar\pi(z)$ for a given $z\in\cm_1$; arguing in the same way as for the expression within 
brackets on the bottom of~(\ref{adj3}), we find that it equals $\hat\pi(z)$ plus an error bounded
above by
\begin{equation}\nn
\sum_{\ell=1}^\infty\left(\frac c{N}\right)^\ell 
\sum_{x'_1,\ldots,x'_{\ell}\in\cm_1}\!
\prod_{i=1}^{\ell} \P(\uxn_1(\tau^{L_{k-1}+i})=x'_{i}|\uxn_1(\tau^{L_{k-1}+i-1})\!=x'_{i-1},F'_{i})\,\hat\pi(x'_\ell),
\end{equation}
where $F'_1,F'_2,\ldots$ are defined in the obvious, parallel way to $F_1,F_2,\ldots$ above. 
From 
(\ref{pis}), 
we have that 
$\hat\pi(x)$, $x\in\cm_1$, are all of the same order of magnitude. It follows that 
\begin{equation}\label{adj7c}
\bar\pi(x_i)=
(1+o_{x_i})\hat\pi(x_i),\, x_i\in\cm_1, 
\end{equation}
where $o_{x}$ is an $o(1)$ for every $x\in\cm_1$.

In particular, we have that $\P_{\mu_2}(\bar A^c_i|\bxn_1(\bar\tau^{i})=x_i)\sim1$ for all $i=1,2,\ldots$ and using also 
Corollary 1.5 of \cite{BG08},
we find that
\begin{equation}\label{adj7d}
 \bar P_{x,y}:=\P(\uxn_1(\bar\tau^{i})=y|\uxn_1(\bar\tau^{i-1})=x,\bar A_{i-1}^c)\sim\frac1{M_1},\,\,x,y\in\cm_1.
\end{equation}
We note that the latter conditional probability does not depend on $i=1,2,\ldots$.

The upshot of the above discussion is that the right hand side of~(\ref{adj1}) may be written as
\begin{equation}
\label{adj8}
(1+o(1))\frac1{M_1}\sum_{x_0\in\cm_1}\sum_{k=1}^\infty \bar R^k_{x_0,x}\frac{\pi(x)}{1-\bar\pi(x)},
\end{equation}
where $\bar R^k$ is the $k$-th power of the matrix
\begin{equation}
\label{adj8a}
\bar R=\bar P(I-\bar\Pi),
\end{equation}
where $I$ is the identity matrix in $\cm_1$ and $\bar\Pi$ is the diagonal matrix in $\cm_1$ with 
entries $\{\bar\pi(y),\,y\in\cm_1\}$.

Now by~(\ref{pis}), we have that
\begin{equation}
\label{adj8a1}
\bar\Pi\sim M_2\eps\G,
\end{equation}
where $\eps=\frac{N_2}{N_1}\frac1{c_1^N2^{N_2+1}}$ and 
\begin{equation}
\label{adj8a2}
\G = \mbox{diag}\,\{\g_1(y),\,y\in\cm_1\}
\end{equation}
is the diagonal matrix in $\cm_1$ with entries $\{\g_1(y),\,y\in\cm_1\}$;
thus, it is an $o(1)$ and from~(\ref{adj7d}), we have that $\bar R$ is a positive matrix 
for all large enough $N$. We may thus apply Perron-Frobenius theory to write the internal sum in~(\ref{adj8})
as
\begin{equation}
\label{adj8b}
\sum_{k=1}^\infty \rho^k[(\rho^{-1}\bar R)^k-\bar S]_{x_0,x}\frac{\pi(x)}{1-\bar\pi(x)}
+\frac1{1-\bar\pi(x)}\frac{\bar S_{x_0,x}\pi(x)}{1-\rho},
\end{equation}
where $\rho$ is the top eigenvalue of $\bar R$, and $\bar S=vw^T$, with $v,w$ the right and left
eigenvectors of $\bar R$ associated to $\rho$ such that $v^Tw=1$. 
See Theorem 8.2.11, its proof, and the preceding and subsequent material of Section 2 of Chapter 8 of~\cite{HJ}.
To check that $\rho<1$ 
for all large enough $N$, we first note 
that $\bar R$ is a perturbation of $\bar P$, which is stochastic, and thus has 1 as its top eigenvalue,
and then resort to a standard perturbation result to the effect that
\begin{equation}
\label{adj8c}
\rho=1-M_2\,\eps'\,\frac{\bar w^T\bar P\G\bar v}{\bar w^T\bar v}+C\eps^2,
\end{equation}
where $\bar v$ and $\bar w$ are right and left eigenvectors of $\bar P$ associated to the eigenvalue 1,
$\eps'\sim\eps$, and $C$ is a constant.
See Theorem~IV.2.3 in~\cite{SS}.
Since the latter matrix is stochastic, we may take $\bar v$ as the vector with all entries equal to 1.
By~(\ref{adj7d}) and again well known perturbation results, we may take $\bar w\sim\bar v$ 
(see Subsection V.2.3 of~\cite{SS}), and from~(\ref{adj7d}) we have that
$\bar P\sim \frac1{M_1}\,I$; we thus get
\begin{equation}
\label{adj8d}
\rho=1-M_2\,\eps''\,\bar\g_1+C\eps^2,
\end{equation}
where $\bar\g_1=\frac1{M_1}\sum_{y\in\cm_1}\g_1(y)$, and $\eps''\sim\eps$. We then have that $\rho<1$ 
for all large enough $N$.

Again Perron-Frobenius theory tells us that the expression within brackets in~(\ref{adj8b}) decays exponentially
fast in $k$, uniformly in $N$. Since $\pi(x),\bar\pi(x)$ are $o(1)$ for all $x\in\cm_1$, the infinite sum in~(\ref{adj8b})
vanishes as $N\to\infty$. 

Again resorting to well known perturbation theory results, since $\bar R$ is also a perturbation of $\frac1{M_1}\,I$,
we have that $v$ and $w$ may be taken as $v\sim\bar v=(1,\ldots,1)$,
and $w\sim\frac1{M_1}v$; 
it follows that $\bar S\sim \frac1{M_1}\,I$; the upshot is that the second summand in~(\ref{adj8b}) is 
asymptotic to
\begin{equation}
\label{adj8e}
\frac{\frac{1}{M_1}\eps\g_1(x)}{M_2\,\eps''\,\bar\g}\sim\frac{\g_1(x)}{\sum_{y\in\cm_1}\g_1(y)}\frac1{M_2},
\end{equation}
and thus so is the left hand side of~(\ref{adj1}).

\subsubsection{Conclusion of the proof of Proposition~\ref{paft}.} 

Let us show that $P(\bxn_1(\tau^{\ci})=x,A_{\ci}^y)$  
agrees with the entrance probability at left hand side of~(\ref{lim1}) apart from an $o(1)$ error.
As it stands, the former probability is actually the probability that $\bxn_2$ visits $y$ 
during the first visit
of $\bxn_1$ to $\cm_1$ where $\bxn_2$ hits $\cm_2$: the visit of $\bxn_2$ to $y$ may be not the first one to 
$\cm_2$. However, this probability is clearly an upper bound for the entrance probability, and if we subtract the 
following probability, we obtain a lower bound. For $i=1,2,\ldots$, let $B_i$ denote the event that $\bxn_2$ hits 
$\cm_2$ at least twice during the $i$-th visit of $\bxn_1$ to $\cm_1$. We may then estimate
$\P_{\mu_2}(\uxn_1(\tau^{\ci})=x,B_{\ci})$ in the same way as above (starting in~(\ref{adj1})), by replacing $A^y_{\ci},A^y_{L_k+\ell}$ by $B_{\ci},B_{L_k+\ell}$, respectively, and $\pi(x)$ by $\P_{\mu_2}(B_1|\bxn_1(\tau^{1})=x)$. But this is bounded above by
the right hand side of~(\ref{bd}), which was shown above to be an $o(\pi(x))$. We thus get that 
$\P_{\mu_2}(\uxn_1(\tau^{\ci})=x,B_{\ci})$ is an $o(1)$, and subtracting it from $\P_{\mu_2}(\uxn_1(\tau^{\ci})=x,A_{\ci}^y)$
gives us a lower bound for the entrance law; the right hand side of~(\ref{adj8e}) as the limit for the 
latter quantity follows.

\section{Proof of Theorem~\ref{thm:cl}}
\label{proofcl}

\setcounter{equation}{0}

In this and next two sections, we present the proofs of the scaling limit theorems for $X^N$, one section for each proof. 
We will use the results of Section~\thv(1.5) on entrance laws. There are two remaining things to establish in the case of Theorems~\ref{thm:cm} and~\ref{thm:cft}: that the process spends virtually all of the time at the top, and what time is spent at each visit of a top configuration. 
The structure of the proof of Theorem~\ref{thm:cl} is not dissimilar: we control time spent off the top states as before, then evaluate the time spent on visits of $X^N_1$ to top first level states, and finally resort to a spectral gap argument to get the behavior on second level.

\medskip

Specifically in this section, we will concentrate on showing that
\begin{enumerate}
	\item $\bxn$ spends virtually all the time on $\cm$;
	\item the time $\bxn_1$ spends on each visit to each $x_1\in\cm_1$ is roughly exponential with mean $\ff_3(x_1)$;
	\item given an interval of constancy $I=[a,b)$ of $\bxn_1$ where $\bxn_1=x_1$ for some $x_1\in\cm_1$, and  $t_1,\ldots,t_k$ such that $a<t_1<\ldots,t_k<b$ for some $k\geq1$, we have that $\bxn_2(t_1),\ldots,\bxn_2(t_k)$ are roughly independent random variables taking values on $\N$, each distributed roughly with probability weights given by $\g_2(x_1,\cdot)$ normalized.
\end{enumerate} 

After arguing these points in variable detail, we sketch an argument on how they fit together in a proof of Theorem~\ref{thm:cl}. We start with the second point, after a few remarks.

Let us notice that the total time spent by $\s^N(\cdot/c_2^N)$ on any single visit to a given $\s_1\in\VV_{N_1}$ can be written as
\begin{equation}
\label{v1}
\sum_{j=0}^{\cg}\tg_2^N(\s_1J^\circ_{N_2}(j))T_j,
\end{equation}
where 
$\cg$ is the geometric random variable $\tau_{\del C}$ with success probability $q^*_N(\s_1)$ given by~\eqv(01.14)
and $T_1,T_2,\ldots$ are iid mean 1 exponential random variables,
and 
\begin{equation}\label{tg}
\tg_2^N(\s)=\frac{\frac N{N_2}\g_2^N(\s)}{1+\frac{N_1}{N_2}e^{\b \sqrt{aN}\Xi^{(1)}_{\s_1}}}.
\end{equation}

\begin{remark}\label{gtg}
	It should be quite clear that~(\ref{g1ntog1}) and~(\ref{gtog}) remain valid when we replace $\g_2^N$ and $\g_2$ by
	$\tg_2^N$ and $\tg_2$, respectively --- see paragraph right above the statement of Theorem~\ref{thm:cm} ---; one reason for this is that the denominator in~(\ref{tg}) tends to 1 as $N\to\infty$ for every fixed $x_1$.
\end{remark}	
\begin{remark}\label{sko}	
	We may indeed assume that all the convergences mentioned in the previous remark, which are weak for the original environment, may be taken strong by going to another, suitable probability space for the environment (resorting to Skorohod's theorem). We will effectively, and for convenience, assume below that we are in the full measure event of such a probability space where those convergences take place, and omit further reference to it.
\end{remark}

\subsection{Time spent on top first level visits}\label{tfl}

Let $x_1\in\cm_1$. Set $\cx_2=(\xi_2^{x_1\cdot})^{-1}(J^\circ_{N_2})$, 
where $\xi_2^{x_1\cdot}$ is the function from $\DD_2$ to $\VV_{N_2}$ mapping $x_2$ to $\xi_2^{x_1x_2}$.
The time spent in $\cm$ by $\bxn_1$ on its $i$-th visit to $x_1\in\cm_1$ can expressed as
\begin{equation}
\label{u1}
\hat\Ups^N_i(x_1):=c_1^N 2^{N_2}\sum_{j=0}^{\cg}1_{\{\cx_2(j)\in\cm_2\}}\tg_2^N(x_1\cx_2(j))\,T_j.
\end{equation}

\begin{lemma}\label{exp_bft}
	For each $x_1\in\N$,
$\hat\Ups^N_i(x_1)$	converges weakly as $N\to\infty$ to an exponential random variable of mean $\ff^M_3(x_1)=\g_1(x_1)\sum_{x_2\in\cm_2}\g_2(x_1x_2)$.
\end{lemma}

\begin{proof} 
We write
\begin{equation}
\label{u2}
\hat\Ups^N_i(x_1)\ed c_1^N(1+\cg) \sum_{x_2\in\cm_2}\g_2^N(x_1x_2)\,L_N(x_2),
\end{equation}
where
\begin{equation}
\label{u3}
L_N(x_2):=\frac{2^{N_2}}{1+\cg}\sum_{j=0}^{\cg}1_{\{\cx_2(j)=x_2\}}\,T_j(x_2),
\end{equation}
with $\{T_j(x_2);\,j\geq0,x_2\geq1\}$ an iid family of mean one exponential random variables
independent from $\cg$.

One may readily check that
\begin{equation}
\label{u4}
c_1^N(1+\cg)\to \g_1(x_1)\,\ce
\end{equation}
in distribution as $N\to\infty$, where $\ce$ is a mean one exponential random variable.

We will next show that for every $x_2\in\cm_2$
\begin{equation}
\label{u5}
L_N(x_2)\to1
\end{equation}
in probability as $N\to\infty$. This and~(\ref{u4}) readily implies that
\begin{equation}
\label{u6}
\hat\Ups^N_i(x_1)\to \ff^M_3(x_1)\,\ce
\end{equation}
in distribution as $N\to\infty$.

From~(\ref{u4}), since $\cg$ is independent from the family of exponential random variables entering
$L_N(x_2)$, we may suppose that $\cg$ is roughly equal to $\hat c_1^Nr$, with $r>0$ a real number, where 
$\hat c_1^N=1/c_1^N$. 

So rather than $L_N(x_2)$, we may consider instead
\begin{equation}
\label{u8}
\hat L_N(x_2,r):=\frac{2^{N_2}}{1+\hat c_1^Nr}\sum_{j=0}^{\hat c_1^Nr}1_{\{\cx_2(j)=x_2\}}\,T_j(x_2),
\end{equation}
and show that
\begin{equation}
\label{u9}
\hat L_N(x_2,r)\to1
\end{equation}
in probability as $N\to\infty$ for every $r>0$.

We now note that the sum in~(\ref{u8}) may be understood as the time spent on $\xi_2^{x_1x_2}$ 
by a continuous time space homogeneous simple symmetric random walk in the $\VV_{N_2}$
during the first $\hat c_1^Nr$ jumps.

Let us consider the delayed renewal process associated to that random walk consisting of successive return
times of that random walk to $\xi_2^{x_1x_2}$. Each such time (with the possible exception of the first one) can be decomposed as the sum of a mean one exponential random variable --- the initial time spent on $\xi_2^{x_1x_2}$ by the random walk prior to that particular return to $\xi_2^{x_1x_2}$ --- and the hitting time of $\xi_2^{x_1x_2}$ by the random walk starting from the neighbor of $\xi_2^{x_1x_2}$ it jumps to after that initial time. We then have a delayed renewal process with renewal times $\ce_1+\cR_1,\ce_2+\cR_2,\ldots$, with $\ce_1,\cR_1,\ce_2,\cR_2,\ldots$ independent, $\ce_2,\ce_3,\ldots$
and $\cR_2,\cR_3,\ldots$ identically distributed, $\ce_2$ a mean one exponential random variable, and $\cR_2$ is distributed as the hitting time of $\xi_2^{x_1x_2}$ starting from a nearest neighbor of $\xi_2^{x_1x_2}$. $\ce_1$ may either be distributed as $\ce_2$ or vanish, depending on whether the state of $\bxn_2$ at the beginning of the visit of $\bxn_1$ to $x_1$ was $x_2$ or not. Similarly, $\cR_1$ may be either distributed as $\cR_2$ or as the hitting time of $\xi_2^{x_1x_2}$ by the random walk starting from another site of $\VV_{N_2}$ not $\xi_2^{x_1x_2}$ or a neighbor of $\xi_2^{x_1x_2}$. 
As will be clear below, neither (the distribution of) $\ce_1$ or $\cR_1$ will play a role in the result. 
Let $S_n=\sum_{i=1}^n(\ce_i+\cR_i)$ and $S'_n=\sum_{i=1}^n\cR_i$, $n\geq1$. Let $N_t$ the counting process associated to $S_n$, namely
$N_t=N(t)=\sup\{n\geq0:\,S_n\leq t\}$, $t\geq0$, $S_0=0$. Notice that the sum in~(\ref{u8}) is bounded below and above respectively by
\be\label{u9a}
\sum_{j=0}^{N(\hat c_1^Nr)}T_j(x_2),\quad\sum_{j=0}^{N(\hat c_1^Nr+1)}T_j(x_2).
\ee 

We now claim that in order to establish~(\ref{u9}), it is enough to show that
\begin{equation}
\label{u10}
\frac1KS'_Q\to1
\end{equation}
in probability as $N\to\infty$, where $K=K_N=\hat c_1^Nr$, $Q=Q_N=\hat c_1^Nr\,2^{-N_2}$.
Indeed, from~(\ref{u10}) and the law of large numbers satisfied by iid mean one exponential random variables,
it readily follows that $\frac1KS_Q\to1$ in probability as $N\to\infty$. This in turn readily implies 
that $\frac1QN_K\to1$ in probability as $N\to\infty$, and again the law of large numbers satisfied by 
iid mean one exponential random variables implies that either of the two expressions in~(\ref{u9a}),
after division by $Q$, converges to 1 in probability as $N\to\infty$. The claim is established.
 
We may ignore $\cR_1$ in the argument for~(\ref{u10}), or take it identically ditributed to $\cR_2$. We take the Laplace transform of the left hand side of~(\ref{u10}) as follows. For $t>0$
\begin{equation}
\label{u11}
\bar\E\left(e^{-t\frac1KS'_Q}\right)=\left[\bar\E\left(e^{-t\frac1K\cR_2}\right)\right]^Q
\end{equation}
where $\bar\E$ denotes expectation with respect to the law of $\bar J^{\circ}_{N_2}$.

It follows from Proposition $7.7.i.b$ of \cite{BG08}, after a straightforward adjustment for continuous time, 
that the expression within square brackets on the right of~(\ref{u11}) can be written as
\begin{equation}
\label{u12}
\frac{1+\tilde o}{1+\frac{t[1+o(1)]}{\hat c_1^Nr\,2^{-N_2}}},
\end{equation}
where $\tilde o=o(\hat c_1^N2^{-N_2})$, and~(\ref{u10}) follows.
\end{proof}

\subsection{Equilibrium on the second level}\label{sga}

Let us check the third point of the list outlined at the beginning of the section.

We initially remark that during a constancy interval of $\bxn_1$, where $\bxn_1=x_1$ for a giben $x_1\in\cd_1$,
$\bxn_2$ is the mapping via $\xi$ of a continuous time simple random walk on the hypercube $\VV_{N_2}$ with mean waiting time 
at $\xi_2^{x_1x_2}$ given by $\g_2^N(x_1,x_2)$, starting from whichever second level configuration $\s^N$ was at the
beginning of the interval. We denote this random walk by $\bar\s_2^N$.

Let us now briefly argue the claim that the time to reach equilibrium for that random walk
is of order smaller than that of the length of the constancy interval -- which we just saw above to be of the
order of the inverse of $\bar c^N=  c_1^N 2^{N_2} c_2^N$.

After straightforward adjustments, we may check that the bound derived in~\cite{FIKP} for the associated Metropolis dynamics for the REM applies for $\bar\s_2^N$, and we find that

\begin{equation}\label{sg}
\limsup_{N\to\infty}\frac1{N_2}\log\mt_2\leq\sqrt{1-a}\b\b_*
\end{equation}
almost surely, where $\mt_2$ is the inverse of the spectral gap of $\bar\s_2^N$.

It follows that
\begin{equation}\label{comp}
\liminf_{N\to\infty}\frac1{N}[\log(\bar c_N)^{-1}-\log\mt_2]>0,
\end{equation}
and
\begin{equation}\label{var1}
\max_{\s_2\in\VV_{N_2}}\left\| \P_{\s_2}\left(\bar\s_2^N(t)=\cdot\right)-G_{N_2}(\cdot)\right\|
\leq\sqrt{Z_{N_2}\max_{\s_2\in\VV_{N_2}}\exp\{\b\sqrt{(1-a)N}\Xi^{(2)}_{\xi_1^{x_1}\s_2} \}}\,e^{-t/\mt_2}, 
\end{equation}
where $G_{N_2}$ is the equilibrium Gibbs measure for $\bar\s_2^N$, which is proportional to the weights 
$\exp\{-\b \sqrt{(1-a)N}\Xi^{(2)}_{\xi_1^{x_1}\cdot}\}$, and $Z_{N_2}$ is the partition function associated to $G_{N_2}$. 

From well known results about the existence and exact expression for the limit of $\frac1N\log$ of both factors 
inside the square root above, we have that almost surely that square root is bounded from above by $e^{cN}$ for some finite
constant $c$ for all large enough $N$. It immediately follows from~(\ref{comp}) and the above that for times of the form $t=s(\bar c_N)^{-1}$, $s>0$, we have that the left hand side of~(\ref{var1}) is almost surely bounded above by $e^{cN}e^{-e^{dN}}$ for all large enough $N$, with $d>0$ related to the left hand side of~(\ref{comp}), and thus it 
almost surely vanishes as $N\to\infty$.
This and~(\ref{g1ntog1}) in turn readily imply the claim of the third point at the beginning of the section.

\subsection{Time spent by $\bxn$ outside $\cm$}\label{tsp}

\subsubsection{Preliminaries} 

We start with results about the number of visits of a given configuration $\s'$ by a random walk on $\VV_N$ before reaching vertex $\s\ne\s'$. There are two initial situations: equilibrium and $\s$.

Let $\tau_{\s}=\inf\{k\geq1:\,J^\circ_N=\s\}$, where $J^\circ_N$ is the random walk on $\VV_N$. We know from elementary theory of Markov chains that
\begin{equation}\label{e1}
\E^\circ_{\s}\left(\sum_{k=0}^{\tau_{\s}-1}1\{J^\circ_N(k)=\s'\}\right)=1
\end{equation}
(see e.g.~Theorems 1.7.5 and 1.7.6 in~\cite{N}; to get~(\ref{e1}) we use also the fact that the uniform measure on the vertices of the hypercube is invariant for $\cx$, which is moreover irreducible).

Let $\mu$ denote the uniform invariant measure for $J^\circ$.
\begin{lemma}\label{lem}
Suppose $\s\ne\s'$.	For all large enough $n$
\begin{equation}\label{e2}
\E^\circ_{\mu}\left(\sum_{k=0}^{\tau_\s-1}1\{J^\circ_N(k)=\s'\}\right)\leq2.
\end{equation}
\end{lemma}

\begin{proof}
We write the left hand side of~(\ref{e2}) as follows
\begin{eqnarray*}
&&\sum_{k=0}^\infty \P^\circ_\mu(J^\circ_N(k)=\s',\tau_\s\geq k+1)=
\frac1{2^n}+\frac1{2^n}\sum_{k=1}^\infty \sum_{\s''\ne \s}\P^\circ_{\s''}(J^\circ_N(k)=\s',\tau_\s\geq k+1)\\
&\stackrel{*}{=}&\frac1{2^n}+
\frac1{2^n}\sum_{k=1}^\infty \sum_{z\ne y}\P^\circ_{\s'}(\tau_\s\geq k,J^\circ_N(k)=\s'')\\
&=&\frac1{2^n}+\frac1{2^n}\sum_{k=1}^\infty \P^\circ_{\s'}(\tau_\s\geq k)
-\frac1{2^n}\sum_{k=1}^\infty \P^\circ_{\s'}(\tau_\s\geq k,J^\circ_N(k)=\s)\\
&=&\frac1{2^n}+\frac1{2^n}\sum_{k=1}^\infty \P^\circ_{\s'}(\t_\s\geq k)-\frac1{2^n}\sum_{k=1}^\infty \P^\circ_{\s'}(\t_\s= k)
=\frac1{2^n}\E^\circ_{\s'}(\t_\s)\leq2,
\end{eqnarray*}
for all large enough $n$, where $\stackrel{*}{=}$ is due to the reversibility of $\cx$, and the inequality at the end follows from Theorem 1.6 of~\cite{BG08}. 
\end{proof}

\subsubsection{Time outside $ T_1$}\label{out1}

Let us estimate the time spent by $\bxn_1$ outside $ T_1$ till the first visit to a vertex $\s_1\in T_1$, and between
two visits to $\s_1$. Let us fix $x_1\in\cm_1$ and take $\s_1=\xi_1^{x_1}$.

Let $\cu$ denote the first such time, which can be written as follows.
\begin{equation}\label{oc1}
\cu=c_1^N 2^{N_2}\sum_{k=0}^{\t_{\s_1}-1}1\{J^\circ_{N_1}(k)\notin T_1\}\!\!\!
\sum_{j=0}^{\cg_{k+1}(J^\circ_{N_1}(k))}\!\!\!
\tg_2^N(J^\circ_{N_1}(k)J^\circ_{N_2}(\ms_{k}+j))\,T^k_j,
\end{equation}
where $\mg:=\{\cg_1(\s'_1),\cg_2(\s'_1),\ldots;\,\s'_1\in\cv_{N_1}\}$ is an independent family of
geometric random variables with mean $\frac{N_2}{N_1}e^{-\b H_1(\s'_1)}$, independent of $J^\circ_{N_1}$ and $J^\circ_{N_2}$;
$\ms_{k}=\sum_{i=1}^k{\cg_i(J^\circ_{N_1}(i))}$, $k\geq1$; $T^k_j$, 
$j,k\geq0$ are iid mean 1 exponential random variables, independent from all the other random variables.

We note that $J^\circ_{N_1}$ and $J^\circ_{N_2}$ are independent discrete time random walks on the hypercubes $\cv_{N_1}$ and $\cv_{N_2}$,
respectively, each starting from its respective equilibrium distribution.
Thus
\begin{equation}\label{oc2}
\E^\circ(\cu|J^\circ_{N_1},\mg)=c_1^N 2^{N_2}\!\sum_{k=0}^{\t_{\s_1}-1}\!1\{J^\circ_{N_1}(k)\notin T_1\}
\!\!\!\!\!\!\!\sum_{j=0}^{\cg_{k+1}(J^\circ_{N_1}(k))}\!\!\!\!\!\!
E\left(\tg_2^N(J^\circ_{N_1}(k)J^\circ_{N_2}(\mg_{k}+j))|J^\circ_{N_1},\mg\right)\!. 
\end{equation}

The conditional expectation on the right hand side may be written as
\begin{equation}\label{oc3}
\sum_{\s_2\in\cv_{N_2}}\g_2^N(J^\circ_{N_1}(k)\s_2)P(J^\circ_{N_2}(\ms_{k}+j)=\s_2))=\frac1{2^{N_2}}\sum_{\s_2\in\cv_{N_2}}\tg_2^N(J^\circ_{N_1}(k)\s_2). 
\end{equation}
Thus
\begin{eqnarray}\nn
\E^\circ(\cu|J^\circ_{N_1})&=&\sum_{\s_2\in\cv_{N_2}}\!\sum_{k=0}^{\t_{\s_1}-1}\!1\{J^\circ_{N_1}(k)\notin T_1\}
\tg_2^N(J^\circ_{N_1}(k)\s_2)c_1^N \E^\circ(\cg_{k+1}(J^\circ_{N_1}(k))|J^\circ_{N_1})\\
\label{oc4}
&=&\sum_{\s_2\in\cv_{N_2}}\!\sum_{k=0}^{\t_{\s_1}-1}\!1\{J^\circ_{N_1}(k)\notin T_1\}
\g_1^N(J^\circ_{N_1}(k))\tg_2^N(J^\circ_{N_1}(k)\s_2).
\end{eqnarray}
Finally,
\begin{eqnarray}\nn
 \E^\circ(\cu)&=&\sum_{\s'_1\notin T_1}\sum_{\s_2\in\cv_{N_2}}\g_1^N(\s'_1)\tg_2^N(\s'_1\s_2) \,
 E\!\left( \sum_{k=0}^{\t_{\s_1}-1}\!1\{J^\circ_{N_1}(k)=\s'_1\}\right) \\
 \label{oc5}
 &\leq&2\sum_{\s'_1\notin T_1}\sum_{\s_2\in\cv_{N_2}}\g_1^N(\s'_1)\tg_2^N(\s'_1\s_2),
\end{eqnarray}
where the inequality holds for all large enough $N$, according to Lemma~\ref{lem}.

It follows from~(\ref{g1ntog1}) and Remark~\ref{gtg} that 
\begin{equation}\label{oc6}
\limsup_{N\to\infty} \E^\circ(\cu)\leq2\sum_{\s'_1\notin T_1}\sum_{\s_2\in\N}\g_1(\s'_1)\g_2(\s'_1\s_2),
\end{equation}
and thus 
\begin{equation}\label{oc7}
\lim_{M_1\to\infty}\limsup_{N\to\infty} \E^\circ(\cu)=0.
\end{equation}

Let now $\cw_i$ denote the time spent by $\bxn_1$ outside $ T_1$ between the $i$-th and $i+1$-st
visit to $\s_1\in T_1$, $i\geq1$. A similar reasoning as above yields
\begin{eqnarray}\nn
\E^\circ(\cw_i)&=&\sum_{\s'_1\notin T_1}\sum_{\s_2\in\cv_{N_2}}\g_1^N(\s'_1)\tg_2^N(\s'_1\s_2) \,
E_{\s_1}\!\left( \sum_{k=0}^{\t_{\s_1}-1}\!1\{J^\circ_{N_1}(k)=\s'_1\}\right) \\
\label{oc8}
&=&\sum_{\s'_1\notin T_1}\sum_{\s_2\in\cv_{N_2}}\g_1^N(\s'_1)\tg_2^N(\s'_1\s_2),
\end{eqnarray}
where we have used~(\ref{e1}) in the last passage, and again
\begin{equation}\label{oc9}
\lim_{M_1\to\infty}\limsup_{N\to\infty} \E^\circ(\cw_i)=0,\,i\geq1.
\end{equation}

\subsubsection{Time inside $ T_1$ and outside $\cm_2$}\label{in1out2}

Let now
\begin{equation}
\label{oc10}
\check\Ups^N_i(\s_1)=c_1^N 2^{N_2}\sum_{j=0}^{\cg}1_{\{J^\circ_{N_2}(j)\notin\cm_2\}}\tg_2^N(\s_1J^\circ_{N_2}(j))\,T_j
\end{equation}
be the time spent outside $\cm$ by $\bxn_1$ on its $i$-th visit to $\s_1\in T_1$ --- recall the notation on 
the paragraph of~(\ref{u1}). A similar reasoning to that leading to~(\ref{oc5}) and~(\ref{oc8}) yields
\begin{equation}
\label{oc11}
\E^\circ(\check\Ups^N_i(\s_1))=
\g_1^N(\s'_1)\sum_{\eta\in T^{x_1'}}\tg_2^N(\eta),\,i\geq1
\end{equation}
and again 
\begin{equation}\label{oc12}
\lim_{M_2\to\infty}\limsup_{N\to\infty} \E^\circ(\check\Ups^N_i(\s_1))=0,\,i\geq1.
\end{equation}

As a corollary to~(\ref{oc12}) and~(\ref{u6}), we have that, recalling that $x_1=(\xi_1)^{-1}(\s_1)$,
\begin{equation}
\label{etsp}
\Ups^N_i(\s_1)\to \ff_3(x_1)\,\ce
\end{equation}
in distribution as $N\to\infty$, where $\Ups^N_i(\s_1)=\hat\Ups^N_i(\s_1)+\check\Ups^N_i(\s_1)$ is
the time spent by $\bxn_1$ on its $i$-th visit to $\s_1\in T_1$, and $\ce$ is a mean one exponential
random variable.

\subsection{Conclusion of the proof of Theorem~\ref{thm:cl}}
\label{conc}

Let us now fit together the above points in an argument for Theorem~\ref{thm:cl}.

From the first claim at the beginning of the section (argued in Subsection~\ref{tsp}), it is enough to show that
$\bxn_1$ restricted to $\cm_1$ converges to $\bar X_1$ restricted to $\cm_1$. We already know from~(\ref{etsp}) 
that the sojourn times of $\bxn_1$ on the various vertices of $\cm_1$ converge in distribution to the respective
sojourn times of $\bar X_1$. We only have to argue that the jump probabilities of $\bxn_1$ restricted to $\cm_1$
converge to the uniform jump probabilities of $\bar X_1$ restricted to $\cm_1$. 
But that is established in \eqv(C.inter.3a). 
We have then that $\bxn_1$ converges in distribution to $\bar X_1$ in Skorohod space,
and the full statement readily follows from the third point claimed at the beginning of the section 
(and argued in Subsection~\ref{sga}).

This concludes the proof of Theorem~\ref{thm:cl}.

\section{Proof of Theorem~\ref{thm:cm}}
\label{proofcm}

\setcounter{equation}{0}

We start by showing that $\bxn$ spends virtually all the time on $\cm$.

\subsection{Time spent by $\bxn$ outside $\cm$}\label{tspe}

We will show that the expected time spent by $\bxn$ outside $\cm$ until the 
first visit to $\cm$ or between consecutive visits to $\cm$ is small. 
Indeed, we will argue that the first such time (the others can be similarly treated)
vanishes in probability as $M_1,M_2\to\infty$ {\em uniformly} in $N$. We will be more 
precise next.

The first such time is bounded above by
\begin{equation}\label{te1}
\bcv=
\sum_{i=1}^{\ci}(\hcv_i+\tcv_i),
\end{equation}
where
\begin{equation}\label{te2}
\hcv_i=\sum_{k=\ttau^{i-1}+1}^{\tilde\t^i-1}
\sum_{j=0}^{\cg_{k+1}(J^{\circ}_{N_1}(k))}\tg_2^N(J^{\circ}_{N_1}(k)J^{\circ}_{N_2}(\ms_{k}+j))\,T^k_j,
\end{equation}
and, writing $\cx_1=(\xi_1)^{-1}(J^{\circ}_{N_1})$,
\begin{equation}\label{te3}
\tcv_i=
\sum_{j=0}^{\cg_{\tilde\t^i+1}(J^{\circ}_{N_1}(\tilde\t^i))}1_{\{J^{\circ}_{N_2}(\ms_{\tilde\t^i}+j)\notin T^{\cx_1}\}}\,
\tg_2^N(J^{\circ}_{N_1}(\tilde\t^i)J^{\circ}_{N_2}(\ms_{\tilde\t^i}+j))\,T^{\tilde\t^i}_j,
\end{equation}
where $\tilde\t^i$, $i=1,2,\ldots$ denote the successive hitting times of $ T_1$ by $J^{\circ}_{N_1}$
(the jump chain of $\bxn_1$) ---
when $i=0$, then in~(\ref{te2}) $\ttau_0$ is either $0$ or $-1$ depending on whether 
$\bxn_1(0)\in T_1$ or not, respectively.

We will show that
\begin{equation}\label{te4}
\lim_{M_1\to\infty}\lim_{M_2\to\infty}\limsup_{N\to\infty}\, \bcv=0
\end{equation}
in probability.

Given the tightness result for $\ci/[c_1^N2^{N_2}]$ given in Lemma~\ref{rm5a}, it is enough to show that
for all $R$
\begin{eqnarray}\label{te5}
&\lim_{M_1\to\infty}\lim_{M_2\to\infty}\limsup_{N\to\infty}\E^{\circ}\left( \sum_{i=1}^{Rc_1^N2^{N_2}}\hcv_i\right)=0,&\\\label{te6}
&\lim_{M_1\to\infty}\lim_{M_2\to\infty}\limsup_{N\to\infty}\E^{\circ}\left( \sum_{i=1}^{Rc_1^N2^{N_2}}\tcv_i\right)=0.& 
\end{eqnarray}

Let us first point out that for every $i\geq1$, $\hcv_i$ is bounded above stochastically by 
$\cu/[c_1^N 2^{N_2}]$ (see~(\ref{oc1}) above). (\ref{te5}) then follows from the above and~(\ref{oc9}).

Now
\begin{equation}\label{te7}
\E^{\circ}(\tcv_i)\leq\max_{x_1\in \cm_1}\E^{\circ}\left( \sum_{j=0}^{\cg(x_1)}
1_{\{J^{\circ}_{N_2}(j)\notin T^{x_1}\}}
\tg_2^N(x_1J^{\circ}_{N_2}(j))\right),
\end{equation}
where $\cg(x_1)$ is a geometric random variable with success parameter $q(x_1)$. We find that
the expectation on the right hand side equals 
$$\g_1^N(x_1)\sum_{x_2>M_2}\tg_2^N(x_1x_2)/[c_1^N 2^{N_2}],$$
and~(\ref{te6}) follows upon substitutions into~(\ref{te7}) and the left hand side of~(\ref{te6}).

\subsection{Conclusion of the proof of Theorem~\ref{thm:cm}}
\label{conccm}

Given the result in Subsection~\ref{tspe} and the usual {\em constancy interval matching} argument
that can be used to show convergence in Skorohod spaces, it is enough to show convergence of the 
transition probabilities among sites in $\cm$ (in the process restricted to $\cm$, which is a Markov jump process)
to the respective ones of the respective limit process (the one restricted to $\cm$, which is also a Markov jump process),
and the convergence of the respective sojourn times. The latter convergence is quite clear, and the former follows immediately 
from Proposition~\ref{paft}.

\section{Proof of Theorem~\ref{thm:cft}}
\label{proofcft}

\setcounter{equation}{0}

We start by observing that we can check that $\bxn$ spends virtually all the time on $\cm$ by virtually the same argument
as for below fine tuning. Indeed, the corresponding expressions in the present regime of (\ref{oc1}), (\ref{oc8}) and~(\ref{oc10}) are the same, 
except for the factor of $c_1^N2^{N_2}$, which is absent in the present regime. But notice that
that factor is bounded as $N\to\infty$, and so the arguments of Subsections \ref{out1} and \ref{in1out2} carry through.

We can then repeat the argument for the conclusion of the proof of Theorem~\ref{thm:cm} on Subsection~\ref{conccm}, once we have the convergence of the transition probabilities of $\bxn$ restricted to $\cm$ to those of the limiting 2-level K process restricted to $\cm$.

For $x,y\in\N$, let $\frp_N(x,y)$ denote the transition probability of $\bxn|_{\cm}$. Then, using the remark in 
Subsubsection~\ref{leaving} and Proposition~\thv(P.top).$i$-1 and $i$-2, we have that 
\begin{equation}\label{eqdif}
\lim_{N\to\infty}\frp_N(x,y)=\frp(x,y):=
\begin{cases}
\left[(1-\bar\l^{x_1})+\ \bar\l^{x_1}\,\bar\nu_1(x_1)\right]\frac{1}{M_2},&\mbox{if }x_1=y_1,\\
\,\,\bar\nu_1(y_1)\bar\l^{y_1}\frac{1}{M_2},&\mbox{otherwise},
\end{cases}
\end{equation}
where $\bar\l^{y_1}=\frac1{1+M_2\tg_1(y_1)}$ and $\bar\nu_1(y_1)=\frac{1-\bar\l^{y_1}}{\sum_{y_1'\in\cm_1}(1-\bar\l^{y_1'})}$.

It is then enough to argue that $\frp(x,y)$ is the transition probability from $x$ to $y$ for the 2-level K process restricted to $\cm$. We do that next.

Let $X|_{\cm}$ denote $X$ restricted to $\cm$. 
We can construct $X|_{\cm}$ as follows. Let $\ax$ denote the 1-level K process used in the construction of $X$ as at the end of Subsection~\ref{kp}. 
Let us now construct a 2-level process $\hat X$ in the same way as $X$, except that we use $\ax|_{\cm_1}$ instead of $\ax$. One readily checks that
\begin{enumerate}
	\item $\ax|_{\cm_1}$ is a Markov jump process on $\cm_1$ with uniform initial state, uniform transitions on $\cm_1$ (we should allow loops), and jump rate at $x_1\in\cm_1$ given by $1/\tg_1(x_1)$;
	\item $X|_{\cm}=\hat X|_{\cm}$;
	\item letting $\hat\cx_1$ denote the jump chain of $\ax|_{\cm_1}$, and, for $z\in\DD_2$, defining the events 
	\begin{eqnarray}
	A_n&=&\{\mbox{during the $n+1$-st sojourn period of }\hat X_1, \hat X_2 \mbox{ visits }\cm_2\},\\\nn
	A_n^z&=&\{\mbox{during the $n+1$-st sojourn period of }\hat X_1, \hat X_2 \mbox{ visits }z\\
	&&\hspace{2cm}\mbox{ before visiting }\cm_2\setminus\{z\}\},
    \end{eqnarray}
    we have that, given $\hat\cx_1$, the events $A_n^{\ast_n}$, $n\geq0$, with $\ast_n=y$ or blank for each $n$, are independent, having respective (conditional) probabilities given by 
    \begin{enumerate}
    	\item in the case of $\ast_n=$ blank: $P(N_{T}>0)$, where $N$ is a Poisson counting process with intensity $M_2$, and $T$ is exponential with mean $\tg_1(\hat\cx_1(n))$, $N$ and $T$ 
    independent; we then have that $N_T$ has a geometric distribution with success parameter $\frac1{1+M_2\tg_1(\hat\cx_1(n))}=\bar\l^{\hat\cx_1(n)}$, and thus $P(N_T>0)=1-\bar\l^{\hat\cx_1(n)}$;
    \item similarly, the probability of $A_n^y$ given $\hat\cx_1$ equals $\frac{1-\bar\l^{\hat\cx_1(n)}}{M_2}$.
    \end{enumerate}
\end{enumerate}	

Let $\frp'(x,y)$ denote the transition probability of $X|_{\cm}$ from $x$ to $y$. From the above, we conclude that 
if $x_1=y_1$, $\frp'(x,y)$ equals
\begin{eqnarray}\nn
&&P(A_0^{y_2})+\sum_{n\geq0}
\sum_{w_1,\ldots,w_{n-1}\in\cm_1}P(A_0^c,\ldots, A_{n-1}^cA_n^{y_2},\\\nn &&\hspace{2cm}\hat\cx_1(1)=w_1,\ldots,\hat\cx_1(n-1)=w_{n-1},\hat\cx_1(n)=x_1)\\\label{tp1}
&=&\frac{1-\bar\l^{x_1}}{M_2}+\bar\l^{x_1}\frac1{M_1}\frac{1-\bar\l^{x_1}}{M_2}
    \sum_{n\geq0}\left(\frac1{M_1}\sum_{w_1\in\cm_1}\bar\l^{x_1}\right)^n,
\end{eqnarray}
which is readily checked to equal $\frp(x,y)$ given in~(\ref{eqdif}), in this case. When $x_1\ne y_1$, we have that the same
expression as in~(\ref{tp1}) holds for $\frp'(x,y)$, except for the first term in the sum, which is absent, and thus it agrees
with $\frp(x,y)$ again in this case. The argument is complete.

\section{Aging in the K processes}
\label{age}

\setcounter{equation}{0}

As anticipated in the introduction, we will derive aging results for $\s^N$ in a two-stage scaling limit process. We first take the limit in the extreme time scale, where there is no aging, since $\s^N$ is close to equilibrium in that time scale: we have already done that in our scaling limit theorems. In the second stage, we take a small time limit of the limiting K processes. We will be concerned with correlation functions which involve only clock processes of the limiting processes, so we will take the second limit only of the relevant clocks. We will keep the presentation brief, in particular at and below fine tuning, since the issues involved are quite clear in those regimes, and technicalities are quite well known and fairly straightforward.

Let $Y=Y_1Y_2$ be the K process representing the scaling limit of either $\tilde X^N$ or $\bar X^N$ in 
Theorems~\ref{thm:cm},~\ref{thm:cft} and~\ref{thm:cl}. We assume $Y(0)=\infty$ in the first case, $\infty\infty$ in the second case, and $(\infty,\bar X_2)$ in the latter case, where $\bar X_2$ is as in Theorem~\ref{thm:cl}. Given $\theta>0$, we are interested in taking the following limit 
\begin{equation}\label{age1}
\lim_{t_w,t\to0\atop{t/t_w\to\theta}}\Pi(t_w,t_w+t),
\end{equation}
where $\Pi(\cdot,\cdot)$ was defined in~\ref{overlap}. Recall also $\fn_i$ defined in~\ref{ni}.

\subsection{Below fine tuning}

This is the simplest case, since on the one hand, $Y_2$ jumps at every time interval, so there is no point in considering $\fn_2$ 
(which has probability $0$). On the other hand, $Y_1$ is a uniform K process with waiting function given by a Poisson process with 
intensity measure $\frac{c_1}{x^{\a_1+1}}dx$ for some constant $c_1$. It is well known (see e.g.,~\cite{5aut}) that the limit in~(\ref{age1}) is given by the
arcsine law $\asl_{\a_1}(1/(1+\theta))$, where
\begin{equation}\label{asin1}
\asl_{\a}(u)=\frac{\sin\pi\a}{\pi}\int_0^u x^{-(1-\a)}(1-x)^{-\a}dx,\,u,\a\in(0,1).
\end{equation}
This follows readily from the scaling limit of the clock process of $Y_1$ at small times. This issue will come up again in the other temperature regimes, so we let it rest for this regime.

\subsection{At fine tuning} We first notice that, for $i=1,2$,
$\fn_i=\{\cR_i\cap(t_w,t_w+t)=\emptyset\}$,
where $\cR_i$, $i=1,2$, is the range of the clock processes $\G_1$ and $\G'$, respectively (see end of Subsection~\ref{kp} below).
It is a simple matter to check that in this case $\fn_2\subset\fn_1$, so indeed $\Pi(t_w,t_w+t)=pP(\fn_1)+(1-p)P(\fn_2)$.
Let us now point out that, as is well known,~(\ref{age1}) follows from a small time scaling limit for the
respective clocks in an appropriate topology. Let us first examine $\G'$ -- recall the definition at the end of Subsection~\ref{kp}, 
and the statement of Theorem~\ref{thm:cft}. As pointed out at the end of Subsection~\ref{kp}, within intervals of constancy of
$\ax$, the increments of $\G'$ are those of a uniform K process with waiting function $\frf(x_1,\cdot)$, where $x_1$ is the constant
value of $\ax$ within the interval. We know that the clock process of a uniform K process with waiting function given by a Poisson
process with intensity measure $\frac c{x^{\a+1}}dx$, $\a\in(0,1)$, $c$ any constant in $(0,\infty)$, converges in the $J_1$ Skorohod metric to an $\a$-stable subordinator
in the small time limit for almost every realization of the Poisson process (see e.g.,~\cite{5aut} with $a=0$). This and the fact that $\frf'_2(x_1,\cdot)$ are given by iid in $x_1$ Poisson processes with intensity measure $\frac{c_2}{x^{\a_2+1}}dx$, for some constant $c_2$, yields
\begin{equation}\label{cl1}
\ve^{-1}\G'(\ve^{\a_2}\times\cdot)\to S_2(\cdot)
\end{equation}
in distribution on the $J_1$ Skorohod space as $\ve\to0$ for a.e.~realization of $\frf_2,\frf'_2$, where $S_2$ is an $\a_2$-stable subordinator. It readily follows that almost surely
\begin{equation}\label{age2}
\lim_{t_w,t\to0\atop{t/t_w\to\theta}}P(\cR_2\cap(t_w,t_w+t)=\emptyset)= \asl_{\a_2}(1/(1+\theta)).
\end{equation}
It can be also readily checked that
\begin{equation}\label{cl2}
\ve^{-1}\G_1(\ve^{\a_1\a_2}\times\cdot)\to S'_1:=S_1\circ S_2(\cdot)
\end{equation}
in distribution on the $J_1$ Skorohod space as $\ve\to0$ for a.e.~realization of $\frf_2,\frf'_2$, where $S_1$ is an $\a_1$-stable subordinator. $S'_1$ is thus an $\a_1\a_2$-stable subordinator, and it follows that almost surely
\begin{equation}\label{age3}
\lim_{t_w,t\to0\atop{t/t_w\to\theta}}P(\cR_1\cap(t_w,t_w+t)=\emptyset)= \asl_{\a_1\a_2}(1/(1+\theta)).
\end{equation}

\subsection{Above fine tuning}
We first point out that right after the weighted K process $Y$ jumps out of any state in $\N^2$, it gives infinitely many jumps
within any nonempty open time interval. This tells us that $\fn_1=\fn_2=\{\cR\cap(t_w,t_w+t)=\emptyset\}$,
where $\cR$ is the range of $\G$, the clock process of $Y$ (see Subsection~\ref{kp} below).

Let us then derive the small time limit of $\G$. Let us recall from~\cite{FP} that we may write 
\begin{equation}\label{claft}
\G(r)=\sum_{x\in\N^2}\sum_{i=1}^{N_x(r)}\g_2(x)T_i^x,
\end{equation}
where $N_x,\,x\in\N^2,$ are independent Poisson counting processes with rate $\g_1(x_1)$, respectively,
and $T_\cdot^\cdot$ are iid mean one exponential random variables. We will now argue that
\begin{equation}\label{cl3}
\ve^{-1}\G(\ve^{\a_2}\times\cdot)\to S_2(\cdot)
\end{equation}
in distribution on the $J_1$ Skorohod space as $\ve\to0$ for a.e.~realization of $\g_1,\g_2$, where $S_2$ is an $\a_2$-stable subordinator. 
It is enough to establish this convergence for 
\begin{equation}\label{cl4}
\tilde\G(r):=\sum_{x\in\N^2}\g_2(x)N_x(r)
\end{equation}
(see Lemma 2.1 in~\cite{5aut}). Since this is a subordinator, in order to get the small time convergence, it is enough to
consider the Laplace exponent of the small time scaled $\tilde\G$, given by
\begin{equation}\label{le1}
\tv_\ve(\l)=r\sum_{x_1\in\N}\g_1(x_1)\,\ve^{\a_2}\sum_{x_2\in\N}(1-e^{\l\ve^{-1}\g_2(x)}).
\end{equation}
It is convenient at this point to write the scaled inner sum as
\begin{equation}\label{le2}
\ve^{\a_2}\sum_{u\in[0,1]}(1-e^{\l\ve^{-1}\g^{x_1}_2(u)}),
\end{equation}
where $\g^{x_1}_2(\cdot)$, $x_1\in\N$, are iid sets of increments of an $\a_2$-stable subordinator in $[0,1]$.
By the scale invariance of that subordinator, we have that it equals in distribution
\begin{equation}\label{le3}
\ve^{\a_2}\sum_{i=1}^{\lfloor\l^{\a_2}\ve^{-\a_2}\rfloor}\sum_{u\in[0,1]}(1-e^{\l\ve^{-1}\g^{i}_2(u)})
\end{equation}
plus an independent random variable which is stochastically dominated by 
$$\ve^{\a_2}\sum_{u\in[0,1]}(1-e^{\l\ve^{-1}\g^{1}_2(u)}).$$
Now standard large deviation estimates coupled with a straightforward application of Campbell's Theorem
showing that $Z:=\sum_{u\in[0,1]}(1-e^{\l\ve^{-1}\g^{1}_2(u)})$ has an exponential moment imply that
$\tv_\ve(\l)\to {\rm const}\times r\l^{\a_2}$, where const $=E(Z)\sum_{x_1\in\N}\g_1(x_1)$. 
Given $\g_1$, this is the Laplace exponent of an $\a_2$-stable subordinator, and~(\ref{cl3}) follows.
Thus
\begin{equation}\label{age4}
\lim_{t_w,t\to0\atop{t/t_w\to\theta}}\Pi(t_w,t_w+t)= \asl_{\a_2}(1/(1+\theta)).
\end{equation}

\subsection{Final remark about aging at $\beta>\beta_2^{cr}$} 

We end this section on aging by pointing out, as anticipated in the introduction, that the aging results in~\cite{SN00}
are consistent with ours only in the fine tuning regime. As also earlier anticipated, this is explained by the shorter
time scale considered in that reference. In those time scales, all levels are supposed to be aging simultaneously. That 
indeed also happens in the fine tuning regime at the {\em short} extreme scale considered in this section. 
Recall the discussion on Subsubsection~\ref{intermediate}.

We may understand the aging results in the other regimes treated in detail so far as follows. 
Below fine tuning, we already explained that the second level is well within equilibrium, 
so it does not age in the short extreme time scale (also in not much shorter time scales). 
The aging behavior in that phase comes from the first level, with its characteristic $\a_1$ exponent.

Above fine tuning, we have the opposite behavior: the first level by itself would be in equilibrium, thus not aging, and
aging comes from the second level, with its characteristic $\a_2$ exponent.

\section{Scaling limit at intermediate temperatures}\label{inter}

In this section, we briefly state and discuss our scaling limit and aging results for $\beta\in(\beta_1^{cr},\beta_2^{cr})$. 
We will be rather sketchy, trusting the experienced reader to be able to readily fill in the gaps with standard arguments.
Recall from the discussion around~(\ref{v1}) that the total time spent by $\s^N(\cdot)$ on a single visit to a given $\s_1\in\VV_{N_1}$ can be written as
\begin{equation}\label{v1a}
\frac N{N_2+{N_1}e^{\b \sqrt{aN}\Xi^{(1)}_{\s_1}}}\sum_{j=0}^{\cg}\exp\{-\b\sqrt{(1-a)N}\, \Xi^{(2)}_{\s_1J^\circ_{N_2}(j)}\}\,T_j.
\end{equation}
where again
$\cg$ is a geometric random variable with success probability $q^*_N(\s_1)$ given by~\eqv(01.14), and $T_0,T_1,\ldots$ are iid standard exponential random variables.

For top first level configurations $\s_1=\xi_1^{x_1}$, the factor in front of the sum above contributes a constant ($1/(1-p)$) almost surely in the limit as $N\to\infty$, and thus we are left to properly scale the sum itself. At this point we may replace $\cg$ by 
$\frac{1-p}p\frac1{c_1^N}\g_1(x_1)T$, with $T$ a standard exponential, independent of all the other remaining random variables,
and then resort to Lemma 10.8 of (the arxiv version of) \cite{G15} which gives the proper scaling of the sum, as well as conditions under which the scaled sum satisfies a law of large numbers. Since the result in  \cite{G15} applies for the REM (which is indeed the model appearing in the above sum), we need to do some translation in terms of our parameters. Upon doing that, we find that the proper scaling is given by 
$\tilde c_N=c_1^N\exp\{-\frac{\b^2N}2(1-a)\}$, and, provided $\b<2\frac{\sqrt{ap}}{1-a}\b_*$, the following law of large numbers holds:
\begin{equation}\label{v1b}
\tilde c_N\sum_{j=0}^{t/c_1^N}\exp\{-\b\sqrt{(1-a)N}\, \Xi^{(2)}_{\s_1J^\circ_{N_2}(j)}\}\,T_j\to t
\end{equation}
as $N\to\infty$ in probability for each $t>0$. 
It follows that the sum in~(\ref{v1a}) scaled by $\tilde c_N$ converges in distribution as $N\to\infty$
to $\hat\g_1(x_1)T$, where $\hat\g_1(\cdot):=\frac1p\g_1(\cdot)$. Given also that the transition among top first level configurations is asymptotically uniform, we find that the asymptotic motion among the top first level configurations is consistent with a (uniform) K process. We can state the following result.

\begin{theorem}[Intermediate temperatures]
	\label{thm:inter}
	If $\b\in(\b_1^{cr},\b_2^{cr})$, and provided also that $\b<\b_{int}:=2\frac{\sqrt{ap}}{1-a}\b_*$, we have that as $N\to\infty$ 
	\be\label{int}
	X_1^N(\cdot/\tilde c^N)\Rightarrow \tilde X_1(\cdot)
	\ee
	where $\tilde X_{1}\sim\ck(\hat\ff,1)$ starting at $\infty$, with $\hat\ff:\N\to(0,\infty)$, $\hat\ff(x_1)=\hat\g_1(x_1)$.
\end{theorem}

We have indeed given all of the main ingredients of the proof above, except for an estimate establishing that 
$X_1^N(\cdot/\tilde c^N)$ spends virtually all of its time on the first level top configurations. This can be done as
in the proofs of the Theorems~\ref{thm:cm},~\ref{thm:cft} and~\ref{thm:cl}, by a first moment estimate. For that to work,
we resort to a further known result, namely that the first level marginal of the Gibbs measure converges to the normalization of
$\hat\ff$. This is given by Theorem 9.2 of~\cite{BS02}; indeed that result is stated for intermediate temperatures for the case
where $p=1/2$, but one may readily check that it holds in general. 

In order that the conditions on $\beta$ above are not empty, we need of course $\b_{int}>\b^{cr}_1$, which is equivalent to $a>1/3$;
in this case, it may or may not happen that $\b_{int}<\b^{cr}_2$; in the former case,~(\ref{int}) holds only in a (nonempty) subinterval of $(\b_1^{cr},\b_2^{cr})$, namely $(\b_1^{cr},\b_{int})$; otherwise, it holds in the full intermediate interval.
Perhaps interestingly, the latter case makes $\b_{FT}<\b_2^{cr}$, and then, as pointed out above -- see Remark~\ref{nocond1} --, 
we are effectivey in the {\em above fine tuning} regime when $\b>\b_2^{cr}$; Theorems~\ref{thm:cl} and~\ref{thm:inter} then tell
us that the scaling limit of $X_1$ is distributed as essentially the same $K$ process in $(\b_1^{cr},\b_2^{cr})$ or above $\b_2^{cr}$.

If $\b_{int}<\b<\b^{cr}_2$, then Theorem 1.1 in~\cite{G15} tells us that the left hand side of~(\ref{v1b}) converges weakly to a stable subordinator instead. This signals aging, and thus the time scale is not extreme. 
This clarifies a point raised in Subsubsection~\ref{intermediate}.

Under the conditions of the above theorem, the following aging result readily follows in the same way as in the previous section 
\begin{equation}\label{age5}
\lim_{t_w,t\to0\atop{t/t_w\to\theta}}\tilde\Pi_1(t_w,t_w+t)=
\asl_{\a_1}(1/(1+\theta)),
\end{equation}
where $\tilde\Pi_1(t_w,t_w+t)$ is the probability that $\tilde X_1$ gives no jump within $(t_w,t_w+t)$.

\section{Appendix.}\label{app} 

\subsection{K-processes}\label{kp}

Let $\md$ be a countably infinite set, and let $\infty$ denote a point not in $\md$, 
and make $\bar\md=\md\cup\{\infty\}$. Let $\ff,\ww:\md\to(0,\infty)$ be such 
that 
\begin{equation}
\label{eq:sumff}
\sum_{x\in\md}\ww(x)=\infty,\quad\sum_{x\in\md}\ww(x)\ff(x)<\infty.
\end{equation} 

Consider $\{\cc_{x},\,x\in\md\}$, an independent family of Poisson counting processes
such that $\cc_{x}$ has intensity  $\ww_{x}$  for each $x\in\md$, 
with associated point processes $\cs=\{(\theta_{x}(i),\,i\geq1),\,x\in\md\}$ 
(the {\em event times} of the respective counting processes). Let $\o:\R^+\to\bar\md$ 
be such that $\o(\theta_{x}(i))=x$ for $x\in\md$, $i\geq1$, and $\o(s)=\infty$ if $s\notin\cs$. 
We note that $\o$ is well defined almost surely. 
Let $\{T_s,\,s\in\R^+\}$ be an iid family of mean 1 exponential random variables.
Let now $\nu$ be an atomic measure on $\R^+$ concentrated on $\cs$ as follows
\begin{equation}
\label{mu}
\nu(\{s\})=\ff(\o(s))\,T_s,\, s\in\cs,
\end{equation}
and let $\G$ be its distribution function, namely, $\G:\R^+\to\R^+$ is such that
\begin{equation}
\label{G}
\G(r)=\nu([0,r]),
\end{equation}
and let $\varphi$ be the right continuous inverse of $\G$. 
Then for $t\geq0$ let
\begin{equation}
\label{x1}
X(t):= \o(\varphi(t)).
\end{equation}

We call $X$ thus defined a K-process on $\bar\md$ with {\em waiting time} function $\ff$, and {\em weight} function $\ww$, starting at $\infty$. Notation: $X\sim\ck(\ff,\ww)$.
In the particular case where $\ww\equiv1$, we call $X$ a {\em uniform} K-process on $\bar\md$
with waiting time function $\ff$, and use the notation $X\sim\ck(\ff,1)$.
Also, we call $\G$ in~(\ref{G}) the clock process of $X$.
\medskip

We next define 2-level K-processes, as follows. Let $\ax$ be
a {\em uniform} K-process on $\bar\md$ with ($\ww\equiv1$ and) $\ff$ as above
such that~(\ref{eq:sumff}) is satisfied.
Let $\md'$ be a countably infinite set and as before
make $\bar\md'=\md'\cup\{\infty\}$. Let $\ff':\md\times\md'\to(0,\infty)$ be such that
\begin{equation}
\label{eq:sumffp}
\sum_{xy\in\md\times\md'}\ff(x)\ff'(xy)<\infty.
\end{equation} 

Let $\{\cc'_{x},\,x\in\md'\}$  be an iid family of intensity 1 Poisson counting processes,
independent of $\{\cc_{x},\,x\in\md\}$,
with associated point processes $\cs'=\{(\theta_{x}'(i),\,i\geq1),\,x\in\md'\}$.

Let $\o':\R^+\to\bar\md'$ 
be such that $\o'(\theta'_{x}(i))=x$ for $x\in\md'$, $i\geq1$, and $\o'(s)=\infty$ if $s\notin\cs'$. 
We note that $\o'$ is well defined almost surely.

Let $\nu'$ be an atomic measure on $\R^+$ concentrated on $\cs'$ as follows.
\begin{equation}
\label{mup}
\nu'(\{s\})=\ff'(\ax(s)\,\o'(s))\,T_s,\, s\in\cs',
\end{equation}
and let $\G'$ be its distribution function, and let $\varphi'$ be the right continuous inverse of $\G'$.
Then for $t\geq0$ let
\begin{equation}
\label{x}
X(t)=X_1(t)X_2(t):=
\ax(\varphi'(t))\,\o'(\varphi'(t)).
\end{equation}

We call $X=(X(t),\,t\geq0)$ a 2-level K-process (starting at $(\infty,\infty)$) 
on $\bar\md\times\bar\md'$
with waiting time functions $\ff,\ff'$. Notation: $X\sim\ck_2(\ff,\ff')$. 
This process was introduced in \cite{FGG} 
(with $\md=\md'=\N=\{1,2,\ldots\}$), 
where some of its properties and those of a finite volume
version were studied (and where illustrations of their construction can be found). 
It may be understood as two 1-level uniform K-processes arranged in hierarchies as follows. 

Given the realization of the (1-level) K-process $\ax$, let $\ci$ be the set of maximal 
intervals of constancy of $\ax$ (maximal time intervals where $\ax$ is constant), 
the second step of the above description amounts to 
constructing within each such interval, say $[a,b)$, a 1-level uniform K-process with waiting
time  function $\ff'(x,\cdot)$, where $x$ is the constant value of $\ax$ within that interval. 
This results in what can be seen as an excursion of a 2-level K-process $X=X_1X_2$ with 
$X_1\equiv x$. 
This excursion takes place within the time interval $[a', b')$, 
with $a'=\G'(a)$, $b'=\G'(b)$.
Outside the union of all such intervals, $X\equiv(\infty,\infty)$. 

We call $\G'$ the clock processes of $X$. We also call $\G_1:=\acute\G\circ\G'$ the 
clock process of $X_1$, where $\acute\G$ is the clock process of $\ax$.

\subsection{Auxiliary results for the proof of Proposition~\ref{paft}}

Let us fix $\eps,\eps_0,\eps_1>0$ satisfying the conditions of  Lemma \thv(L2.ls) and  Lemma \thv(rm2), and such that $\eps_0<\eps$. Let us recall the definition of the event $C_N$ from the paragraph containing \eqv(bin) and
 let $B_N$ be the event that  for all $\eta, \bar\eta\in T_1$, $\eta\neq\bar\eta$, \eqv(ls.15) holds.

\begin{lemma}\label{rm3}
	Let $\mj_2$ be the number of jumps of $\bxn_2$ till $\bxn_1$ reaches $\xi_1^{-1}(T_1)$.
	Given $\s_1\in\VV_{N_1}$ such that $d_1(\s_1, T_1)>\eps_0N_1$, then provided $B_N$ and $C_N$ occur, and $\bxn_1$ starts from $\xi_1^{-1}(\s_1)$, we have that $\P_{\s_1}(\mj_2\leq N^3)$ vanishes exponentially fast as $N\to\infty$, i.e., there exists $R>1$ such that
	\begin{equation}\label{bd1}
	\P_{\s_1}(\mj_2< N^3)\leq R^{-N}.
	\end{equation}
\end{lemma}

\begin{proof}
Let $\mj_1$ be the number of jumps of $\bxn_1$ till $\bxn_1$ reaches $\xi_1^{-1}(T_1)$. Then 
\begin{equation}
\label{jj1}
\mj_2\stackrel{d}{=}\sum_{j=0}^{\mj_1-1}\cg_j(J^\circ_{N_1}(j))
\geq\sum_{j=0}^{\mj_1-1}\cg_j(J^\circ_{N_1}(j))\,1_{\left\{d_1\left(J^\circ_{N_1}(j), T_1\right)\leq\eps_0N_1\right\}},
\end{equation} 
where $J^\circ_{N_1}$ is the jump chain of $\bxn_1$, namely, a simple symmetric discrete time random walk on $\VV_{N_1}$,
and $\cg_j(\s'_1)$, $j\geq0,\,\s'_1\in\VV_{N_1}$, are independent geometric random variables with success parameter
$q'_N:=re^{-\sqrt{N}}\wedge1$. The right hand side of~(\ref{jj1}) may be bounded stochastically from below as follows.
Notice that since $B_N$ occurs, the chain $d_1(J^\circ_{N_1}(\cdot), T_1)$ observed only when $J^\circ_{N_1}$ is at a distance at 
most $\eps_0N_1$ from $T_1$ may be identified (in distribution) to a Markov chain $\cz$ in $\{0,1,\ldots,\eps_0N\}$ 
which has the same transition probabilities as $\cy:=d'(J^\circ_N(\cdot),O)$, where $J^\circ_N$ is a simple symmetric random 
walk on a hypercube of dimension $N$, and $O$ is a given site of such a hypercube, except that $\cz$ is {\em lazy} at $\eps_0N$ --- the jumps of $\cy$ to the right of $\eps_0N$ are replaced by self jumps at $\eps_0N$ of $\cz$. 
Also, in this case, $\cz$ starts  from $\eps_0N$. 
Since $C_N$ occurs, every time $J^\circ_{N_1}$ gives a jump within 
distance at most $\eps_0N$ from $ T_1$, independent af all else, there is an at least $\eps_1$ probability it will land on a 
site $\s_1'\in\VV_{N_1}$ such that $q^*_N(\s_1')\leq q'_N$. The above justifies 
bounding stochastically the right hand side of~(\ref{jj1}) from below by
\begin{equation}
\label{jj2}
\mj_2':=\sum_{j=0}^{\mj_1'-1}\cg'_j\geq N^3\sum_{j=0}^{\mj_1'-1}1_{\{\cg'_j\geq N^3\}}
\geq N^3\sum_{j=0}^{\eps_0N-1}1_{\{\cg'_j\geq N^3\}},
\end{equation} 
where $\mj'_1$ is the number of jumps $\cy$ takes to reach $0$ starting from $\eps_0N$, and $\cg'_j$, 
$j=0,1,\ldots$, are iid random variables, independent of $\mj'_1$; $\cg'_0$ is a mixture of two 
random variables, the first of which, with weight $1-\eps_1$ is $0$, and the other, with weight $\eps_1$,
is geometric with success parameter $q'_N$.

Now, $P(\cg'_0\geq N^3)=\eps_0(1-q'_N)^{N^3}\geq\eps_0''>0$ for all large enough $N$, and~(\ref{bd1}) follows
with $R=(1-\eps_0'')^{-\eps_0}$. 
\end{proof}

\begin{remark}\label{rm4}
	From Theorem 3.1 of~\cite{BG08} 
	starting from any $\s_i\in\VV_{N_i}$, the probability that
	$\bxn_i$ does not go a distance $\frac12 N_i$ from $\s_i$ 
	before it returns to $\s_i$ is bounded above by  
	$1/N+4/N^2$.
\end{remark}

\begin{lemma}\label{rm5a}
        For ${\ci}$ as in \eqv(rel1) we have that
	\begin{equation}
	\label{ci0}
	\lim_{R\to\infty}\limsup_{N\to\infty}\P({\ci}\geq Rc_1^N2^{N_2})\to0.
	\end{equation}
\end{lemma}

\begin{proof}
We may stochastically bound ${\ci}$ from below by a geometric random variable with success parameter
$
\max_{x_1\in\MM_1}\max_{\s\in W^{x_1}}\P_{\s}(\cg^*>\T^*)
$
where $\T^*=\t_{\{\s'_2\in\pi_2T^{x_1} : d_2(\s'_2, \s_2)=N_2\}}$ and $\cg^*$ is a geometric random variable with success parameter $q^*_N(\xi_1^{x_1})$, independent of $\T^*$. Notice that the distribution of $\cg^*$ only depends on $\s$ through 
$x_1$, and that the distribution of $\T^*$ is independent of $\s$. Therefore we can drop the maximum over $\s\in W^{x_1}$ in the above formula. Furthermore the maximum over $x_1$ is achieved at $x_1=M_1$.
We then have
\begin{equation}
\label{ci2}
\max_{x_1\in\MM_1}\P(\cg^*>\T^*)=\E^{\circ,2}[(1-q^*_N(\xi_1^{M_1}))^{\T^*}]
=:\rho_N
\geq \frac{\psi_N\gamma^N_1(\xi_1^{M_1})}{1+\psi_N\gamma^N_1(\xi_1^{M_1})}(1+o(1)),
\end{equation}
where the inequality follows from \eqv(P1.lev2.2) and \eqv(L1new.lev2.1) (with $A=\{\xi_1^{x_1}\}$ and where in \eqv(P1.lev2.2'), $d_2=1$ and $F_{N_2,d_2}$ is absent).
The probability on the left hand side of~(\ref{ci0}) is  bounded above by
$(1-\rho_N)^{Rc_1^N2^{N_2}}$.
Since $\rho_N\to0$ as $N\to\infty$ and
\begin{equation}
\label{ci6}
\rho_Nc_1^N2^{N_2}\gtrsim\mbox{const }
\frac{N_2}{N_1}\,\g_1(M_1),
\end{equation}
indeed bounded away from zero as $N\to\infty$, the result follows.
\end{proof}

\begin{lemma}\label{lm:pis}
	(\ref{pis}) holds.
\end{lemma}

\begin{proof}[Proof of Lemma \thv(lm:pis)]
We start by computing $\pi(x)$. Let $\cg$ denote the number of jumps of $\bxn_2$ before $\bxn_1$ leaves $x$.
$\cg$ is a geometric random variable with success parameter $q^*_N(\xi_1^{x})$, independent of $\t_{\s_2}$.
Then, 
setting $\l'_N=\frac{N_1}2\frac{c_1^N}{\g_1^N(x)}$, and applying~(\ref{fr1},\ref{fr2}),
we get
\begin{eqnarray}\nn
\pi(x)&=&\frac1{2^{N_2}}\sum_{\s'_2\in\VV_{N_2}} \P_{\s'_2}(\cg\geq\t_{\s_2})
=\frac1{2^{N_2}}\sum_{\s'_2\in\VV_{N_2}} \E^{\circ,2}_{\s'_2}[(1-q^*_N(\xi_1^{x}))^{\t_{\s_2}}]\\\nn
&=&\frac1{B_0(\l'_N)}\frac1{2^{N_2}}\sum_{i=0}^{N_2}\sum_{\s'_2\in\VV_{N_2}:d(\s_2,\s_2')=i}\!\!\!\!\!\! B_i(\l_N')
=\frac1{B_0(\l'_N)}\frac1{2^{N_2}}\sum_{i=0}^{N_2}{N_2\choose i} B_i(\l_N')
\\\nn
&=&\frac1{B_0(\l'_N)}\int_0^1u^{\l_N'-1}du\frac1{2^{N_2}}\sum_{i=0}^{N_2}{N_2\choose i} (1-u)^i(1+u)^{N_2-i}
\\\label{ap6}
&=&\frac1{B_0(\l'_N)}\int_0^1u^{\l_N'-1}du=\frac1{\l'_NB_0(\l'_N)}
=\frac1{1+\l'_N\sum_{i=1}^{N_2}{N_2\choose i}\frac1{i+\l_N'}},
\end{eqnarray}
and the first claim of~(\ref{pis}) follows upon noticing that the sum in the denominator on the 
right hand side of~(\ref{ap6}) is $\sim 2^{N_2+1}/N_2$.

As for the second claim, we write
\begin{equation}\nn
\hat\pi(x_1)=\P(\cup_{x_2\in\cm_2}H_{x_2}|\bxn_1(0)=x_1),
\end{equation}
where $H_{x_2}$ is the event that $\bxn_2$ hits $x_2$ before the first jump of $\bxn_1$. 
By the Bonferroni inequalities, we have that
\begin{equation}\label{bd}
0\leq\!\!\sum_{x_2\in\cm_2}\!\!\P(H_{x_2}|\bxn_1(0)=x_1)-\hat\pi(x_1)\leq\!\!\!\!
\sum_{x_2,x_2'\in\cm_2\atop{x_2\ne x_2'}}\!\!\!\!
\P(H_{x_2}\cap H_{x'_2}|\bxn_1(0)=x_1).
\end{equation}
Since the summands on the central expression above are identically equal to $\pi(x_1)$, and the expression 
on the right hand side equals
\begin{equation}\nn
\pi(x_1)\sum_{x_2,x_2'\in\cm_2\atop{x_2\ne x_2'}} \P(H_{x'_2}|\bxn_1(0)=x_1,H_{x_2}),
\end{equation}
it is enough to argue that each summand in the expression above is an $o(1)$. But, given Lemma \thv(L2.ls) above,
each such summand is, apart form an $o(1)$ error, the probability that, starting from the origin, an Ehrenfest 
chain on $\{0,\ldots,N_2\}$ passes by $bN_2$, with $b=1/3$, before an independent time which is geometrically 
distributed with success probability $q^*_N(x_1)$. 
Writing that probability as a moment generating function as above (see e.g.~the first equality in~(\ref{ci2})),
and applying~(\ref{fr1},\ref{fr2}), we have that that equals
\begin{equation}\label{app2}
\frac1{B_0(\l'_N)}\sum_{j=0}^{\bar b N_2}{\bar b N_2\choose j}\frac{\G(bN_2+1)\G(j+\l_N')}{\G(bN_2+1+j+\l_N')},
\end{equation}
where 
$\bar b=1-b$. The quotient inside the latter sum is bounded above by 1, 
and thus~(\ref{app2}) is bounded above by
\begin{equation}\label{app3}
\frac1{\l'_NB_0(\l'_N)}+\frac1{B_0(\l'_N)}\sum_{j=1}^{\bar b N_2}{\bar b N_2\choose j}.
\end{equation}
As we saw above the first term of this sum is 
$\sim\g_1(x_1)\frac{N_2}{N_1}\frac1{c_1^N2^{N_2}}$, which is an $o(1)$. The second term is readily
checked to also be an $o(1)$, and  
the claim is established.
\end{proof}

\bibliographystyle {abbrv}      
%

\begin{thebibliography}{10}

\bibitem{GJ17}
G.~B. Arous and A.~Jagannath.
\newblock Spectral gap estimates in mean field spin glasses.
\newblock preprint, 2017.
\newblock arXiv:1705.04243.

\bibitem{BBC08}
G.~Ben~Arous, A.~Bovier, and J.~{\v{C}}ern{\'y}.
\newblock Universality of the {REM} for dynamics of mean-field spin glasses.
\newblock {\em Comm. Math. Phys.}, 282(3):663--695, 2008.

\bibitem{BBG03a}
G.~Ben~Arous, A.~Bovier, and V.~Gayrard.
\newblock Glauber dynamics of the random energy model. {I}. {M}etastable motion
  on the extreme states.
\newblock {\em Commun. Math. Phys.}, 235(3):379--425, 2003.

\bibitem{BBG03b}
G.~Ben~Arous, A.~Bovier, and V.~Gayrard.
\newblock Glauber dynamics of the random energy model. {II}. {A}ging below the
  critical temperature.
\newblock {\em Commun. Math. Phys.}, 236(1):1--54, 2003.

\bibitem{BC06b}
G.~Ben~Arous and J.~{\v{C}}ern{\'y}.
\newblock The arcsine law as a universal aging scheme for trap models.
\newblock {\em Comm. Pure Appl. Math.}, 61(3):289--329, 2008.

\bibitem{BG08}
G.~Ben~Arous and V.~Gayrard.
\newblock Elementary potential theory on the hypercube.
\newblock {\em Electron. J. Probab.}, 13:no. 59, 1726--1807, 2008.

\bibitem{BAGun11}
G.~Ben~Arous and O.~G\"un.
\newblock Universality and extremal aging for dynamics of spin glasses on
  subexponential time scales.
\newblock {\em Commun. Pure Appl. Math.}, pages 77--127, 2012.

\bibitem{5aut}
S.~C. Bezerra, L.~R.~G. Fontes, R.~J. Gava, V.~Gayrard, and P.~Mathieu.
\newblock Scaling limits and aging for asymmetric trap models on the complete
  graph and {$K$} processes.
\newblock {\em ALEA Lat. Am. J. Probab. Math. Stat.}, 9(2):303--321, 2012.

\bibitem{BB07}
E.~Bolthausen and A.~Bovier, editors.
\newblock {\em Spin glasses}, volume 1900 of {\em Lecture Notes in
  Mathematics}.
\newblock Springer, Berlin, 2007.

\bibitem{BS02}
E.~Bolthausen and A.-S. Sznitman.
\newblock {\em Ten lectures on random media}, volume~32 of {\em DMV Seminar}.
\newblock Birkh\"auser Verlag, Basel, 2002.

\bibitem{Bou92}
J.-P. Bouchaud.
\newblock Weak ergodicity breaking and aging in disordered systems.
\newblock {\em J. Phys. I (France)}, 2:1705--1713, september 1992.

\bibitem{BCKM98}
J.-P. Bouchaud, L.~Cugliandolo, J.~Kurchan, and M.~M\'ezard.
\newblock Out of equilibrium dynamics in spin-glasses and other glassy systems.
\newblock In A.~P. Young, editor, {\em Spin glasses and random fields}. World
  Scientific, Singapore, 1998.

\bibitem{BD95}
J.-P. Bouchaud and D.~S. Dean.
\newblock Aging on {P}arisi's tree.
\newblock {\em J. Phys I(France)}, 5:265, 1995.

\bibitem{BEGK1}
A.~Bovier, M.~Eckhoff, V.~Gayrard, and M.~Klein.
\newblock Metastability in stochastic dynamics of disordered mean-field models.
\newblock {\em Probab. Theory Related Fields}, 119(1):99--161, 2001.

\bibitem{BG13}
A.~Bovier and V.~Gayrard.
\newblock Convergence of clock processes in random environments and ageing in
  the {$p$}-spin {SK} model.
\newblock {\em Ann. Probab.}, 41(2):817--847, 2013.

\bibitem{BGS13}
A.~Bovier, V.~Gayrard, and A.~{\v{S}}vejda.
\newblock Convergence to extremal processes in random environments and extremal
  ageing in {SK} models.
\newblock {\em Probab. Theory Related Fields}, 157(1-2):251--283, 2013.

\bibitem{BK}
A.~Bovier and I.~Kurkova.
\newblock Derrida's generalised random energy models. {I}. {M}odels with
  finitely many hierarchies.
\newblock {\em Ann. Inst. H. Poincar\'e Probab. Statist.}, 40(4):439--480,
  2004.

\bibitem{CCP}
D.~Capocaccia, M.~Cassandro, and P.~Picco.
\newblock On the existence of thermodynamics for the generalized random energy
  model.
\newblock {\em J. Statist. Phys.}, 46(3-4):493--505, 1987.

\bibitem{FP}
L.~R. Fontes and G.~R.~C. Peixoto.
\newblock Elementary results on {K} processes with weights.
\newblock {\em Markov Process. Related Fields}, 19(2):343--370, 2013.

\bibitem{FGG}
L.~R.~G. Fontes, R.~J. Gava, and V.~Gayrard.
\newblock The {$K$}-process on a tree as a scaling limit of the {GREM}-like
  trap model.
\newblock {\em Ann. Appl. Probab.}, 24(2):857--897, 2014.

\bibitem{FIKP}
L.~R.~G. Fontes, M.~Isopi, Y.~Kohayakawa, and P.~Picco.
\newblock The spectral gap of the {REM} under {M}etropolis dynamics.
\newblock {\em Ann. Appl. Probab.}, 8(3):917--943, 1998.

\bibitem{Gav}
R.~J. Gava.
\newblock Scaling limit of the trap model on a tree.
\newblock {\em Ph.D.~thesis, University of S\~ao Paulo}, 2011 [in portuguese].

\bibitem{Ga}
V.~Gayrard.
\newblock Thermodynamic limit of the {$q$}-state {P}otts-{H}opfield model with
  infinitely many patterns.
\newblock {\em J. Statist. Phys.}, 68(5-6):977--1011, 1992.

\bibitem{G10b}
V.~Gayrard.
\newblock Aging in reversible dynamics of disordered systems. {II}. {E}mergence
  of the arcsine law in the random hopping time dynamics of the {REM}.
\newblock preprint, 2010.
\newblock arXiv:1008.3849.

\bibitem{G15}
V.~Gayrard.
\newblock Convergence of clock processes and aging in {M}etropolis dynamics of
  a truncated {REM}.
\newblock {\em Annales Henri Poincar{\'e} (arXiv:1402.0388)}, 17(3):537--614,
  2015.

\bibitem{G16}
V.~Gayrard.
\newblock Aging in {M}etropolis dynamics of the {REM}: a proof.
\newblock preprint, 2016.
\newblock arXiv:1602.06081.

\bibitem{GG16}
V.~Gayrard and O.~G\"un.
\newblock Aging in the {GREM}-like trap model.
\newblock {\em Markov Process. Related Fields}, 22(1):165--202, 2016.

\bibitem{GH}
V.~Gayrard and L.~Hartung.
\newblock Aging at the critical temperature in the {REM}.
\newblock In preparation.

\bibitem{GK15}
V.~Gayrard and N.~Kistler, editors.
\newblock {\em Correlated random systems: five different methods}, volume 2143
  of {\em Lecture Notes in Mathematics}.
\newblock Springer, Cham; Soci\'et\'e Math\'ematique de France, Paris, 2015.
\newblock Lecture notes from the 1st CIRM Jean-Morlet Chair held in Marseille,
  Spring 2013, CIRM Jean-Morlet Series.

\bibitem{Gun09}
O.~G\"un.
\newblock {\em Universality of Transient dynamic and aging for Spin-Glasses}.
\newblock PhD thesis, New York University, 2009.

\bibitem{HJ}
R.~A. Horn and C.~R. Johnson.
\newblock {\em Matrix analysis}.
\newblock Cambridge University Press, Cambridge, 1985.

\bibitem{K}
J.~H.~B. Kemperman.
\newblock {\em The passage problem for a stationary {M}arkov chain}.
\newblock Statistical Research Monographs, Vol. I. The University of Chicago
  Press, Chicago, Ill., 1961.

\bibitem{Ku01}
J.~Kurchan.
\newblock Recent theories of glasses as out of equilibrium systems.
\newblock {\em Comptes Rendus de l'Acad\'emie des Sciences - Series {IV} -
  Physics-Astrophysics}, 2(2):239 -- 247, 2001.

\bibitem{LLR}
M.~R. Leadbetter, G.~Lindgren, and H.~Rootz\'en.
\newblock {\em Extremes and related properties of random sequences and
  processes}.
\newblock Springer Series in Statistics. Springer-Verlag, New York-Berlin,
  1983.

\bibitem{N}
J.~R. Norris.
\newblock {\em Markov chains}, volume~2 of {\em Cambridge Series in Statistical
  and Probabilistic Mathematics}.
\newblock Cambridge University Press, Cambridge, 1998.
\newblock Reprint of 1997 original.

\bibitem{SN00}
M.~Sasaki and K.~Nemoto.
\newblock Analysis on aging in the generalized random energy model.
\newblock {\em J. Phys. Soc. Jpn.}, 69:3045--3050, 2000.

\bibitem{S}
E.~Seneta.
\newblock {\em Nonnegative matrices and {M}arkov chains}.
\newblock Springer Series in Statistics. Springer-Verlag, New York, second
  edition, 1981.

\bibitem{SS}
G.~W. Stewart and J.~G. Sun.
\newblock {\em Matrix perturbation theory}.
\newblock Computer Science and Scientific Computing. Academic Press, Inc.,
  Boston, MA, 1990.

\end{thebibliography}

\def\cprime{$'$}

\end{document}